\documentclass[11pt, reqno, a4paper]{amsart}
\usepackage[margin=1.2in]{geometry}
\numberwithin{equation}{section}
\usepackage{amssymb,amsfonts,amsthm}
\usepackage[utf8]{inputenc}
\usepackage{listings}
\usepackage{bm}
\usepackage[hyperpageref]{backref}
\usepackage{esint}
\usepackage{color}
\usepackage[bookmarks]{hyperref}
\usepackage{siunitx}
\usepackage{bigints}
\usepackage{hyperref}
\usepackage{mathtools}
\usepackage{tikz}

\allowdisplaybreaks[4]

\newtheoremstyle{myremark}{10pt}{10pt}{}{}{\bfseries}{.}{.5em}{}

 \newtheorem{theorem}{Theorem}[section]
 
 \newtheorem{lemma}{Lemma}[section]
 
 \theoremstyle{definition}
 \newtheorem{definition}{Definition}[section]
 
 \newtheorem{rem}{Remark}[section]
\usepackage{hyperref}

\allowdisplaybreaks[4]

\begin{document}

\title[Quantitative Stability of the Fractional Hardy Inequalities]{Quantitative Stability for Fractional Hardy Inequalities: Rearrangement-Free Techniques and Emden-Fowler Analysis}

\author{Avas Banerjee, Debdip Ganguly, and Vivek Sahu}

\address{ Theoretical Statistics and Mathematics Unit,
Indian Statistical Institute, Delhi Centre, S.J. Sansanwal Marg, New Delhi, Delhi 110016, India}

\email{A. Banerjee: avas24r@isid.ac.in}
\email{D. Ganguly: debdip@isid.ac.in}
\email{V. Sahu: vivek@isid.ac.in, viiveksahu@gmail.com}



\date{}

\dedicatory{}

\begin{abstract}
A classical result due to Frank and Seiringer \cite{Frank2008} asserts that for $1 \leq p < \frac{N}{s}$, there exists a sharp constant $\mathcal{C}_{N,s,p} > 0$ such that
$$
\delta_{s,p}(u)
:= \int_{\mathbb{R}^{N}}\int_{\mathbb{R}^{N}} \frac{|u(x)-u(y)|^{p}}{|x-y|^{N+sp}} \, dx \, dy 
- \mathcal{C}_{N,s,p} \int_{\mathbb{R}^{N}} \frac{|u(x)|^{p}}{|x|^{sp}} \, dx 
\;\geq\; 0,
$$
for all $u \in W^{s,p}(\mathbb{R}^N)$. The optimal constant is explicitly known. In this article, we investigate quantitative refinements of the above inequality. Our first result shows that, under the normalization
$
\int_{\mathbb{R}^{N}} \frac{|u(x)|^{p}}{|x|^{sp}} \, dx = 1,
$
the inequality 
\[
\delta_{s,p}(u) \;\gtrsim \;\bigl(\mathrm{dist}_{s,p}(u,\mathcal{Z})\bigr)^{\alpha},
\]
holds, where $\alpha = \max\{4, 2p\}$, $\mathcal{Z}$ denotes the family of ``virtual'' extremals, and the distance is measured in Marcinkiewicz (weak-$L^{p^{*}_{s}}$) space. Surprisingly, the stability exponent remains constant for \(p \leq 2\), while it depends on \(p\) for \(p > 2\). Our approach is based on a localized Poincar\'e-Sobolev inequality, combined with a suitable rescaling and Lorentz embeddings.  In particular, we exploit a careful decomposition of the nonlocal energy together with Lorentz space estimates, which enables us to control the deficit $\delta_{s,p}(u)$ in terms of the distance to the set $\mathcal{Z}$. The method also applies to the local case $s=1$, the argument is entirely rearrangement-free and strikingly, the exponent in the stability estimate is substantially improved compared to the existing literature. 
For $p=2$, via an Emden--Fowler correspondence and pseudo-differential operators, we show that the nonlocal Hardy deficit coincides with the local one, map extremizers accordingly, and obtain quantitative stability on $\mathbb{R}\times \mathbb{S}^{N-1}$ by exploiting the diagonalization of the fractional Hardy quadratic form on the cylinder due to Frank, Lieb, and Seiringer \cite{FrankLieb2008}. As an application, we establish a Hardy–Heisenberg-type uncertainty principle in the nonlocal setting, which appears to be new in the literature.

\end{abstract}

\maketitle


\section{Introduction}\label{Section 1: Introduction}

In recent years, fractional Hardy inequalities have attracted considerable attention due to their fundamental role in the analysis of nonlocal operators, fractional Sobolev spaces, and a wide range of problems in partial differential equations and mathematical physics. These inequalities provide a sharp quantitative link between the nonlocal $W^{s,p}$-seminorm and singular weights, and can be viewed as a natural extension of the classical Hardy inequality to the fractional setting. A substantial body of work has been devoted to their refinement, including the identification of optimal constants, extensions to more general kernels and domains, and the study of remainder terms and stability properties. Among the most influential contributions is the work of Frank and Seiringer \cite{Frank2008}, who established a sharp fractional Hardy inequality by means of a nonlinear ground state representation. Before stating their result let us introduce the Homogeneous Sobolev spaces: 

\[
[u]_{\dot W^{s,p}(\mathbb{R}^{N})}^{p}
:=
\int_{\mathbb{R}^{N}} \int_{\mathbb{R}^{N}}
\frac{|u(x)-u(y)|^{p}}{|x-y|^{N+sp}}\,dx\,dy .
\]

For \( N\geq 1\) and  \(0<s<1\), the homogeneous fractional Sobolev spaces are defined by
\[
\dot W^{s,p}(\mathbb{R}^{N})
:= \overline{C_c^\infty(\mathbb{R}^{N})}^{\, [\,\cdot\,]_{\dot W^{s,p}}},
\qquad 1\le p<\frac Ns ,
\]
and
\[
\dot W^{s,p}(\mathbb{R}^{N}\setminus\{0\})
:= \overline{C_c^\infty(\mathbb{R}^{N}\setminus\{0\})}^{\, [\,\cdot\,]_{\dot W^{s,p}}},
\qquad p>\frac Ns .
\]
More precisely, they proved that for $1 \leq p < \frac{N}{s}$ the inequality holds for all $u  \in \dot W^{s,p}(\mathbb{R}^N)$, while for $p > \frac{N}{s}$ it remains valid for $u \in \dot W^{s,p}(\mathbb{R}^N \setminus \{0\})$, and in both cases one has
\begin{equation}\label{Fractional Hardy}
\int_{\mathbb{R}^{N}}\int_{\mathbb{R}^{N}} \frac{|u(x)-u(y)|^{p}}{|x-y|^{N+sp}} \, dx \, dy \geq \mathcal{C}_{N,s,p} \int_{\mathbb{R}^{N}} \frac{|u(x)|^{p}}{|x|^{sp}} \, dx,
\end{equation}
where the constant $\mathcal{C}_{N,s,p}>0$ is optimal and admits the explicit representation
\begin{equation*}
\mathcal{C}_{N,s,p} := 2 \int_{0}^{1} r^{sp-1} \bigl|1-r^{(N-sp)/p}\bigr|^{p} \Phi_{N,s,p}(r) \, dr,
\end{equation*}
with
\begin{align*}
\Phi_{N,s,p}(r) &:= |\mathbb{S}^{N-2}| \int_{-1}^{1} \frac{(1-t^{2})^{(N-3)/2}}{(1-2rt+r^{2})^{(N+sp)/2}} \, dt, \quad N \geq 2, \\
\Phi_{1,s,p}(r) &:= \frac{1}{(1-r)^{1+sp}} + \frac{1}{(1+r)^{1+sp}}, \quad N=1.
\end{align*}
\medskip 

In \cite{Frank2008}, Frank and Seiringer established fractional Hardy inequalities with remainder terms for $p \geq 2$ via the ground state representation; this was later extended to the range $1<p<2$ by Dyda and Kijaczko \cite{Dyda2024a}. In \cite{FrankLieb2008}, Frank, Lieb, and Seiringer proved Hardy--Lieb--Thirring inequalities for fractional Schrödinger operators. For the half-space $\mathbb{R}^N_+$, the optimal constant in the case $p=2$ and $sp \neq 1$ was obtained by Bogdan and Dyda \cite{Dyda2011}, and subsequently extended to all $p \geq 1$ by Frank and Seiringer \cite{Frank2010}, again using the ground state representation. 

This method has proved fundamental in extending fractional Hardy inequalities to weighted fractional spaces, yielding optimal results for point singularities and the half-space (Dyda--Kijaczko \cite{Dyda2024}), and more generally for singularities supported on $k$-dimensional planes, $1 \leq k < N$ (see \cite{kijaczko2025}). In a different direction, Loss and Sloane \cite{Loss2010} established fractional Hardy inequalities on general domains with optimal constants expressed via suitable distance functions, recovering the boundary distance in convex domains; related results were obtained by Brasco and Cinti \cite{Brasco2018}. 

For boundary singularities, Chen and Song \cite{Chen2003} first treated the case $p=2$, later extended to $p>0$ by Dyda \cite{Dyda2004} under suitable assumptions on $s$ and $p$. More recently, Adimurthi and collaborators \cite{AdiJFA2026, AdimurthiCVPDE2026} resolved the remaining critical cases, introducing optimal logarithmic corrections and extending the theory to singularities supported on smooth submanifolds. For further recent developments, we refer to \cite{AdimurthiCCM2026, Adi2025, Bal2022, Bianchi2024a, Bianchi2024, Bianchi2026, Bogdan2022, Cinti2024, Dyda2023, Gyula2026, Lizaveta2026, Vivek2025}, without any claim of completeness.

\medskip

In contrast, the lack of compactness and the non-attainment phenomena associated with \eqref{Fractional Hardy} remain relatively less explored, particularly in comparison with their local counterparts. Likewise, the question of quantitative stability—namely, whether the deficit in \eqref{Fractional Hardy} can effectively control the distance to the family of extremal profiles—has not yet been fully understood and continues to present significant analytical challenges.
For the fractional Hardy inequality, the extremal profiles were identified in \cite{Frank2008} and are explicitly given by
\begin{equation}\label{fractional-extremals}
\omega_{a}(x) = a \;|x|^{-\frac{N-sp}{p}}, \qquad a \in \mathbb{R} \setminus \{0\}.
\end{equation}
 However, these functions do not belong to $W^{s, p}(\mathbb{R}^N)$ due to their singular behavior at the origin and lack of integrability at infinity and consequently equality in \eqref{Fractional Hardy} is never attained within the natural functional framework. These profiles capture the exact scaling and homogeneity of the inequality, and in this sense reflects the reason behind the non-attainment. This non-attainment phenomenon motivates the search for refined versions of the inequality. In particular, one seeks improvements that incorporate a remainder term measuring, in a suitable sense, the distance between a given function and the manifold of extremal profiles \eqref{fractional-extremals}. Such refinements fall under the scope of \emph{quantitative stability}.

\medskip 

A central issue in this direction is the stability of \eqref{Fractional Hardy}. In the absence of true extremizers, it is natural to investigate whether near-extremizing sequences exhibit rigidity. More precisely, one is led to the following questions:

\begin{itemize}

\item[$(1)$] \emph{If a function nearly saturates the fractional Hardy inequality, must it be close (in an appropriate topology) to one of the profiles \eqref{fractional-extremals}?}

\medskip 

\item[$(2)$] \emph{Can the deficit in \eqref{Fractional Hardy} be quantitatively controlled by the distance to the set of virtual extremals?}
\end{itemize}
To this end, it is natural to look for an ambient space where  the virtual extremizers of type \eqref{fractional-extremals} lie in, it is essential to identify an ambient space that accommodates their singular behavior while remaining compatible with the scaling of the fractional Hardy inequality. In this direction, a canonical candidate is the weak Lebesgue (Marcinkiewicz) space $L^{p_s^*,\infty}(\mathbb{R}^N)$, where 
\[
p_s^*=\frac{Np}{N-sp}
\]
is the critical Sobolev exponent associated with the fractional setting. As indicated in Lemma~\ref{Marcinkiewicz}, this space contains all the extremal profiles and is expected to be the smallest rearrangement invariant space with this property.

Recall that the quasi-norm in $L^{p_s^*,\infty}(\mathbb{R}^N)$ is defined by
\[
\|u\|_{L^{p_s^*,\infty}(\mathbb{R}^N)}=\sup_{t>0} t^{\frac{1}{p_s^*}} u^*(t),
\]
where $u^*$ denotes the decreasing rearrangement of $u$, as defined in \eqref{1D}. This characterization captures the borderline integrability of functions with power-type singularities. In particular, for the extremal profiles $\omega_a(x)=a \, |x|^{-\frac{N-sp}{p}}$, one verifies that
\[
|\{x \in \mathbb{R}^N : |\omega_a(x)|>\lambda\}| \sim \lambda^{-p_s^*},
\]
which shows that $\omega_a \in L^{p_s^*,\infty}(\mathbb{R}^N)$ while $\omega_a \notin L^{p_s^*}(\mathbb{R}^N)$. We refer Section \ref{Section 2 : preliminaries} for more details.  Within this framework, $L^{p_s^*,\infty}(\mathbb{R}^N)$ emerges as the smallest rearrangement invariant space that precisely captures the decay and scaling of the virtual extremals, thereby providing a natural minimal setting for the \it quantitative stability analysis of the fractional Hardy inequality\rm.
Moreover, in a recent work \cite[Theorem~1.3]{NGT}, Nitti, Glaudo, and K\"onig establish quantitative stability estimates for the fractional Caffarelli–Kohn–Nirenberg(CKN) inequality, showing that the deficit controls the distance from the manifold of optimizers which lies in the space $D^{s}_{\alpha}(\mathbb{R}^{n})$ denotes the closure of $C^{\infty}_{c}(\mathbb{R}^{n})$ with respect to the norm induced by the fractional CKN inequality.

\medskip


\subsection{Local case: comparison and known results}

 The classical Hardy inequality: for $N \geq 2$ and $1 < p < N$,
\begin{equation}\label{classical-hardy}
 \int_{\mathbb{R}^{N}} |\nabla u(x)|^{p} \,dx \geq \left(\frac{N-p}{p}\right)^{p} \int_{\mathbb{R}^{N}} \frac{|u(x)|^{p}}{|x|^{p}}\,dx ,
\end{equation}
for every $u \in D^{1,p}(\mathbb{R}^N)$, the completion of $C_c^\infty(\mathbb{R}^N)$ with respect to $\|u\| = \big(\int_{\mathbb{R}^N} |\nabla u|^p\,dx\big)^{1/p}$. The constant in \eqref{classical-hardy} is optimal but not attained in $W^{1,p}(\mathbb{R}^N)$, reflecting the presence of the virtual extremals
\begin{equation}\label{Local virtual-extremals}
v_a(x) = a\,|x|^{-\frac{N-p}{p}}, \qquad a \neq 0,
\end{equation}
which lie outside the energy space but belong to the weak Lebesgue space $L^{p^*,\infty}(\mathbb{R}^N)$, where $p^* = \frac{Np}{N-p}$.

Quantitative stability for \eqref{classical-hardy} was established by Cianchi--Ferone \cite{Cianchi2008}. Introducing the scale-invariant distance
\begin{equation*}
d_p(u) = \inf_{a \in \mathbb{R}} 
\frac{\|u - v_a\|_{L^{p^*,\infty}(\mathbb{R}^N)}}{\|u\|_{L^{p^*,p}(\mathbb{R}^N)}},
\end{equation*}
they proved that there exists $C = C(N,p) > 0$ such that
\begin{equation}\label{stable-hardy-local}
\int_{\mathbb{R}^{N}} |\nabla u(x)|^{p} \,dx 
- \left(\frac{N-p}{p}\right)^{p} \int_{\mathbb{R}^{N}} \frac{|u(x)|^p}{|x|^p} \, dx 
\geq C\, (d_p(u))^{2p^*} \int_{\mathbb{R}^{N}} \frac{|u(x)|^p}{|x|^p} \, dx.
\end{equation}

For further developments on local Hardy inequalities, including geometric improvements, sharp constants, and stability in various settings, we refer to 
\cite{chaudhuri2002improved, akutagawa2013geometric, berchio2017sharp, berchio2020optimal, carron1997inegalites, chen2023sharp, chen2024stability, Chen2026MathAnn, do2024scale, Do2026, flynn2023hardy, Nguyen2026} and the references therein, without claim of any completeness.
For the classical (Local) Hardy inequality, stability was previously obtained by Cianchi and Ferone \cite{Cianchi2008} via rearrangement and symmetrization techniques. Their approach reduces the problem to a one-dimensional setting for radially decreasing functions, where stability is established and subsequently lifted to the full inequality \eqref{classical-hardy}. In contrast, our method avoids rearrangement arguments altogether. Indeed, several techniques that are effective in the local setting do not appear to extend in a natural way to the nonlocal framework, primarily due to the lack of compatibility with the intrinsic nonlocal structure of the energy. To overcome these difficulties, we develop a different approach based on a scale-invariant formulation of the fractional Poincar\'e-Sobolev inequality, which allows us to capture the correct notion of distance and to derive quantitative stability estimates without relying on symmetrization.   This framework allows us to control the associated distance functional and to derive quantitative stability estimates in the fractional setting. Our approach yields an improved quantitative estimate in the local case, in which the exponent of the distance term $(d_p(u))^{2p^*}$ is replaced by $(d_p(u))^{\max\{4,\,2p\}}$. In particular, for $p \geq 2$ this provides a genuine strengthening of the corresponding local stability results. For $1 < p < 2$, such a strengthening holds under the additional condition $p > \frac{2N}{N+2}$. The optimality of this exponent, however, remains an open problem; see Theorem~\ref{Theorem: Improvement of stability exponent C and F result} for the precise formulation of the result.

\medskip
\subsection{Highlights of our main stability results in the non-local case.}
The next two theorems provide quantitative stability estimates for the fractional Hardy inequality. Owing to the different structure of the remainder terms in the regimes $p \geq 2$ and $1 < p < 2$, the analysis naturally splits into two cases. In each setting, we introduce a suitable notion of distance, formulated in terms of appropriate Marcinkiewicz-type norms, and establish stability with respect to these distances. In accordance with \eqref{Fractional Hardy}, we define the fractional Hardy deficit for $0 < s < 1$ by
\begin{equation}\label{fractional hardy deficit}
    \delta_{s,p}(u)
:= \int_{\mathbb{R}^{N}}  \int_{\mathbb{R}^{N}} \frac{|u(x)-u(y)|^{p}}{|x-y|^{N+sp}} \, dx \, dy 
- \mathcal{C}_{N,s,p} \int_{\mathbb{R}^{N}} \frac{|u(x)|^{p}}{|x|^{sp}} \, dx,
\end{equation}
which is nonnegative in view of \eqref{Fractional Hardy}.

As shown in \cite[Lemma~3.1]{FrankLieb2008}, for $p=2$, \eqref{fractional hardy deficit} is equivalent (up to a constant) to the following definition,
\begin{equation}\label{Fourier definition}
    \delta_{s,2}(u):= \int_{\mathbb{R}^{N}} |\xi|^{2s}|\widehat{u}(\xi)|^2\, d\xi  -C_{N,s}\int_{\mathbb{R}^{N}}|x|^{-2s}|u(x)|^2\,dx,
\end{equation}
where $C_{N,s}$ is the sharp constant given by
\begin{equation}\label{Defn: C_N,s}
    C_{N,s} = 2^{2s} \frac{\Gamma^2 \left( \frac{N+2s}{4} \right)}{\Gamma^2 \left( \frac{N-2s}{4} \right)},
\end{equation}
and $$\widehat{u}(\xi):=(2\pi)^{-\frac{N}{2}}\int_{\mathbb{R}^N}u(x)e^{-i\xi\cdot x}\,\mathrm{d}x, \quad \xi \in \mathbb{R}^N$$ denotes the Fourier transform of $u$.

\medskip

We introduce a notion of distance from the family of fractional extremals for the case $p \geq 2$ as follows:
\begin{equation}\label{fractional-distance}
    d_{s,p}(u) = \inf_{a \in \mathbb{R}} \frac{\| u-\omega_{a} \|_{L^{p^{*}_{s}, \infty}(\mathbb{R}^{N})}}{ \| u \|_{L^{p^{*}_{s}, p}(\mathbb{R}^{N})}},
\end{equation}
where $ \| u \|_{L^{p^{*}_{s}, p}(\mathbb{R}^{N})}$ denotes the Lorentz norm (See Definition~\ref{defi-lorentz}, Section~\ref{Section 2 : preliminaries}). The main results of this article is a quantitative improvement of the fractional Hardy inequality, which involves the distance \eqref{fractional-distance}. 
In this article, we shall restrict ourselves to the case
\(1 \le p < \frac{N}{s}.\)
The statement of the main result for $p \geq 2$ is given below:

\begin{theorem}\label{Theorem 1}
Let $N \geq 1$, $p \geq 2$, and $s \in (0,1)$ satisfy $sp < N$. Then there exists a constant $C = C(N,s,p) > 0$ such that the stability estimate
\begin{equation}\label{Stability ineq: p>2}
\delta_{s,p}(u) \geq C\;\big(d_{s,p}(u)\big)^{2p} \int_{\mathbb{R}^N} \frac{|u(x)|^p}{|x|^{sp}}\; dx
\end{equation}
holds for all $u \in \dot W^{s,p}(\mathbb{R}^{N})$.
\end{theorem}

\smallskip

\begin{rem}
The quantitative improvement above is proved under the assumption $p \geq 2$, since it relies on a remainder term that is only available in this range. For $1 < p < 2$, the remainder term has a different structure, and the argument does not extend. By homogeneity, one may equivalently impose the normalization
$
\int_{\mathbb{R}^N} \frac{|u(x)|^p}{|x|^{sp}} \, dx = 1.$
Under this constraint, Theorem~\ref{Theorem 1} reduces to
\[
\delta_{s,p}(u) \geq C \, \big(d_{s,p}(u)\big)^{2p}.
\]
\end{rem} 

\medskip

For the range $1 < p < 2$, the quantitative improvement requires a different choice of remainder term, leading naturally to a modified notion of distance. This distance is intrinsically defined to the the case $1<p<2$.

We first observe that
\[
\| \omega \|_{L^{p^{*}_{s}, \infty}(\mathbb{R}^{N})}
= \| \omega^{\frac{p}{2}} \|^{\frac{2}{p}}_{L^{q, \infty}(\mathbb{R}^{N})},
\qquad \text{where} \quad q = \frac{2N}{N-sp}.
\]
Motivated by this relation, throughout the paper we set
\[
q := \frac{2N}{N-sp}.
\]

Accordingly, for $1 < p < 2$, we introduce the distance functional $\mathcal{D}_{s,p}$ defined by
\begin{equation}
\mathcal{D}_{s,p}(u)
= \inf_{a \in \mathbb{R}}
\frac{\big\| u^{\langle \frac{p}{2} \rangle}
- \omega_{a}^{\langle \frac{p}{2} \rangle} \big\|_{L^{q, \infty}(\mathbb{R}^{N})}}
{\| u \|^{\frac{p}{2}}_{L^{p^{*}_{s}, p}(\mathbb{R}^{N})}},
\end{equation}
where we use the notation $a^{\langle b \rangle} := |a|^{b}\,\mathrm{sgn}(a)$.

With this definition at hand, we are now in a position to state the corresponding quantitative stability result in the regime $1 < p < 2$.

\begin{theorem}[Quantitative stability for $1<p<2$]\label{Theorem 2}
Let $N \geq 1$, $1<p<2$, and $s \in (0,1)$ be such that $sp<N$. Then there exists a constant $C = C(N,s,p)>0$ with the following property: for every $u \in \dot W^{s,p}(\mathbb{R}^{N})$, the fractional Hardy deficit satisfies the estimate
\begin{equation}\label{Stability ineq: p<2}
\delta_{s,p}(u)
\;\geq\;
C \,\big(\mathcal{D}_{s,p}(u)\big)^{4}
\int_{\mathbb{R}^{N}} \frac{|u(x)|^{p}}{|x|^{sp}} \, dx.
\end{equation}
\end{theorem}
\medskip

\begin{rem}
Theorems \ref{Theorem 1} and \ref{Theorem 2} provide quantitative forms of the fractional Hardy inequality. In particular, \eqref{Stability ineq: p>2} and \eqref{Stability ineq: p<2} show that the deficit not only remains nonnegative, but also controls the distance of $u$ from the family of virtual extremals $\{\omega_a\}_{a \in \mathbb{R}\setminus\{0\}}$. Consequently, any near-extremizing sequence for \eqref{Fractional Hardy} converges, in a Lorentz-type sense, to this family.
\end{rem}


We now turn to the local case and show that the above approach can be adapted to recover the stability of classical (local) Hardy inequality. In fact, our approach provides an alternative proof of the stability results due to Cianchi--Ferrone \cite{Cianchi2008}, with a refinement in the quantitative estimate. More precisely, for $p \geq 2$, we obtain an improved stability exponent in the associated distance functional, yielding a sharper control of the deficit in terms of the distance to the family of extremals. Our approach is entirely rearrangement-free and readily extends to more general settings where scaling invariance plays a fundamental role. This further highlights the flexibility of the method and its ability to unify the fractional and local frameworks within a common quantitative stability. For the local case $s=1$, we define the (classical) Hardy deficit for $u \in C_c^{\infty}(\mathbb{R}^N)$ as
\begin{equation}
\delta_{p}(u)
:=
\int_{\mathbb{R}^{N}} |\nabla u(x)|^{p} \, dx
- \mathcal{C}_{N,p} \int_{\mathbb{R}^{N}} \frac{|u(x)|^{p}}{|x|^{p}} \, dx,
\end{equation}
where $\mathcal{C}_{N,p} = \left(\frac{N-p}{p}\right)^p$ is the optimal Hardy constant. In particular, next theorem recovers the stability estimate of Cianchi--Ferrone \cite[Theorem~1.1]{Cianchi2008} with an improved stability exponent. 

\medskip 

\begin{theorem}[Improved stability exponent]\label{Theorem: Improvement of stability exponent C and F result}
Let $N \geq 2$ and $1 < p < N$. Then there exists a constant $C = C(N,p)>0$ such that, for every real-valued weakly differentiable function $u$ in $\mathbb{R}^N$ decaying to zero at infinity and such that $|\nabla u| \in L^{p}(\mathbb{R}^N)$, the Hardy deficit satisfies
\begin{equation}
    \delta_{p}(u)
\;\geq\;
C\, \big(d_{p}(u)\big)^{\alpha}
\int_{\mathbb{R}^{N}} \frac{|u(x)|^{p}}{|x|^{p}} \, dx,
\end{equation}
where $\alpha = \emph{max}\{ 4, 2 p\}.$ 
\end{theorem}


\begin{rem}
Obtaining a stability exponent of this strength is challenging. For $p \geq 2$, we improve the exponent of the distance function to $2p$, whereas in the result of Cianchi--Ferone \cite{Cianchi2008} it is $2p^{*}$, yielding a strictly smaller and hence sharper exponent. For $1 < p < 2$, we obtain the exponent $4$, which is notably independent of both $N$ and $p$, in contrast to the Cianchi--Ferone exponent $2p^{*}$, which depends on these parameters. Moreover, since $4 < 2p^{*}$ for $p > \frac{2N}{N+2}$, this also gives a strict improvement for $1<p<2$ in that range; it seems that improvement via rearrangement techniques is difficult, and achieving sharper exponents requires new approaches and technical innovations. It is interesting to note that such phenomenon on the stability exponents were first observed by Figalli-Zhang in \cite{Figalli-Zhang} for the $p-$ Sobolev inequality. Moreover, they proved that corresponding exponent is sharp. 

\end{rem}

\medskip

\subsection{Strategy of the proof and main Novelty} 

We briefly outline the strategy of our approach, together with the main technical novelties:

\begin{itemize}

\item As mentioned earlier, for the classical (local) Hardy inequality, quantitative stability was established by Cianchi and Ferone \cite{Cianchi2008} via rearrangement and symmetrization arguments.

\item Such techniques do not readily extend to the nonlocal setting, primarily due to the presence of the Gagliardo seminorm, which involves a genuinely nonlocal double integral and does not admit a straightforward one-dimensional reduction.

\item To overcome this difficulty, we adopt a different strategy based on a geometric decomposition of the domain, partitioning $\mathbb{R}^N$ into a family of concentric annuli.

\item On each annulus, we employ suitable Poincar\'e-type inequalities, carefully tracking their scaling behavior with respect to the size of the annulus.

\item A key ingredient is the precise control of the Gagliardo seminorm across different annuli, which requires a refined localization argument.

\item Finally, through a delicate summation procedure, we combine the local estimates while ensuring that the constants remain uniform and do not depends on the domain and the scaling parameter. Finally applying the Lorentz-Embedding we obtain the desired stability inequality. 

\end{itemize}

This approach provides a robust framework for handling the nonlocal nature of the problem and represents a significant departure from classical symmetrization-based methods.

\medskip 

Our next objective is to further investigate the structure of the fractional Hardy deficit. A natural question is whether it admits a direct comparison with its local counterpart, at least at the level of quantitative estimates. Remarkably, for $p=2$, we show that the two deficits are in fact equivalent under a transformation $T$ constructed based on Emden--Fowler transformation. As a consequence, the fractional Hardy deficit can be analyzed through its local analogue on the cylinder, allowing us to transfer stability properties. In particular, the quantitative estimates established earlier can be reinterpreted in this framework, yielding stability results on $\mathbb{R}\times \mathbb{S}^{N-1}$ via the diagonalization of the associated quadratic form (see \cite{FrankLieb2008}).

\subsection{Equivalence of Local and Nonlocal Hardy Inequalities via the Emden--Fowler Transform}

The classical local  and fractional Hardy inequalities play a fundamental role in analysis, reflecting the behavior of local and nonlocal operators, respectively. Although they arise in different settings, it is natural to seek a unified framework that connects them.
The core of our approach is a transformation $T$ that links the fractional ($0<s<1$) and local ($s=1$) Hardy regimes. It combines the Emden--Fowler transform with a pseudo-differential multiplier $M_s$ on the cylinder. The symbol of $M_s$, defined mode-by-mode via the Fourier--spherical harmonic decomposition, is chosen to exactly captures the gap between the fractional and local Hardy energies. Thus, $T$ acts as a geometric--spectral transport map, converting the nonlocal structure into a local one on a conformally equivalent cylinder, and should be viewed as an intertwining operator between the two Hardy geometries. Moreover, in \cite{WAM}, Emden-Fowler transformation is used to reformulate fractional Caffarelli-Kohn-Nirenberg (CKN) inequality in cylindrical variables.

\medskip 

We shall briefly explain the geometric interpretation behind our transformation 
$T.$ Our method lies in the conformal equivalence between the punctured Euclidean space $\mathbb{R}^N \setminus \{0\}$ and the infinite Emden-Fowler cylinder $\mathcal{C} = \mathbb{R} \times \mathbb{S}^{N-1}$.

The fractional Hardy inequality is invariant under the continuous group of Euclidean dilations: $D_\lambda u(x) = u(\lambda x)$ for $\lambda > 0$. However, on $\mathbb{R}^N$, dilations act multiplicatively, which complicates spectral analysis. By moving to logarithmic polar coordinates, we define $x = e^t w$, where $t = \ln|x| \in \mathbb{R}$ acts as the axial coordinate and $w = \frac{x}{|x|} \in \mathbb{S}^{N-1}$ is the angular variable. Under this diffeomorphism, the multiplicative Euclidean scaling $x \mapsto \lambda x$ is mapped to a simple, additive translation along the axis of the cylinder: $t \mapsto t + \ln \lambda$.

Because translation invariant operators on a cylinder can be globally diagonalized using the Fourier transform along the axis ($\mathbb{R}$) and a spherical harmonic decomposition over the cross-section ($\mathbb{S}^{N-1}$), the cylinder provides the perfect geometric theater to analyze scale-invariant operators.

To transition functions from $\mathbb{R}^N$ to $\mathcal{C}$ while preserving the natural measure of the Hardy inequalities, we define a family of weighted, isometric lifting maps. For a weight parameter $\alpha \in \mathbb{R}$, let $L^2(\mathbb{R}^N; |x|^{-2\alpha}dx)$ be the weighted Lebesgue space. We define the conformal map $\Phi_\alpha: L^2(\mathbb{R}^N; |x|^{-2\alpha}dx) \to L^2(\mathcal{C})$ by:$$(\Phi_\alpha u)(t, w) := e^{\frac{N-2\alpha}{2}t} u(e^t w).$$The exponential weight $e^{\frac{N-2\alpha}{2}t}$ is geometrically precise: it absorbs the volume measure $|x|^{-2\alpha} dx = e^{-2\alpha t} e^{Nt} e^{-t} dt \, d\omega = e^{(N-2\alpha)t} dt \, d\omega$, rendering $\Phi_\alpha$ an isometry. 
To this end we shall now define the transformation $T$ taking into account the spectral and geometric properties of the transformation $\Phi_{\alpha}.$
\medskip 

Let $\mathcal{S}(\mathbb{R}^N)$ denote the Schwartz class of rapidly decaying smooth functions on $\mathbb{R}^N$.  
We construct a new transformation
\[
T : \mathcal{S}(\mathbb{R}^N)\to L^2_{\mathrm{loc}}(\mathbb{R}^N\setminus\{0\}),
\]
which plays a central role in our analysis. The main feature of this transformation is that it provides a geometric correspondence between the fractional Hardy deficit and the classical Hardy deficit. In particular, under this map, the fractional Hardy structure is converted into its local counterpart in a precise way.

The transformation is defined by
\begin{equation}\label{Defn: The operator T}
T[u](x)
:=
|x|^{-\frac{N-2}{2}}
\left(
M_s\!\left[
|\cdot|^{\frac{N-2s}{2}}u(\cdot)
\right]
\right)
\left(
\ln|x|,
\frac{x}{|x|}
\right),
\end{equation}
where 
$M_s$ 
is the pseudo-differential operator defined in \eqref{defM}.

The novelty of the transformation lies in the fact that it combines the Emden--Fowler type change of variables with the nonlocal operator $M_s$, thereby creating an explicit bridge between nonlocal and local Hardy-type structures. This allows us to transfer information between the fractional and classical settings in an effective and geometrically transparent manner. In fact, the transformation $T$ admits a natural decomposition into composition of three mapping  (see Theorem~\ref{properties of T}), namely
\[
T = \Phi_1^{-1} \circ M_s \circ \Phi_s.
\]
This decomposition highlights the geometric structure underlying the transformation which we now briefly discuss.

\begin{itemize}
    \item[Step 1.]  \textbf{Fractional lift ($\Phi_s$):} The function $u$ is lifted to the cylinder with weight $\alpha=s$, yielding $\Phi_s u \in L^2(\mathcal{C})$. Geometrically, this isolates the dilation symmetries of the fractional problem into pure axial translations on $\mathcal{C}$.
    \item[Step 2.]  {\bf Spectral Modulation ($M_s$):} On the cylinder, the fractional deficit becomes a pseudo-differential operator, whose action is strictly local in the momentum space. We apply a multiplier $M_s$ which operates on the dual space of the cylinder. Let $\mu_{\ell} = \ell (\ell+N-2)$ be the eigenvalues of the Laplacian-Beltrami operator $- \Delta_{\mathbb{S}^{N-1}}$. Taking the Fourier transform in $t$ (dual variable $\xi$) and projecting onto spherical harmonics (index $\ell$), $M_s$ acts as:
    \begin{equation}\label{defM}
      ( \widehat{M_{s} \phi}) (\xi, \ell, m) = \sqrt{ \frac{P_{s}(\xi, \ell)}{\xi^{2}+ \mu_{\ell}}} \widehat{\phi} (\xi, \ell, m).  
    \end{equation}
    
The numerator $P_s(\xi, \ell)$ is the exact spectral symbol of the fractional Hardy operator on the cylinder, derived from the Mellin transform of the fractional Laplacian:
\begin{equation}\label{PS}
   P_{s} (\xi, \ell) := 2^{2s} \left| \frac{\Gamma \left( \frac{N+2s +2\ell+ 2i \xi}{4} \right)}{\Gamma \left( \frac{N-2s+2 \ell+2i \xi}{4} \right)} \right|^{2} - C_{N,s}, 
\end{equation}
where $C_{N,s}$ is given by \eqref{Defn: C_N,s}. The denominator $\xi^{2}+ \mu_{\ell}$ is the exact symbol of the local ($s=1$) classical Hardy operator on the cylinder. Therefore, geometrically, $M_s$ acts as a ``spectral lens". It scales the amplitude of every continuous frequency mode $\xi$ and angular mode $\ell$ of the fractional function precisely into the spectral signature of a classical function. For the details on the symbols on cylinders, we refer to \cite[Section~5]{gonzalez2016recent}.
    \item[Step 3.] \textbf{Classical projection ($\Phi_1^{-1}$):} After applying $M_s$, we map back (to the Euclidean space) via $\Phi_1^{-1}$ (not $\Phi_s^{-1}$), thereby placing the function into the classical weighted Sobolev space.
\end{itemize}

Below, we describe our correspondence theorem formally.


\begin{theorem}\label{Theorem: Equivalence bw local and nonlocal via Emden Trans}
Let $N \geq 3.$ The following properties hold:
\begin{itemize}
    \item[(i)] \textbf{Preservation of the deficit.} 
    For every $u \in \mathcal{S}(\mathbb{R}^N),$ the transformation $T$ preserves the Hardy deficit, namely
    \[
        \delta_{2}(T[u]) = \delta_{s,2}(u).
    \]

    \item[(ii)] \textbf{Action on extremizers.} 
    Let $\omega_{s}(x) = |x|^{-\frac{N-2s}{2}}$ denote the fractional Hardy extremizer and $\omega_{1}(x) = |x|^{-\frac{N-2}{2}}$ the classical Hardy extremizer. Then $T$ extends (in the natural weighted sense) so that
    \[
        T[\omega_{s}] = K_{N,s} \, \omega_{1},
    \]
    where the constant $K_{N,s}$ is given by
    \begin{equation}\label{Defn: K_N,s}
        K_{N,s} 
        = \frac{1}{2} 
        \sqrt{ \psi'\!\left( \frac{N-2s}{4} \right) - \psi'\!\left( \frac{N+2s}{4} \right)} 
        \, \sqrt{C_{N,s}},
    \end{equation}
    and $\psi'$ denotes the trigamma function defined in \eqref{Trigamma function}.
\end{itemize}
\end{theorem}

\medskip 

The transform $T$ allows us to transfer the stability of the local Hardy inequality to the cylinder via a pullback structure. To this end, we introduce the weighted norm
\[
\|u\|_{\mathcal{W}_{N,s}}
:=
\left(
\int_{\mathbb{R}^{N}} \frac{|T[u](x)|^{2}}{|x|^{2}}\,dx
\right)^{1/2},
\]
as well as the pullback Lorentz norms
\[
\|u\|_{\mathcal{L}^{q,r}_{N,s}}
:=
\|T[u]\|_{L^{q,r}(\mathbb{R}^{N})}.
\]

Recall that the local stability result measures the distance to the family of virtual extremals $v_a(x)=a|x|^{-\frac{N-2}{2}}$. Let $w_s(x)=|x|^{-\frac{N-2s}{2}}$ denote the nonlocal extremal. By construction,
\[
T[w_s]=K_{N,s}\,v_1,
\]
and, by linearity of $T$, the families of extremals are in one-to-one correspondence:
\[
v_a = T\!\left[\frac{a}{K_{N,s}}\,w_s\right].
\]

Consequently, the local distance functional transforms as
\[
d_2(T[u])
=
\inf_{\alpha \in \mathbb{R}}
\frac{\|T[u-\alpha w_s]\|_{L^{2^*,\infty}(\mathbb{R}^{N})}}
{\|T[u]\|_{L^{2^*,2}(\mathbb{R}^{N})}}.
\]
This naturally leads to the nonlocal distance, defined intrinsically in terms of $u$:
\[
\mathfrak{D}_{s,2}(u)
:=
\inf_{\alpha \in \mathbb{R}}
\frac{\|u-\alpha w_s\|_{\mathcal{L}^{2^*,\infty}_{N,s}}}
{\|u\|_{\mathcal{L}^{2^*,2}_{N,s}}}.
\]


\medskip

\begin{theorem}[Quantitative stability for $p=2$ via pullback formulation]\label{qspf}
Let $N \geq 3,$ $s \in (0,1)$ and $\delta_{s,2}(u)$ be defined as in \eqref{Fourier definition}. 
Then there exists a constant $C = C(N)>0$ such that
\begin{equation}
    \delta_{s,2}(u)
\;\geq\;
C\, \|u\|_{\mathcal{W}_{N,s}}^{2}
\big(\mathfrak{D}_{s,2}(u)\big)^{4},
\end{equation}
holds for all $u \in \dot W^{s, p}(\mathbb{R}^N).$
\end{theorem}


\medskip 

\subsection{An Application: Hardy-Heisenberg Uncertainty Principle}

Having shown that the transport map \(T\) preserves the fractional Hardy deficit by converting it into a local deficit on the Emden--Fowler cylinder \(\mathcal{C} = \mathbb{R} \times \mathbb{S}^{N-1}\), we now exploit this structure to derive a sharp uncertainty principle for the underlying non-local operators.

In classical quantum mechanics, the Heisenberg uncertainty principle bounds the product of spatial and momentum variances via the Dirichlet energy. 
For the fractional Hardy inequality, such a formulation is obstructed by the non-local nature of \(\delta_{s,2}(u)\), which lacks a pointwise chain rule and integration-by-parts structure in \(\mathbb{R}^N\). To overcome this, we exploit the cylindrical geometry \(\mathcal{C}\), where the classical deficit $\delta_{2}(v)$ represents kinetic energy. 
This allows us to define positional variance on \(\mathcal{C}\) and transfer it to \(\mathbb{R}^N\) via the map \(T\). To this end, let us define the following quantities:

For \(u \in \mathcal{S}(\mathbb{R}^N)\), and $N \geq 4,$ define
\[
M_T(u) := \int_{\mathbb{R}^N} \frac{|T[u](x)|^2}{|x|^2}\,dx, 
\quad 
V_T(u) := \int_{\mathbb{R}^N} (\ln|x|)^2 \frac{|T[u](x)|^2}{|x|^2}\,dx.
\]
Under \(x = e^t w\), \(M_T(u)\) is the \(L^2(\mathcal{C})\) mass, while $V_{T}(u)$ strictly measures the axial dispersion (the positional variance) of the lifted state along the infinite cylindrical axis.

Under the Emden--Fowler transform, the deficit splits as
\[
\delta_{2}(T[u])=\int_{\mathbb{S}^{N-1}}\int_{\mathbb{R}}\Big(|\partial_t \phi_v|^2+|\nabla_w \phi_v|^2\Big)\,dt\,dw.
\]
Since the angular term is non-negative, it follows that \(\delta_{2}(T[u])\) is bounded below by the axial energy, reducing the problem to a 1-dimensional system in \(t\).

Now we are in a position to state the main uncertainty result: Let $\mathcal{S}(\mathbb{R}^N)$ denote the Schwartz class of rapidly decreasing smooth functions on $\mathbb{R}^N$.

\begin{theorem}[Fractional Hardy--Heisenberg Uncertainty Principle]\label{HHUP}
Let $N\geq 4$ and let \(u \in \mathcal{S}(\mathbb{R}^N)\). Then the fractional Hardy deficit controls the logarithmic dispersion of the transformed function \(T[u]\) through the sharp uncertainty-type inequality
\begin{equation}
   \delta_{s,2}(u)\, V_T(u)
\;\geq\;
\frac{1}{4}\,\big(M_T(u)\big)^2. 
\end{equation}
Moreover, the constant \(\frac{1}{4}\) is optimal. Equality is attained precisely by Gaussian functions of the form
$\{ A e^{-\alpha t^2}, \
A\in\mathbb{R}, \ 
\alpha>0\}.$
\end{theorem}
\medskip 

\begin{rem}
Although the optimal constant \(\frac{1}{4}\) is not attained for smooth compactly supported functions, it is asymptotically achieved. More precisely, there exists a sequence of non-local cylindrical states \(\phi_v(t,w) := (M_s \phi_u)(t,w)\) that converges to Gaussian profiles of the form \( \phi(t,w) = A e^{-\alpha t^2}\), where \(A \in \mathbb{R}\setminus\{0\}\) and \(\alpha > 0\). We refer to Theorem~\ref{thmsharp} for a precise statement.
\end{rem}


\medskip 

The structure of the paper is as follows.

\begin{itemize}

\item[Section \ref{Section 1: Introduction}:] Introduction, overview of the fractional Hardy inequality, and the notion of stability of fractional Hardy inequality. Statement of main results concerning the stability of fractional Hardy inequality. Improved stability exponent of classical (local) Hardy inequality. Equivalence of Local and Nonlocal Hardy Inequalities via the Emden--Fowler Transform, and state main results. 

\item[Section \ref{Section 2 : preliminaries}:] Preliminaries: fractional Sobolev spaces, extremizers for the nonlocal Hardy inequality, weak Lebesgue spaces, and the analytical tools used throughout—namely the Emden–Fowler transformation, pseudo-differential operators, and the Mellin transform.

\item[Section \ref{Section : 3}:] Scale-invariant properties of the weighted Gagliardo seminorm and the associated scaling behavior of the fractional Sobolev inequality.

\item[Section \ref{Section : 4}:] Stability analysis of the nonlocal Hardy inequality, treating separately the cases $1<p<2$ and $p \geq 2$, via rearrangement-free arguments.

\item[Section \ref{Section : 5}:] Rearrangement-free stability analysis of the local Hardy inequality based on scaling methods, yielding an improved stability exponent.

\item[Section \ref{Section : 6}:] Correspondence between nonlocal and local Hardy inequalities for $p=2$, via a transformation $T$ constructed through Emden–Fowler type transform and pseudo-differential operator techniques; preservation of the Hardy deficit and action on extremizers. 

\item[Section \ref{Section : 7}:] Applications:  Hardy-Heisenberg-type uncertainty principle, and existence of extremals. 

\end{itemize}


\medskip

\section{Preliminary: functional analytic settings and auxiliary results}\label{Section 2 : preliminaries} 
In this section, we introduce the definitions, and some preliminary results that will be used throughout the paper. We use the following notation: for a measurable set $\Omega \subset \mathbb{R}^{N}$, the symbol $(u)_{\Omega}$ denotes the average of the function $u$ over $\Omega$, that is,
\begin{equation*}
    (u)_{\Omega} :=  \frac{1}{|\Omega|} \int_{\Omega} u(x) \,  dx = \fint_{\Omega} u(x) \, dx. 
\end{equation*}
Here, $|\Omega|$ denotes the Lebesgue measure of $\Omega$ and we assume $|\Omega| < \infty.$ The parameter $s$ is always assumed to belong to $(0,1)$. Also, $C>0$ denotes a generic constant, which may change from line to line.

\smallskip

\subsection{Function spaces} Let $\Omega$ be an open set in $\mathbb{R}^{N}$, and let $s \in (0,1)$. For any $p \in [1, \infty)$, the fractional Sobolev space is defined by
\begin{equation}\label{defn frac sob space}
W^{s,p}(\Omega) := \left\{    u \in L^{p}(\Omega) : \int_{\Omega} \int_{\Omega}  \frac{|u(x)-u(y)|^{p}}{|x-y|^{N+sp}} \, dx \, dy < \infty \right\}.
\end{equation}
This space is equipped with the norm
\begin{equation*}
    \|u\|_{W^{s,p}(\Omega)} := \left( \|u\|^{p}_{L^{p}(\Omega)} + [u]^{p}_{W^{s,p}(\Omega)} \right)^{\frac{1}{p}},
\end{equation*}
where
\begin{equation*}
    [u]_{W^{s,p}(\Omega)} := \left( \int_{\Omega} \int_{\Omega}  \frac{|u(x)-u(y)|^{p}}{|x-y|^{N+sp}} \, dx \, dy \right)^{\frac{1}{p}} 
\end{equation*}
is called the Gagliardo seminorm.

\smallskip

The fractional $p$-Laplacian of a function  $\varphi \in C^{\infty}_{c}(\mathbb{R}^{N})$ is defined by
\begin{align}\label{frac-lap}
(-\Delta_{p})^s \varphi(x)
 & :=
2 \,
\mathrm{P.V.}\!\int_{\mathbb{R}^{N}}
\frac{|\varphi(x)-\varphi(y)|^{p-2}(\varphi(x)-\varphi(y))}{|x-y|^{N+sp}}\, dy \nonumber \\ & =  2 \lim_{\varepsilon \to 0} \int_{||x|-|y||> \varepsilon} \frac{|\varphi(x)-\varphi(y)|^{p-2}(\varphi(x)-\varphi(y))}{|x-y|^{N+sp}} \, dy, \qquad   x \in \mathbb{R}^{N},
\end{align}
where $\mathrm{P.V.}$ denotes the Cauchy principal value.

\smallskip

\smallskip

\subsection{Space of virtual ground state solution} Recall that the fractional Hardy deficit is defined by
\begin{equation*}
    \delta_{s,p}(u) = \int_{\mathbb{R}^{N}}\int_{\mathbb{R}^{N}} \frac{|u(x)-u(y)|^{p}}{|x-y|^{N+sp}} \, dx \, dy - \mathcal{C}_{N,s,p} \int_{\mathbb{R}^{N}} \frac{|u(x)|^{p}}{|x|^{sp}} \, dx.
\end{equation*}

\smallskip

\begin{definition}
    We say $u$ is a virtual extremizer of \eqref{Fractional Hardy} if 
    \begin{align*}\label{dist-solution}
\int_{\mathbb{R}^{N}} \int_{\mathbb{R}^{N}} &
\frac{|u(x)-u(y)|^{p-2}(u(x)-u(y))(\varphi(x)-\varphi(y))}
{|x-y|^{N+sp}} \, dx\,dy \nonumber
\\ & =
\mathcal{C}_{N,s,p}
\int_{\Omega} \frac{|u(x)|^{p-2}u(x)}{|x|^{sp}}\, \varphi(x)\, dx, \quad \forall \, \varphi \in C^{\infty}_{c}(\Omega),
\end{align*}
where $\Omega=\mathbb{R}^N\setminus \{0\}$.
\end{definition}

The following lemma shows that $\omega(x)= |x|^{-\frac{N-sp}{p}}$ is a positive solution of the Euler–Lagrange equation associated with \eqref{Fractional Hardy}, i.e., $\delta_{s,p}(\omega)=0$, and hence $\omega$ is a formal extremizers of \eqref{Fractional Hardy}.

\begin{lemma}\label{Lemma: Ground state Hardy inequality}
    One has uniformly for $x$ from compacts in $\mathbb{R}^{N} \setminus \{ 0 \}$
\begin{equation}
        2 \lim_{\varepsilon \to 0} \int_{||x|-|y||> \varepsilon} \frac{|\omega(x)-\omega(y)|^{p-2}(\omega(x) - \omega(y)) }{|x-y|^{N+sp}} \, dy = \mathcal{C}_{N,s,p} \frac{\omega(x)^{p-1}}{|x|^{sp}}.
    \end{equation}
\end{lemma}
\begin{proof}
    See \cite[Lemma $3.1$]{Frank2008} for the proof.
\end{proof}

Let $\mathcal{Z}$ denotes the set of virtual extremizer defined as 

\[ 
\mathcal{Z} := \{ a \, \omega : a \in \mathbb{R} \}.
\]

\begin{rem}
Heuristically, let $u_{0} (\neq 0 )$ be an extremizer  such that $u_{0} \notin \mathcal{Z}.$  Then $\delta_{s,p}(u_{0})=0$. For $p \geq 2$, from \eqref{Fractional Hardy with a remainder for p>2} with $u_{0}$, we have 
\begin{equation*}
    \int_{\mathbb{R}^{N}} \int_{\mathbb{R}^{N}} \frac{|v(x)-v(y)|^{p}}{|x-y|^{N+sp}} \, \frac{dx}{|x|^{\frac{N-sp}{2}}} \, \frac{dy}{|y|^{\frac{N-sp}{2}}}=0,
\end{equation*}
where $v= \omega^{-1}u_{0}$. Since the integrand is nonnegative, this implies
\begin{equation*}
 \frac{|v(x)-v(y)|^{p}}{|x-y|^{N+sp}} = 0 \quad a.e. \quad \text{with respect to} \quad \frac{dx}{|x|^{\frac{N-sp}{2}}} \times \frac{dy}{|y|^{\frac{N-sp}{2}}}  ,
\end{equation*}
and hence $v(x)=v(y)$ a.e. for $(x,y) \in \mathbb{R}^{N} \times \mathbb{R}^{N}$. Therefore, $v$ is constant almost everywhere, i.e.,
\begin{equation*}
    v(x) = C \quad \text{a.e. } \quad \text{for some } C \neq 0. 
\end{equation*}
Recalling that $v= \omega^{-1} u_{0}$, we obtain $u_{0}= C \, \omega$ a.e. for some $C \neq 0$. Therefore,
\begin{equation*}
    \mathcal{Z}= \{ a \, \omega : a  \in \mathbb{R}\}.
\end{equation*}
Similarly, applying \eqref{Fractional Hardy with a remainder for p<2}, the above relation also hold true for the case $1<p<2$. This shows, at least heuristically, that $\mathcal{Z}$ is the space of all virtual extremizers.
\end{rem}

\subsection{Rearrangement Techniques}

In this subsection, we briefly define some symmetrization techniques, Lorentz norms, and related embeddings. Although these techniques are not directly used in the proofs of our main results, they are needed to show that the  virtual extremizers  belongs to the Marcinkiewicz space and to define the distance function. For this purpose, we recall some basic definitions and properties of Lorentz norms.

\smallskip

Let $f$ be any measurable function on $\mathbb{R}^{N}$. The decreasing rearrangement of $f$, denoted by $f^{*}$, is defined by  
\begin{equation}\label{1D}
    f^{*}(t)
    :=
    \inf \left\{ \lambda>0 : \mu_f(\lambda) \leq t \right\},
\end{equation}
where $\mu_{f}$ is the distribution function given by  
\begin{equation*}
    \mu_{f}(\lambda)
    :=
    \bigl| \{ x \in \mathbb{R}^N : |f(x)| > \lambda \} \bigr|.
\end{equation*}

\begin{definition}[Lorentz space]\label{defi-lorentz}
Let $f$ be a measurable function on $\mathbb{R}^{N}$, and let $0 < p,q \leq \infty$. We define  
\begin{equation*}
    \|f\|_{L^{p,q}(\mathbb{R}^{N})} =
    \begin{cases}
        \left( \displaystyle \int_{0}^{\infty} \left( t^{\frac{1}{p}} f^{*}(t) \right)^{q} \,\frac{dt}{t} \right)^{\frac{1}{q}}, & \text{if } q < \infty, \\[1em]
        \displaystyle \sup_{t>0} \, t^{\frac{1}{p}} f^{*}(t), & \text{if } q = \infty.
    \end{cases}
\end{equation*}
The set of all functions $f$ such that $\| f \|_{L^{p,q}(\mathbb{R}^{N})} < \infty$ is denoted by $L^{p,q}(\mathbb{R}^{N})$, and is called the Lorentz space with indices $p$ and $q$. In the case $q = \infty$, the space $L^{p,\infty}(\mathbb{R}^{N})$ is also known as the Marcinkiewicz space or the weak $L^{p}$ space.
\end{definition}

\smallskip

\textbf{Hardy--Littlewood rearrangement inequality.} Let $f,g$ be nonnegative measurable functions on $\mathbb{R}^{N}$. Then
\begin{equation}\label{Hardy-Littlehood}
\int_{\mathbb{R}^{N}} f(x)g(x)\,dx
\leq \int_{0}^{\infty} f^{*}(t)\, g^{*}(t)\, dt.
\end{equation}

\smallskip

The following lemma shows that the  virtual extremizers set $\mathcal{Z}$ associated with the fractional Hardy inequality \eqref{Fractional Hardy} is contained in $L^{p^{*}_{s}, \infty}(\mathbb{R}^{N})$.

\begin{lemma}\label{Marcinkiewicz}
Let $sp<N$ and define
\begin{equation*}
    \omega(x) := |x|^{-\frac{N-sp}{p}}, \qquad x \in \mathbb{R}^{N} \setminus \{0\}.
\end{equation*}
Then $\omega \in L^{p^{*}_{s},\infty}(\mathbb{R}^{N})$.
\end{lemma}
\begin{proof}
Let $\lambda>0$. We compute the distribution function of $\omega$:
\begin{align*}
\{ x \in \mathbb{R}^{N} : \omega(x) > \lambda \}
=
\left\{ x \in \mathbb{R}^{N} : |x|^{-\frac{N-sp}{p}} > \lambda \right\} =
\left\{ x \in \mathbb{R}^{N} : |x| < \lambda^{-\frac{p}{N-sp}} \right\}.
\end{align*}
Since the measure of a ball of radius $r$ in $ \mathbb{R}^{N}$ is $\left( \frac{\mathbb{S}^{N-1}}{N} \right) r^{N}$, we obtain
\begin{equation*}
\mu_{\omega}(\lambda)
=
\left( \frac{\mathbb{S}^{N-1}}{N} \right) \lambda^{- \frac{Np}{N-sp}}
=
\left( \frac{\mathbb{S}^{N-1}}{N} \right) \lambda^{-p^{*}_{s}}.
\end{equation*}
By definition of the decreasing rearrangement,
\begin{align*}
\omega^{*}(t)
=
\inf \left\{ \lambda>0 : \left(\frac{\mathbb{S}^{N-1}}{N} \right) \lambda^{-p^{*}_{s}} \leq t \right\} & =
\inf \left\{ \lambda>0 : \lambda \geq \left( \frac{\mathbb{S}^{N-1}}{N} \right)^{\frac{1}{p^{*}_{s}}}  \frac{1}{t^{\frac{1}{p^{*}_{s}}}}\right\} \\ & =
\left( \frac{\mathbb{S}^{N-1}}{N} \right)^{\frac{1}{p^{*}_{s}}}  \frac{1}{t^{\frac{1}{p^{*}_{s}}}}.
\end{align*}
Therefore,
\begin{equation*}
\|\omega\|_{L^{p^{*}_{s},\infty}}
=
\sup_{t>0}
t^{\frac{1}{p^{*}_{s}}}
\left( \frac{\mathbb{S}^{N-1}}{N} \right)^{\frac{1}{p^{*}_{s}}}  \frac{1}{t^{\frac{1}{p^{*}_{s}}}}
=
\left( \frac{\mathbb{S}^{N-1}}{N} \right)^{\frac{1}{p^{*}_{s}}} 
<
\infty.
\end{equation*}
This proves that $\omega \in L^{p^{*}_{s},\infty}(\mathbb{R}^N)$.
\end{proof}

The following result establishes a relation between the fractional Hardy potential and its Lorentz norms. This follows from the Hardy–Littlewood rearrangement inequality.

\begin{lemma}\label{Lemma: Hardy potential and Lorentz}
For every measurable function $u$ on $\mathbb{R}^{N}$,
\begin{equation}
\int_{\mathbb{R}^{N}} \frac{|u(x)|^{p}}{|x|^{sp}} \, dx
\leq \left( \frac{\mathbb{S}^{N-1}}{N} \right)^{\frac{sp}{N}} \, \|u\|_{L^{p_s^*,\,p}(\mathbb{R}^{N})}^{p}.
\end{equation}
\end{lemma}
\begin{proof}
By the Hardy--Littlewood rearrangement inequality \eqref{Hardy-Littlehood},
\begin{equation*}
\int_{\mathbb{R}^{N}} \frac{|u(x)|^{p}}{|x|^{sp}} \, dx
\leq \int_{0}^{\infty} (u^{*}(t))^{p} (|x|^{-sp})^{*}(t)\, dt.
\end{equation*}
Let $g(x)=|x|^{-sp}$. Then direct computation gives $g^{*}(t) = \left( \frac{\mathbb{S}^{N-1}}{N} \right)^{\frac{sp}{N}}  \frac{1}{t^{\frac{sp}{N}}}$. Hence,
\begin{equation*}
\int_{\mathbb{R}^{N}} \frac{|u(x)|^{p}}{|x|^{sp}} \, dx
\leq \left( \frac{\mathbb{S}^{N-1}}{N} \right)^{\frac{sp}{N}} \int_{0}^{\infty}
\left(t^{\frac{1}{p^{*}_{s}}} u^{*}(t)\right)^p \frac{dt}{t}
= \left( \frac{\mathbb{S}^{N-1}}{N} \right)^{\frac{sp}{N}} \|u\|_{L^{p_s^*,\,p}}^{p}.
\end{equation*}
This proves the lemma.
\end{proof}

\medskip

\subsection{Emden-Fowler Transformation and related Cylindrical Manifolds}

The Emden-Fowler transformation is a fundamental conformal mapping used to relate the analysis of elliptic partial differential equations on $\mathbb{R}^{N} \setminus \{0\}$ to equations on the infinite cylinder $\mathcal{C} = \mathbb{R} \times \mathbb{S}^{N-1}$. This transformation is particularly vital in the study of critical Sobolev exponents and the classification of singularities for nonlinear elliptic equations.

\medskip

Let $x \in \mathbb{R}^{N} \setminus \{0\}$. We introduce polar coordinates $(r, \theta)$ where $r = |x| \in (0, \infty)$ and $\theta = \frac{x}{|x|} \in \mathbb{S}^{N-1}$. The Emden-Fowler transformation is defined by the diffeomorphism $\Phi: \mathbb{R}^{N} \setminus \{0\} \to \mathbb{R} \times \mathbb{S}^{N-1}$ given by $$t = -\log r, \quad \theta = \frac{x}{|x|}.$$Under this change of variables, the Euclidean metric $g_{\mathbb{R}^N} = dr^2 + r^2 g_{\mathbb{S}^{N-1}}$ is transformed into a product metric. Specifically, since $dr = -e^{-t}dt$, we have $$g_{\mathbb{R}^{N}} = e^{-2t} (dt^2 + g_{\mathbb{S}^{N-1}}).$$This demonstrates that the Euclidean space (minus the origin) is conformally equivalent to the cylinder $\mathcal{C}$ equipped with the standard product metric $g_{\mathcal{C}} = dt^2 + g_{\mathbb{S}^{N-1}}$.

\medskip

In the context of the conformal Laplacian, let $u$ be a function on $\mathbb{R}^{N}$. We define the transformed function $v$ on the cylinder by $$v(t, \theta) = r^{\frac{N-2}{2}} u(r, \theta) = e^{-\frac{N-2}{2}t} u(e^{-t}, \theta).$$The operator $-\Delta_{\mathbb{R}^{N}}$ is related to the operator on the cylinder by $$-\Delta_{\mathbb{R}^{N}} u = r^{-\frac{N+2}{2}} \left( -\partial_t^2 v - \Delta_{\mathbb{S}^{N-1}} v + \left( \frac{N-2}{2} \right)^2 v \right).$$This identity, where $\Delta_{\mathbb{S}^{N-1}}$ denotes the standard negative-definite Laplace-Beltrami operator on the sphere is central to the analysis of the Yamabe problem (see \cite{fowler1931further}).

\subsection{Pseudo-Differential Operators (PDOs)}

In the seminal paper \cite{kohn1965algebra}, Kohn and Nirenberg introduced the modern framework of PDOs. Pseudo-differential operators generalize differential operators, allowing for the study of elliptic and hyperbolic equations with variable coefficients through the lens of Fourier analysis.

\medskip 

A pseudo-differential operator $A$ is defined via its symbol $a(x, \xi)$. We define the standard symbol class $S^m_{\rho, \delta}(\mathbb{R}^{N} \times \mathbb{R}^{N})$ for $m \in \mathbb{R}$ and $0 \leq \delta \leq \rho \leq 1$. A smooth function $a(x, \xi)$ belongs to $S^m_{\rho, \delta}$ if for every pair of multi-indices $\alpha, \beta$, there exists a constant $C_{\alpha, \beta}$ such that $$|\partial_\xi^\alpha \partial_x^\beta a(x, \xi)| \leq C_{\alpha, \beta} (1 + |\xi|)^{m - \rho|\alpha| + \delta|\beta|}.$$

\medskip 

\begin{definition}
   Given a symbol $a \in S^m_{\rho, \delta}$, the associated pseudo-differential operator $A = \text{Op}(a)$ acting on a Schwartz function $u \in \mathcal{S}(\mathbb{R}^{N})$ is defined by the integral $$Au(x) = \frac{1}{(2\pi)^{N}} \int_{\mathbb{R}^{N}} e^{i x \cdot \xi} a(x, \xi) \hat{u}(\xi) \, d\xi,$$where $\hat{u}(\xi)$ denotes the Fourier transform of $u$. 
\end{definition}

\begin{rem}
When $a(x, \xi)$ is a polynomial in $\xi$, $A$ reduces to a classical linear partial differential operator. For the calculus of symbols we refer to \cite{hormander2007analysis}.
The connection between Euclidean harmonic analysis and the symbolic calculus can be found in \cite{stein1993harmonic}.
\end{rem}

Now we can list some properties: 

\begin{itemize}
    \item[1)]  Adjoints: For $a \in S^m_{1, 0}$, the adjoint $A^*$ is also a PDO with symbol $a^* \in S^m_{1, 0}$.
    \item[2)] Composition: If $A \in \text{Op}(S^{m_1}_{\rho, \delta})$ and $B \in \text{Op}(S^{m_2}_{\rho, \delta})$ with $0 \leq \delta < \rho \leq 1$, then the composition $AB \in \text{Op}(S^{m_1 + m_2}_{\rho, \delta})$.
    \item[3)] Sobolev Continuity: An operator $A \in \text{Op}(S^m_{1, 0})$ is a bounded linear map from $H^s(\mathbb{R}^{N})$ to $H^{s-m}(\mathbb{R}^{N})$ for any $s \in \mathbb{R}$.
\end{itemize}

\subsection{ The Mellin Transform}

In domains with conical singularities or cylindrical ends (such as $\mathbb{R}^{N} \setminus \{0\}$ under the Emden-Fowler map), the standard Fourier transform is often replaced by the Mellin transform to better respect the scaling symmetry of the domain.
\begin{definition}
    Let $f(r)$ be a function defined on $\mathbb{R}_+ = (0, \infty)$. The Mellin transform $\mathcal{M}f$ is defined for $s \in \mathbb{C}$ by $$(\mathcal{M}f)(s) = \tilde{f}(s) = \int_0^\infty r^{s-1} f(r) \, dr.$$The integral converges in a vertical strip $a < \text{Re}(s) < b$, determined by the growth of $f$ at $0$ and $\infty$. The inverse Mellin transform is given by $$f(r) = \frac{1}{2\pi i} \int_{\gamma-i\infty}^{\gamma+i\infty} r^{-s} \tilde{f}(s) \, ds,$$ where $\gamma \in (a, b)$.
\end{definition}
The Mellin transform is the natural tool for ``Euler-type" differential operators. By integration by parts, assuming boundary terms vanish, we have $$\mathcal{M}(r \partial_r f)(s) = -s \tilde{f}(s).$$By applying the Emden-Fowler change of variables $r = e^{-t}$ (and thus $dr = -e^{-t} dt$), the Mellin transform of $f(r)$ is unitarily equivalent to the Fourier transform of $g(t) = f(e^{-t})$. Noting the reversal of integration bounds as $r \to 0$ corresponds to $t \to \infty$:$$\int_0^\infty r^{s-1} f(r) dr = \int_{\infty}^{-\infty} (e^{-t})^{s-1} f(e^{-t}) (-e^{-t} dt) = \int_{-\infty}^\infty e^{-st} g(t) dt.$$Writing $s = \gamma + i\xi$, this becomes the standard Fourier transform of $g(t)e^{-\gamma t}$ with respect to the frequency variable $\xi$. Thus, taking $s = \gamma + i\xi$, we recover the Fourier transform along the vertical line $\text{Re}(s) = \gamma$.

\section{Fractional Sobolev inequalities and Gagliardo-type estimates}\label{Section : 3}

In this section, we recall the fractional Sobolev inequalities. We exploit the scaling properties of these inequalities on bounded Lipschitz domains. In particular, when considering an appropriate weighted Gagliardo seminorm, the constant in the fractional Sobolev inequality remains independent of the scaling parameter. This scaling-invariant property plays a crucial role in the proof of our main results. Furthermore, we establish some basic inequalities for functions that will be essential in our analysis. Also note that fractional Poincar\'e inequality for measures more general than L\'evy measures was proved in \cite{MRS} for $p=2.$

\smallskip

The next lemma gives a scaling property of the fractional Sobolev inequality on bounded Lipschitz domains with scaling parameter $\lambda>0$. In particular, it provides a constant that depends explicitly on $\lambda>0$. 

\begin{lemma}\label{Lemma: fractional Sobolev inequality}
Let $\Omega$ be a bounded Lipschitz domain in $\mathbb{R}^{N}$, and let $sp<N$. Define $\Omega_{\lambda}= \{ \lambda x : x \in \Omega \}$. Then for all $u \in W^{s,p}(\Omega_{\lambda})$, there exists a constant $C=C(N,s,p, \Omega)>0$ such that 
\begin{align*}
    \left( \fint_{\Omega_{\lambda}} |u(x)-(u)_{\Omega_{\lambda}}|^{p^{*}_{s}} \, dx \right)^{\frac{1}{p^{*}_{s}}} \leq C \left( \lambda^{sp-N} \int_{\Omega_{\lambda}} \int_{\Omega_{\lambda}} \frac{|u(x) - u(y)|^{p}}{|x-y|^{N+sp}}  \, dx \, dy \right)^{\frac{1}{p}},
\end{align*}
where $p^{*}_{s} = \frac{Np}{N-sp}$.
\end{lemma}
    \begin{proof}
  Let  $\Omega$ be a bounded Lipschitz domain in $\mathbb{R}^{N}$. Then from Theorem  $6.7$ of  \cite{di2012hitchhikers}, we have
\begin{equation*}
      \|u\|_{L^{p^{*}_{s}}(\Omega)} \leq C \|u\|_{W^{s,p}(\Omega)},
  \end{equation*}
   where  $C=C(N,p,s, \Omega)$ is a constant. Applying the above inequality with  $u-(u)_{\Omega}$, and using fractional Poincar\'e inequality \cite[Theorem $3.9$]{EdmundsBook2023}, we obtain
\begin{equation*}
       \left(  \fint_{\Omega} |u(x)-(u)_{\Omega}|^{p^{*}_{s} } \, dx \right)^{\frac{1}{p^{*}_{s}}}  \leq C  [u]_{W^{s,p}(\Omega)}  . 
    \end{equation*}
    Let us apply the above inequality to  $u(\lambda x)$ instead of  $u(x)$. This gives
    \begin{equation*}
        \left(  \fint_{\Omega}  \Big|u(\lambda x)-\fint_{\Omega} u(\lambda x) \, dx \Big|^{p^{*}_{s} } \, dx \right)^{\frac{1}{p^{*}_{s}}}  \leq C \left( \int_{\Omega} \int_{\Omega} \frac{|u(\lambda x) - u(\lambda y)|^{p}}{|x-y|^{N+sp}} \, dx \, dy  \right)^{\frac{1}{p}}  . 
    \end{equation*}
    Using the fact $ \fint_{\Omega} u(\lambda x) \, dx = \fint_{\Omega_{\lambda}} u(x) \, dx$, we have
    \begin{equation*}
       \left(  \fint_{\Omega} |u(\lambda x)-(u)_{\Omega_{\lambda}}|^{p^{*}_{s}} \, dx \right)^{\frac{1}{p^{*}_{s}}}  \leq C \left( \int_{\Omega} \int_{\Omega} \frac{|u(\lambda x) - u(\lambda y)|^{p}}{|x-y|^{N+sp}} \, dx \, dy  \right)^{\frac{1}{p}}  . 
    \end{equation*}
    By changing the variables \(X = \lambda x\) and \(Y = \lambda y\), we obtain
    \begin{equation*}
         \left( \fint_{\Omega_{\lambda}} |u(x)-(u)_{\Omega_{\lambda}}|^{p^{*}_{s}} \, dx \right)^{\frac{1}{p^{*}_{s}}}  \leq C \left( \lambda^{sp-N} [u]^{p}_{W^{s,p}(\Omega_{\lambda})} \right)^{\frac{1}{p}}  . 
    \end{equation*}
    This finishes the proof of the lemma.
\end{proof}

\begin{rem}\label{Remark: Scale invariant inequality}
    In the above lemma, if we take $\Omega = \{ x \in \mathbb{R}^{N} : 1<|x|<2 \}$, then for any $x,y \in \Omega_{\lambda}$, we have $|x|^{\frac{N-sp}{2}} \leq (2\lambda)^{\frac{N-sp}{2}}$ and $|y|^{\frac{N-sp}{2}} \leq (2\lambda)^{\frac{N-sp}{2}}$. Therefore, from the above fractional Sobolev inequality, we obtain
    \begin{equation*}
        \left( \fint_{\Omega_{\lambda}} |u(x)-(u)_{\Omega_{\lambda}}|^{p^{*}_{s}} \, dx \right)^{\frac{1}{p^{*}_{s}}} \leq C \left(  \int_{\Omega_{\lambda}} \int_{\Omega_{\lambda}} \frac{|u(x) - u(y)|^{p}}{|x-y|^{N+sp}}  \, \frac{dx}{|x|^{\frac{N-sp}{2}}} \, \frac{dy}{|y|^{\frac{N-sp}{2}}} \right)^{\frac{1}{p}},
    \end{equation*}
    where the constant $C$ does not depend on $\lambda$. Hence, we call this inequality a \textit{scale-invariant weighted fractional Sobolev inequality}. It is called ``weighted" because the right-hand side involves a weighted Gagliardo seminorm, which appears as a remainder term in the fractional Hardy inequality for $p \geq 2$, as given in \eqref{Fractional Hardy with a remainder for p>2}. A similar type of scale-invariant inequality also holds for another weighted Gagliardo seminorm, which appears as a remainder term in the fractional Hardy inequality when $1<p<2$, as given in \eqref{Fractional Hardy with a remainder for p<2}.
\end{rem}

\smallskip

The following lemma establishes a basic inequality for the Gagliardo seminorm involving a weight $\mathcal{K}$. Although the presence of the weight is not essential for the proof, we include it here since this more general formulation will be useful in the proof of our main results. The lemma shows that the Gagliardo seminorm is subadditive with respect to the decomposition of a function into its positive and negative parts.

\begin{lemma}\label{Lemma: Positive and Negative part}
Let $u:\mathbb{R}^{N} \to \mathbb{R}$ be a function and define
\begin{equation*}
    u_{+}(x)=\max\{u(x),0\}, \qquad u_{-}(x)=\max\{-u(x),0\}.
\end{equation*}
Then for $0<s<1$, $1 \leq p < \infty$ and any function $\mathcal{K}$ on $\mathbb{R}^{N} \times \mathbb{R}^{N}$, we have
\begin{align*}
    \int_{\mathbb{R}^{N}} \int_{\mathbb{R}^{N}}  \frac{|u_{+}(x) - u_{+}(y)|^{p}}{|x-y|^{N+sp}} & \mathcal{K}(x,y) \, dx \, dy  +  \int_{\mathbb{R}^{N}} \int_{\mathbb{R}^{N}} \frac{|u_{-}(x) - u_{-}(y)|^{p}}{|x-y|^{N+sp}} \mathcal{K}(x,y) \, dx \, dy \\ & \leq \int_{\mathbb{R}^{N}} \int_{\mathbb{R}^{N}} \frac{|u(x) - u(y)|^{p}}{|x-y|^{N+sp}} \mathcal{K}(x,y) \, dx \, dy.
\end{align*}
\end{lemma}
\begin{proof}
Note that 
\begin{align*}
    \int_{\mathbb{R}^{N}} \int_{\mathbb{R}^{N}}  &\frac{|u_{+}(x) - u_{+}(y)|^{p}}{|x-y|^{N+sp}} \mathcal{K}(x,y) \, dx \, dy  +  \int_{\mathbb{R}^{N}} \int_{\mathbb{R}^{N}} \frac{|u_{-}(x) - u_{-}(y)|^{p}}{|x-y|^{N+sp}} \mathcal{K}(x,y) \, dx \, dy \\ & = \int_{\mathbb{R}^{N}} \int_{\mathbb{R}^{N}} \frac{|u_{+}(x) - u_{+}(y)|^{p} + |u_{-}(x) - u_{-}(y)|^{p}}{|x-y|^{N+sp}} \mathcal{K}(x,y) \, dx \, dy.
\end{align*}
Therefore, it is sufficient to prove the pointwise inequality: for all $x,y \in \mathbb{R}^N$,
\begin{equation}\label{pointwise estimate}
    |u_{+}(x)-u_{+}(y)|^{p} + |u_{-}(x)-u_{-}(y)|^{p} \leq |u(x)-u(y)|^{p}.
\end{equation}
If $u(x), \, u(y)\geq 0$ or $u(x), \, u(y)\leq 0$ or $u(x)=0$ for some $x,y \in \mathbb{R}^{N}$, then
\begin{equation*}
   |u_{+}(x)-u_{+}(y)|^{p} + |u_{-}(x)-u_{-}(y)|^{p}
=
|u(x)-u(y)|^{p}. 
\end{equation*}
Now, assume $u(x) >0$ and $u(y)<0$ for some $x,y \in \mathbb{R}^{N}$. Then $u_{+}(x) = u(x)$ and $u_{-}(y) = -u(y)$. Then,
\begin{equation*}
    |u_{+}(x)-u_{+}(y)|^{p} + |u_{-}(x)-u_{-}(y)|^{p} = |u_{+}(x)|^{p} + |u_{-}(y)|^{p}.
\end{equation*}
But 
\begin{equation*}
    |u(x)-u(y)|^{p} = |u_{+}(x) + u_{-}(y)|^{p} \geq |u_{+}(x)|^{p} + |u_{-}(y)|^{p}.
\end{equation*}
Therefore, combining the above two inequalities, we obtain
\begin{equation*}
    |u_{+}(x)-u_{+}(y)|^{p} + |u_{-}(x)-u_{-}(y)|^{p} \leq |u(x)-u(y)|^{p}.
\end{equation*}
Hence, in all the cases, the pointwise inequality \eqref{pointwise estimate} holds for all $x,y \in \mathbb{R}^{N}$. This proves the lemma.
\end{proof}

The following lemma establishes a basic inequality involving the averages over two disjoint sets. It is particularly useful for connecting two sets that are far apart by means of iterative interactions through intermediate disjoint sets.

\begin{lemma}\label{Lemma : on two disjoint set}
    Let $E$ and $F$ be two disjoint sets in $\mathbb{R}^{N}$. Then there exists a constant $C= C(N,s,p) > 0$ such that
    \begin{equation}
        |(u)_{E} - (u)_{F}|^{p^{*}_{s}} \leq C \frac{|E \cup F|}{\min \{ |E|, |F| \} }  \fint_{E \cup F} |u(x)-(u)_{E \cup F}|^{p^{*}_{s}} \, dx  . 
    \end{equation}
\end{lemma}
\begin{proof}
Let us consider  $|(u)_{E}-(u)_{F}|^{p^{*}_{s}}$, we have
\begin{align*}
    |(u)_{E}-(u)_{F}|^{p^{*}_{s}} & = |(u)_{E} - (u)_{E \cup F} - (u)_{F} + (u)_{E \cup F}|^{p^{*}_{s}} \\ &
        \leq C|(u)_{E} - (u)_{E \cup F}|^{p^{*}_{s}} +  C | (u)_{F} + (u)_{E \cup F}|^{p^{*}_{s}} \\ &
        = C \left| \fint_{E} \left\{ u(x) - (u)_{E \cup F} \right\}  \, dx  \right|^{p^{*}_{s}} + C \left| \fint_{F} \left\{ u(x) - (u)_{E \cup F} \right\} \, dx \right|^{p^{*}_{s}}.
\end{align*}
    By using H$\ddot{\text{o}}$lder's inequality with  $\frac{1}{p^{*}_{s}} + \frac{1}{(p^{*}_{s})' } = 1$, we have
    \begin{align*}
         |(u)_{E}-(u)_{F}|^{p^{*}_{s}} & \leq  C \fint_{E} |u(x) - (u)_{E \cup F} |^{p^{*}_{s}} \,  dx + C \fint_{F} | u(x) - (u)_{E \cup F} |^{p^{*}_{s}} \,  dx \\ &
       \leq  \frac{C}{\min \{ |E|, |F| \} } \int_{E \cup F} |u(x) - (u)_{E \cup F} |^{p^{*}_{s}} \,  dx \\ & 
       = C \frac{|E \cup F|}{\min \{ |E|, |F| \}} \fint_{E \cup F} |u(x) - (u)_{E \cup F} |^{p^{*}_{s}}  \, dx .
    \end{align*}
    This completes the proof of the lemma.
\end{proof}

 \section{Proof of Theorem \ref{Theorem 1} and Theorem \ref{Theorem 2}}\label{Section : 4}

In this section, we prove Theorem \ref{Theorem 1} and Theorem \ref{Theorem 2}. We begin by employing the fractional Hardy inequality with a remainder for $p \geq 2$, established in \cite[Theorem $1.2$]{Frank2008}, and for $1<p<2$, established in \cite[Theorem $2$]{Dyda2024} and reduce the proof to estimating the remainder term. 

\smallskip

The proof proceeds by decomposing $\mathbb{R}^{N}$ into two sets $A$ and $A^{c}$, where $A$ is a bounded open set to be specified later. On $A$, we further decompose the domain into a suitable family of subsets and apply the scale-invariant weighted fractional Sobolev inequality (see Remark \ref{Remark: Scale invariant inequality}). A summation argument then yields control of the Marcinkiewicz (weak-$L^{p^{*}_{s}}$) norm on $A$.
On the complement $A^{c}$, we establish appropriate decay estimates and use embedding properties of the Marcinkiewicz space to control the corresponding norm. Combining these estimates completes the proof of the theorem.

\smallskip

\subsection{The case \texorpdfstring{$p \geq 2$:}{p geq 2:} Proof of Theorem~\ref{Theorem 1}} Frank and Seiringer in \cite{Frank2008} established the fractional Hardy inequality with a remainder for $p \geq 2$. The fractional Hardy inequality with a remainder, proved in \cite[Theorem 1.2]{Frank2008}, states that

\begin{align}\label{Fractional Hardy with a remainder for p>2}
\int_{\mathbb{R}^{N}} \int_{\mathbb{R}^{N}}  \frac{|u(x)-u(y)|^{p}}{|x-y|^{N+sp}} \, dx \, dy & - \mathcal{C}_{N,s,p} \int_{\mathbb{R}^{N}} \frac{|u(x)|^{p}}{|x|^{sp}} \, dx \nonumber \\ & \geq c_{p} \int_{\mathbb{R}^{N}} \int_{\mathbb{R}^{N}} \frac{|v(x)-v(y)|^{p}}{|x-y|^{N+sp}} \, \frac{dx}{|x|^{ \frac{N-sp}{2}}} \, \frac{dy}{|y|^{\frac{N-sp}{2}}},
\end{align}
where $v(x) = u(x) \omega^{-1}(x) = u(x)|x|^{\frac{N-sp}{p}}$ and $0<c_p \leq 1$ is given by
\begin{equation}\label{Defn: c_p}
    c_p := \min_{0 < \tau < 1/2} \left( (1-\tau)^{p} -\tau^{p} + p \tau^{p-1} \right).
\end{equation}
If $p=2$, then the above inequality is an equality with $c_{2} = 1$. 

\smallskip

Let us first assume $u \in C_c^{1}(\mathbb{R}^{N})$ and assume that $u \geq 0$. Without loss of generality, we normalize $u$ so that
\begin{equation*}
   \| u \|_{L^{p^{*}_{s},p}(\mathbb{R}^{N})} = 1.
\end{equation*}
We denote the remainder term by
\begin{equation*}
  \varepsilon (u) :=  \int_{\mathbb{R}^{N}} \int_{\mathbb{R}^{N}} \frac{|v(x)-v(y)|^{p}}{|x-y|^{N+sp}} \, \frac{dx}{|x|^{ \frac{N-sp}{2}}} \, \frac{dy}{|y|^{\frac{N-sp}{2}}}.
\end{equation*}
Assume first that $\varepsilon(u) > 1$. Using the normalization $\| u \|_{L^{p^{*}_{s},p}(\mathbb{R}^{N})} = 1$, and the embedding $L^{p^{*}_{s},p}(\mathbb{R}^N) \hookrightarrow L^{p^{*}_{s},\infty}(\mathbb{R}^N)$, we obtain
\begin{equation*}
   \varepsilon(u)^{\frac{1}{2p}} >1 =  \| u \|_{L^{p^{*}_{s},p}(\mathbb{R}^{N})} \geq C \| u \|_{L^{p^{*}_{s},\infty} (\mathbb{R}^{N})} \geq C \inf_{a \geq 0}  \| u - \omega_{a} \|_{L^{p^{*}_{s},\infty} (\mathbb{R}^{N})},
\end{equation*}
where $\omega_{a}(x) = a \, |x|^{-\frac{N-sp}{p}}$ and $C= C(N,s,p)>0$. Hence,
\begin{equation}\label{ineqn100}
  \inf_{a \geq 0}  \| u - \omega_{a} \|^{2p}_{L^{p^{*}_{s},\infty} (\mathbb{R}^{N})} \leq C  \varepsilon(u).
\end{equation}
Therefore, it suffices to restrict to the case $\varepsilon(u) \leq 1$. Fix $\gamma > 0$, to be chosen later, and define the set
\begin{equation*}
    A = \{ x \in \mathbb{R}^{N} : v(x) > \varepsilon(u)^{\gamma} \}.
\end{equation*}
Since $u \in C^{1}_{c}(\mathbb{R}^{N})$, it follows that the set $A$ is open and bounded. 

For each $k \in \mathbb{Z}$, define
\begin{equation}\label{Defn: A_{k}}
    A_{k} := \{ x \in \mathbb{R}^{N} : 2^{k} < |x| \leq 2^{k+1} \}.
\end{equation}
Since $\lim_{|x| \to 0} v(x) = \lim_{|x| \to 0} u(x)|x|^{\frac{N-sp}{p}} = 0$, consider the smallest $k_{0} \in \mathbb{Z}$ such that 
\begin{equation*}
    v(x) \leq \varepsilon(u)^{\gamma} \quad \forall \ x \in A_{k_{0}}  \quad \text{and} \quad A \cap A_{k_{0}+1} \neq \phi \quad \Rightarrow \quad (v)_{A_{k_{0}}} \leq \varepsilon(u)^{\gamma}.
\end{equation*}
Therefore, using $0< v(x) - \varepsilon(u)^{\gamma} \leq v(x) - (v)_{A_{k_{0}}}$, for all $x \in A$, we obtain
\begin{align}\label{Ineq01}
    \int_{A} |v(x) - \varepsilon(u)^{\gamma}|^{p^{*}_{s}} \, \frac{dx}{|x|^{N}} & \leq \int_{A} |v(x) - (v)_{A_{k_{0}}}|^{p^{*}_{s}} \, \frac{dx}{|x|^{N}} \nonumber \\ 
    & = \sum_{\substack{k \in \mathbb{Z} \\ A \cap A_{k} \neq \phi}} \int_{A \cap A_{k}} |v(x) - (v)_{A_{k_{0}}}|^{p^{*}_{s}} \, \frac{dx}{|x|^{N}}. 
\end{align}
For each $k \in \mathbb{Z}$ such that $A \cap A_{k} \neq \phi$, we estimate
\begin{align}\label{ineq03}
  \int_{A \cap A_{k}} |v(x) - (v)_{A_{k_{0}}}|^{p^{*}_{s}} \, \frac{dx}{|x|^{N}} & \leq 2^{p^{*}_{s}-1} \int_{A \cap A_{k}} |v(x) - (v)_{A_{k}}|^{p^{*}_{s}} \, \frac{dx}{|x|^{N}} \nonumber \\ & \quad + 2^{p^{*}_{s}-1} |(v)_{A_{k}} - (v)_{A_{k_{0}}}|^{p^{*}_{s}} \int_{A \cap A_{k}}  \, \frac{dx}{|x|^{N}} \nonumber \\ & \leq 2^{p^{*}_{s}-1} \int_{A_{k}} |v(x) - (v)_{A_{k}}|^{p^{*}_{s}} \, \frac{dx}{|x|^{N}} \nonumber \\ & \quad + 2^{p^{*}_{s}-1} |(v)_{A_{k}} - (v)_{A_{k_{0}}}|^{p^{*}_{s}} \int_{A \cap A_{k}}  \, \frac{dx}{|x|^{N}}  .
\end{align}
Applying Lemma \ref{Lemma: fractional Sobolev inequality} with $\Omega = \{ x \in \mathbb{R}^{N} : 1<|x|<2 \}$ and $\lambda = 2^{k}$, we obtain
\begin{align*}
    \int_{A_{k}} |v(x) - (v)_{A_{k}}|^{p^{*}_{s}} \, \frac{dx}{|x|^{N}} & \leq C \fint_{A_{k}} |v(x) - (v)_{A_{k}}|^{p^{*}_{s}} \, dx \\ & \leq C \left( 2^{k(sp-N)} \int_{A_{k}} \int_{A_{k}} \frac{|v(x)-v(y)|^{p}}{|x-y|^{N+sp}} \, dx \, dy  \right)^{\frac{p^{*}_{s}}{p}}.
\end{align*}
Here the constant $C:=C(N,s,p, \Omega)$ does not depend on $k.$ For all $x,y \in A_k$, we have $|x| \leq 2^{k+1}$, $|y| \leq 2^{k+1}$, and hence  $|x|^{\frac{N-sp}{2}} |y|^{\frac{N-sp}{2}} \leq 2^{(k+1)(N-sp)}$. Therefore, we obtain
\begin{align*}
     \int_{A_{k}} |v(x) - (v)_{A_{k}}|^{p^{*}_{s}} \, \frac{dx}{|x|^{N}} \leq C \left( \int_{A_{k}} \int_{A_{k}} \frac{|v(x)-v(y)|^{p}}{|x-y|^{N+sp}} \, \frac{dx}{|x|^{\frac{N-sp}{2}}} \, \frac{dy}{|y|^{\frac{N-sp}{2}}}  \right)^{\frac{p^{*}_{s}}{p}} ,
\end{align*}
where $C= C(N,s,p)>0$. Substituting the above estimate into \eqref{ineq03}, we obtain
\begin{align}\label{ineq04}
    \int_{A \cap A_{k}} |v(x) - (v)_{A_{k_{0}}}|^{p^{*}_{s}} \, \frac{dx}{|x|^{N}} & \leq C \left( \int_{A_{k}} \int_{A_{k}} \frac{|v(x)-v(y)|^{p}}{|x-y|^{N+sp}} \, \frac{dx}{|x|^{\frac{N-sp}{2}}} \, \frac{dy}{|y|^{\frac{N-sp}{2}}}  \right)^{\frac{p^{*}_{s}}{p}} \nonumber \\ & \quad + 2^{p^{*}_{s}-1}|(v)_{A_{k}} - (v)_{A_{k_{0}}}|^{p^{*}_{s}}  \int_{A \cap A_{k}}  \, \frac{dx}{|x|^{N}}  .
\end{align}
Now, applying Lemma \ref{Lemma : on two disjoint set} with $E = A_{\ell}$ and $F = A_{\ell+1}$, for $k_{0} \leq \ell \leq k-1$, we obtain
\begin{align*}
  |(v)_{A_{k}} - (v)_{A_{k_{0}}}|^{p^{*}_{s}} & \leq 2^{p^{*}_{s}-1} \sum_{\ell = k_{0}}^{k-1}  |(v)_{A_{\ell}} - (v)_{A_{\ell+1}}|^{p^{*}_{s}} \\ & \leq C \sum_{\ell = k_{0}}^{k-1} \fint_{A_{\ell} \cup A_{\ell+1}}  |v(x) - (v)_{A_{\ell} \cup A_{\ell+1}}|^{p^{*}_{s}} \, dx, 
\end{align*}
where $C=C(N,s,p)>0$. Next, we apply Lemma \ref{Lemma: fractional Sobolev inequality} with $\Omega = \{x \in \mathbb{R}^{N} : 1 < |x| < 4\}$ and scaling parameter $\lambda = 2^{\ell}$ and using $|x|^{\frac{N-sp}{2}}|y|^{\frac{N-sp}{2}} \leq 2^{(\ell+2)(N-sp)}$ for all $x,y \in A_{\ell} \cup A_{\ell+1}$, we obtain
\begin{align*}
    |(v)_{A_{k}} - (v)_{A_{k_{0}}}|^{p^{*}_{s}} & \leq C \sum_{\ell = k_{0}}^{k-1} \left( 2^{\ell (sp-N)} \int_{A_{\ell} \cup A_{\ell+1}} \int_{A_{\ell} \cup A_{\ell+1}} \frac{|v(x)-v(y)|^{p}}{|x-y|^{N+sp}} \, dx \, dy \right)^{\frac{p^{*}_{s}}{p}} \\ & \leq C \sum_{\ell = k_{0}}^{k-1} \left( \int_{A_{\ell} \cup A_{\ell+1}} \int_{A_{\ell} \cup A_{\ell+1}} \frac{|v(x)-v(y)|^{p}}{|x-y|^{N+sp}} \, \frac{dx}{|x|^{\frac{N-sp}{2}}} \, \frac{dy}{|y|^{\frac{N-sp}{2}}} \right)^{\frac{p^{*}_{s}}{p}} \\ &  \leq C \left( \int_{\mathbb{R}^{N}} \int_{\mathbb{R}^{N}} \frac{|v(x)-v(y)|^{p}}{|x-y|^{N+sp}} \, \frac{dx}{|x|^{\frac{N-sp}{2}}} \, \frac{dy}{|y|^{\frac{N-sp}{2}}} \right)^{\frac{p^{*}_{s}}{p}},
\end{align*}
where $C = C(N,s,p)>0$. Also, using the fact that for any $x \in A \cap A_{k}$, we have $v(x) > \varepsilon(u)^\gamma$, it follows that
\begin{equation*}
    \int_{A \cap A_{k}}  \, \frac{dx}{|x|^{N}} \leq \frac{1}{\varepsilon(u)^{\gamma p}} \int_{A \cap A_{k}} \frac{|u(x)|^{p}}{|x|^{sp}} \, dx.
\end{equation*}
Therefore, combining the above two estimates and using the definition of $\varepsilon(u)$, we obtain
\begin{align*}
    |(v)_{A_{k}} & -  (v)_{A_{k_{0}}}|^{p^{*}_{s}}  \int_{A \cap A_{k}}  \, \frac{dx}{|x|^{N}} \\ & \leq C \left( \int_{\mathbb{R}^{N}} \int_{\mathbb{R}^{N}} \frac{|v(x)-v(y)|^{p}}{|x-y|^{N+sp}} \, \frac{dx}{|x|^{\frac{N-sp}{2}}} \, \frac{dy}{|y|^{\frac{N-sp}{2}}} \right)^{\frac{p^{*}_{s}}{p}}   \frac{1}{\varepsilon(u)^{\gamma p}} \int_{A \cap A_{k}} \frac{|u(x)|^{p}}{|x|^{sp}} \, dx \\ & \leq C \varepsilon(u)^{\frac{p^{*}_{s}}{p} - \gamma p} \int_{A \cap A_{k}} \frac{|u(x)|^{p}}{|x|^{sp}} \, dx.
\end{align*}
Substituting the above estimate into \eqref{ineq04}, we obtain
\begin{align*}
\int_{A \cap A_{k}} |v(x) - (v)_{A_{k_{0}}}|^{p^{*}_{s}} \, \frac{dx}{|x|^{N}} & \leq C \left( \int_{A_{k}} \int_{A_{k}} \frac{|v(x)-v(y)|^{p}}{|x-y|^{N+sp}} \, \frac{dx}{|x|^{\frac{N-sp}{2}}} \, \frac{dy}{|y|^{\frac{N-sp}{2}}}  \right)^{\frac{p^{*}_{s}}{p}} \nonumber \\ & \quad +  C \varepsilon(u)^{\frac{p^{*}_{s}}{p} - \gamma p} \int_{A \cap A_{k}} \frac{|u(x)|^{p}}{|x|^{sp}} \, dx.
\end{align*}
Therefore, from \eqref{Ineq01}, the estimates above, and the definition of $\varepsilon(u)$, together with Lemma \ref{Lemma: Hardy potential and Lorentz} and the normalization $\| u \|_{L^{p^{*}_{s},p}(\mathbb{R}^{N})} = 1$, we obtain
\begin{align*}
    \int_{A} |v(x) - \varepsilon(u)^{\gamma}|^{p^{*}_{s}} \, \frac{dx}{|x|^{N}} & \leq \sum_{\substack{k \in \mathbb{Z} \\ A \cap A_{k} \neq \phi}} \int_{A \cap A_{k}} |v(x) - (v)_{A_{k_{0}}}|^{p^{*}_{s}} \, \frac{dx}{|x|^{N}} \\ & \leq C \varepsilon(u)^{\frac{p^{*}_{s}}{p}} + C \varepsilon(u)^{\frac{p^{*}_{s}}{p} - \gamma p} \int_{\mathbb{R}^{N}} \frac{|u(x)|^{p}}{|x|^{sp}} \, dx \\ & \leq C \left(\varepsilon(u)^{\frac{p^{*}_{s}}{p}} +  \varepsilon(u)^{\frac{p^{*}_{s}}{p} - \gamma p} \right),
\end{align*}
where $C=C(N,s,p)>0$. Choose $\gamma = \frac{p^{*}_{s}}{2 p^{2}}$.  Then, using the assumption $\varepsilon(u) \leq 1$, we obtain
\begin{align*}
    \int_{A} |v(x) - \varepsilon(u)^{\gamma}|^{p^{*}_{s}} \, \frac{dx}{|x|^{N}} & \leq C \left( \varepsilon(u)^{\frac{p^{*}_{s}}{p}} + \varepsilon(u)^{\frac{p^{*}_{s}}{2p}} \right) \\ & \leq C \left( \varepsilon(u)^{\frac{p^{*}_{s}}{2p}} + \varepsilon(u)^{\frac{p^{*}_{s}}{2p} }  \right) \leq C \varepsilon(u)^{\frac{p^{*}_{s}}{2p}} .
\end{align*}
Therefore, using $\omega(x) = |x|^{- \frac{N-sp}{p}}$, we obtain using Lorentz-Embedding, 
\begin{align*}
    \| u - \varepsilon(u)^{\gamma} \omega \|_{L^{p^{*}_{s}, \infty}(A)} & \leq \left( \int_{A} | u(x) - \varepsilon(u)^{\gamma} \omega(x) |^{p^{*}_{s}} \, dx \right)^{\frac{1}{p^{*}_{s}}} \\ &  =  \left( \int_{A} |v(x) - \varepsilon(u)^{\gamma}|^{p^{*}_{s}} \, \frac{dx}{|x|^{N}} \right)^{\frac{1}{p^{*}_{s}}}  \leq C \varepsilon(u)^{\frac{1}{2p}},
\end{align*}
where the constant $C:= C(N, s, p)$ is independent of the set $A.$
Consequently,
\begin{equation*}
    \| u - \varepsilon(u)^{\gamma} \omega \|^{2p}_{L^{p^{*}_{s}, \infty}(A)} \leq C \varepsilon(u) .
\end{equation*}
For any $x \in A^{c}$, we have $v(x) \leq \varepsilon(u)^{\gamma}$. In particular,
\begin{equation*}
 \omega(x)^{-1} |u(x)- \varepsilon(u)^{\gamma} \omega(x)| =  |v(x) - \varepsilon(u)^{\gamma}| \leq 2 \varepsilon(u)^{\gamma}, \quad \forall \ x \in A^{c},
\end{equation*}
and therefore
\begin{equation*}
     |u(x)- \varepsilon(u)^{\gamma} \omega(x)| \leq 2 \varepsilon(u)^{\gamma} \omega(x), \quad \forall \ x \in A^{c} .
\end{equation*}
Using that $\omega \in L^{p^{*}_{s}, \infty}(\mathbb{R}^{N})$, the definition of $\gamma$, and the assumption $\varepsilon(u) \leq 1$, we obtain
\begin{equation*}
    \|u- \varepsilon(u)^{\gamma} \omega\|_{L^{p^{*}_{s}, \infty}(A^{c})} \leq C \varepsilon(u)^{\gamma} \| \omega \|_{L^{p^{*}_{s}, \infty}(A^{c})} \leq C \varepsilon(u)^{\frac{p^{*}_{s}}{2p^{2}}} \leq C \varepsilon(u)^{\frac{1}{2p}}. 
\end{equation*}
Consequently,
\begin{equation*}
     \|u- \varepsilon(u)^{\gamma} \omega\|^{2p}_{L^{p^{*}_{s}, \infty}(A^{c})} \leq C \varepsilon(u).
\end{equation*}
Combining this with the estimate on $A$, we conclude that
\begin{align*}
    \|u- \varepsilon(u)^{\gamma} \omega\|_{L^{p^{*}_{s}, \infty}(\mathbb{R}^{N})} & \leq C \left( \|u- \varepsilon(u)^{\gamma} \omega\|_{L^{p^{*}_{s}, \infty}(A)}  + \|u- \varepsilon(u)^{\gamma} \omega\|_{L^{p^{*}_{s}, \infty}(A^{c})} \right) \\ & \leq C \varepsilon(u)^{\frac{1}{2p}}. 
\end{align*}
Therefore, combining the cases $\varepsilon(u) > 1$ and $\varepsilon(u) \leq 1$, and under the normalization $\| u \|_{L^{p^{*}_{s},p}(\mathbb{R}^{N})} =1$, we obtain
\begin{equation*}
   \inf_{a \geq 0} \|u-  \omega_{a}\|^{2p}_{L^{p^{*}_{s}, \infty}(\mathbb{R}^{N})} \leq \|u- \varepsilon(u)^{\gamma} \omega\|^{2p}_{L^{p^{*}_{s}, \infty}(\mathbb{R}^{N})} \leq C \varepsilon(u). 
\end{equation*}
To remove the normalization assumption $\| u \|_{L^{p^{*}_{s},p}(\mathbb{R}^{N})} =1$, we apply the previous estimate to the function
\begin{equation*}
    \widetilde{u} = \frac{u}{\| u \|_{L^{p^{*}_{s},p}(\mathbb{R}^{N})} }.
\end{equation*} 
Since $\widetilde{u}$ satisfies the normalization, we conclude that 
\begin{align*}
  \inf_{a \geq 0} \frac{\| u - \omega_{a}\|^{2p}_{L^{p^{*}_{s},\infty}(\mathbb{R}^{N})}}{ \| u \|^{2p}_{L^{p^{*}_{s},p}(\mathbb{R}^{N})} }  \leq C \frac{\varepsilon(u)}{ \| u \|^{p}_{L^{p^{*}_{s},p}(\mathbb{R}^{N})}  }.   
\end{align*}
Therefore, for $u \geq 0$, we have
\begin{equation}\label{ineq1234}
    \inf_{a \geq 0} \| u - \omega_{a}\|^{2p}_{L^{p^{*}_{s},\infty}(\mathbb{R}^{N})} \leq C \varepsilon(u)\| u \|^{p}_{L^{p^{*}_{s},p}(\mathbb{R}^{N})}  .
\end{equation}
Now, given any $u \in C^{1}_{c}(\mathbb{R}^{N})$, define $u_{+}(x) = \max \{ u(x),0 \} $ and $u_{-}(x) = \max \{ -u(x),0 \}$, the positive and the negative parts of $u$ respectively, so that $u=u_{+}- u_{-}$. Then,
\begin{align*}
  \inf_{a \in \mathbb{R}}  \| u -\omega_{a}\|_{L^{p^{*}_{s},\infty}(\mathbb{R}^{N})} & = \inf_{b,c \geq 0 }  \| u_{+} - u_{-} - \omega_{b-c}\|_{L^{p^{*}_{s},\infty}(\mathbb{R}^{N})}  \\ & \leq \inf_{b \geq 0}  \| u_{+} -\omega_{b}\|_{L^{p^{*}_{s},\infty}(\mathbb{R}^{N})}  + \inf_{c \geq 0}  \| u_{-} -\omega_{c}\|_{L^{p^{*}_{s},\infty}(\mathbb{R}^{N})}.
\end{align*}
Applying the nonnegative functions $u_{+}$ and $u_{-}$ in \eqref{ineq1234} and using Lemma \ref{Lemma: Positive and Negative part}, we obtain 
\begin{align*}
    \sum_{\pm} \inf_{a \geq 0} \| u_{\pm} - \omega_{a} \|_{L^{p^{*}_{s}, \infty}(\mathbb{R}^{N})} & \leq C \sum_{\pm}  \left( \varepsilon(u_{\pm}) \right)^{\frac{1}{2p}} \| u_{\pm} \|^{\frac{1}{2}}_{L^{p^{*}_{s},p}(\mathbb{R}^{N})}   \\ & \leq C \left( \varepsilon(u_{+})+ \varepsilon(u_{-}) \right)^{\frac{1}{2p}}  \| u \|^{\frac{1}{2}}_{L^{p^{*}_{s},p}(\mathbb{R}^{N})}  \\ & \leq C \left( \varepsilon(u)  \right)^{\frac{1}{2p}}  \| u \|^{\frac{1}{2}}_{L^{p^{*}_{s},p}(\mathbb{R}^{N})} .
\end{align*}
Combining the above two inequalities and using Lemma \ref{Lemma: Hardy potential and Lorentz}, we obtain
\begin{align*}
\left( \frac{N}{\mathbb{S}^{N-1}} \right)^{\frac{sp}{N}} \inf_{a \in \mathbb{R}} \frac{ \| u -\omega_{a}\|^{2p}_{L^{p^{*}_{s},\infty}(\mathbb{R}^{N})}}{\| u \|^{2p}_{L^{p^{*}_{s},p}(\mathbb{R}^{N})}} & \int_{\mathbb{R}^{N}} \frac{|u(x)|^{p}}{|x|^{sp}} \, dx \\ & \leq  \inf_{a \in \mathbb{R}} \frac{ \| u -\omega_{a}\|^{2p}_{L^{p^{*}_{s},\infty}(\mathbb{R}^{N})}}{\| u \|^{2p}_{L^{p^{*}_{s},p}(\mathbb{R}^{N})}}  \| u \|^{p}_{L^{p^{*}_{s},p}(\mathbb{R}^{N})}  \leq C \varepsilon(u).
\end{align*}
Therefore, using the definition of $d_{s,p}$ and $\varepsilon(u)$, and applying the fractional Hardy inequality with a remainder for $p \geq 2$ in \eqref{Fractional Hardy with a remainder for p>2}, we obtain
\begin{align*}
\int_{\mathbb{R}^{N}} \int_{\mathbb{R}^{N}}  \frac{|u(x)-u(y)|^{p}}{|x-y|^{N+sp}} \, dx \, dy & - \mathcal{C}_{N,s,p} \int_{\mathbb{R}^{N}} \frac{|u(x)|^{p}}{|x|^{sp}} \, dx \geq C (d_{s,p}(u))^{2p}
    \int_{\mathbb{R}^{N}} \frac{|u(x)|^{p}}{|x|^{sp}} \, dx,
\end{align*}
for all $u \in C^1_c(\mathbb{R}^N).$ We aim to extend the above stability inequality to functions
$u \in \dot W^{s,p}(\mathbb{R}^{N}).$ To prove this, we exploit a density argument. We recall  $\dot{W}^{s,p}(\mathbb{R}^N)$ the closure of $C_c^\infty(\mathbb{R}^N)$ with respect to the Gagliardo seminorm $[\cdot]_{W^{s,p}(\mathbb{R}^N)}$. Indeed by density, there exists \(\{u_n\}\subset C_c^\infty(\mathbb{R}^N)\) such that
\[
[u_n-u]_{W^{s,p}(\mathbb{R}^N)}\to0 .
\]
Using the fractional Sobolev and Lorentz-space embeddings (see \cite[Theorem $6.5$]{di2012hitchhikers} and \cite[Theorem $4.1$]{Frank2008}), and 
\cite[Proposition 1.1.9]{GrafakosBook}, we obtain
\[
u_n\to u
\quad \text{in } 
L^{p_s^*}(\mathbb{R}^N),
\ L^{p_s^*,p}(\mathbb{R}^N),
\ \text{and } 
L^{p_s^*,\infty}(\mathbb{R}^N).
\]
Hence,
\[
d_{s,p}(u_n)\to d_{s,p}(u).
\]
Moreover, up to a subsequence \(u_n\to u\) a.e., and therefore Fatou's lemma yields
\[
\int_{\mathbb{R}^N}\frac{|u(x)|^p}{|x|^{sp}}\,dx
\le
\liminf_{n\to\infty}
\int_{\mathbb{R}^N}\frac{|u_n(x)|^p}{|x|^{sp}}\,dx .
\]
Therefore, passing to the limit in the inequality for $u_n \in C^1_c(\mathbb{R}^N),$ we conclude the Theorem~\ref{Theorem 1}.

\medskip

\subsection{The case \texorpdfstring{$1<p < 2$:}{p < 2} Proof of Theorem~\ref{Theorem 2}} Dyda and Kijaczko in \cite{Dyda2024} extended the fractional Hardy inequality with a reminder to the case $1<p<2$. They established the result in weighted fractional Sobolev space, where the weight function is of the form $|x|^{\alpha} |y|^{\beta}$.  In particular, for the case $\alpha=\beta =0$, they established the following result (unweighted fractional Hardy inequality with a remainder). We assume the notation $a^{\langle b \rangle} = |a|^{b} \, \text{sgn}(a)$.
     \begin{align}\label{Fractional Hardy with a remainder for p<2}
\int_{\mathbb{R}^{N}} \int_{\mathbb{R}^{N}}  & \frac{|u(x)-u(y)|^{p}}{|x-y|^{N+sp}} \, dx \, dy - \mathcal{C}_{N,s,p} \int_{\mathbb{R}^{N}} \frac{|u(x)|^{p}}{|x|^{sp}} \, dx \nonumber \\ & \geq c^{*}_{p} \int_{\mathbb{R}^{N}} \int_{\mathbb{R}^{N}} \frac{(v(x)^{\langle p/2 \rangle}-v(y)^{\langle p/2 \rangle})^{2}}{|x-y|^{N+sp}} \, W(x,y) \, dx \, dy,
    \end{align}
    where $v(x)=|x|^{\frac{N-sp}{p}}u(x)= \omega^{-1}(x) u(x)$, 
    \begin{equation*}
        W(x,y) = \min \left\{ |x|^{-\frac{N-sp}{p}},  |y|^{-\frac{N-sp}{p}} \right\} \max \left\{ |x|^{-\frac{N-sp}{p}},  |y|^{-\frac{N-sp}{p}} \right\}^{p-1}, 
    \end{equation*}
   and the constant 
    \begin{equation*}
        c^{*}_{p}:=  \max \left\{ \frac{p-1}{p}, \frac{p(p-1)}{2} \right\}.
    \end{equation*}
    When $u \geq 0$ or $u$ is complex-valued the constant $c^{*}_{p}$ in the inequality above can be replaced by $p-1$.

    \smallskip

    Let $u \in C_c^{1}(\mathbb{R}^{N})$ and assume that $u \geq 0$. Therefore, $v(x)^{\langle p/2 \rangle}=  (v(x))^{\frac{p}{2}}$ for all $x \in \mathbb{R}^{N}$. Again, without loss of generality, we normalize $u$ so that
\begin{equation*}
    \| u \|_{L^{p^{*}_{s},p}(\mathbb{R}^{N})}   = 1.
\end{equation*}
For $1<p<2$, we define the remainder term by
\begin{equation*}
  \varepsilon_{W} (u) :=  \int_{\mathbb{R}^{N}} \int_{\mathbb{R}^{N}} \frac{((v(x))^{\frac{p}{2}}-(v(y))^{\frac{p}{2}})^{2}}{|x-y|^{N+sp}} \, W(x,y) \, dx \, dy,
\end{equation*}
and let
\begin{equation*}
    W^{p}_{r}(x,y) := \min \left\{ |x|^{-\frac{N-sp}{r}},  |y|^{-\frac{N-sp}{r}} \right\} \max \left\{ |x|^{-\frac{N-sp}{r}},  |y|^{-\frac{N-sp}{r}} \right\}^{r-1}.
\end{equation*}
Note that $W^{p}_{p}=W$. Assume first that $\varepsilon_{W}(u) > 1$. Using the normalization $\| u \|_{L^{p^{*}_{s},p}(\mathbb{R}^{N})} = 1$, and the embedding $L^{p^{*}_{s},p}(\mathbb{R}^N) \hookrightarrow L^{p^{*}_{s},\infty}(\mathbb{R}^N)$, we obtain
\begin{align*}
   \varepsilon_{W}(u)^{\frac{1}{2p}} >1 = \| u \|_{L^{p^{*}_{s},p}(\mathbb{R}^{N})}  \geq C \| u \|_{L^{p^{*}_{s},\infty} (\mathbb{R}^{N})} .
\end{align*}
Using $\| (u^{\frac{p}{2}})^{\frac{2}{p}} \|_{L^{p^{*}_{s},\infty} (\mathbb{R}^{N})}  = \| u^{\frac{p}{2}} \|^{\frac{2}{p}}_{L^{q,\infty}(\mathbb{R}^{N})} $, where $q=\frac{2N}{N-sp}$, we obtain
\begin{equation*}
    \varepsilon_{W}(u)^{\frac{1}{2p}} >  C   \| u^{\frac{p}{2}} \|^{\frac{2}{p}}_{L^{q,\infty}(\mathbb{R}^{N})} \geq C\inf_{a \geq 0}  \| u^{\frac{p}{2}} -  \omega^{\frac{p}{2}}_{a} \|^{\frac{2}{p}}_{L^{q,\infty} (\mathbb{R}^{N})},
\end{equation*}
where $\omega_{a}(x) = a \, |x|^{-\frac{N-sp}{p}}$ and $C= C(N,s,p)>0$. Hence,
\begin{equation*}
  \inf_{a \geq 0}  \| u^{\frac{p}{2}} -  \omega^{\frac{p}{2}}_{a} \|^{4}_{L^{q,\infty} (\mathbb{R}^{N})} \leq C  \varepsilon_{W}(u).
\end{equation*}
Therefore, it is enough to consider the regime $\varepsilon_{W}(u) \leq 1$. Fix $\gamma_{W} > 0$, to be chosen later, and define the set
\begin{equation*}
    A_{W} = \{ x \in \mathbb{R}^{N} : (v(x))^{\frac{p}{2}} > \varepsilon_{W}(u)^{\gamma_{W}} \}.
\end{equation*}
Since $u \in C^{1}_{c}(\mathbb{R}^{N})$, it follows that the set $A_{W}$ is open and bounded. For each $k \in \mathbb{Z}$, consider the set $A_{k}$, defined in \eqref{Defn: A_{k}}. Also, $\lim_{|x| \to 0} v(x) = 0$. Therefore, consider the smallest $j_{0} \in \mathbb{Z}$ such that 
\begin{equation*}
    (v(x))^{\frac{p}{2}} \leq \varepsilon_{W}(u)^{\gamma_{W}}, \quad \forall \ x \in A_{j_{0}}  \quad \text{and} \quad A_{W} \cap A_{j_{0}+1} \neq \phi \quad \Rightarrow \quad (v^{\frac{p}{2}})_{A_{j_{0}}} \leq \varepsilon_{W}(u)^{\gamma_{W}}.
\end{equation*}
Therefore, using $0< (v(x))^{\frac{p}{2}} - \varepsilon_{W}(u)^{\gamma_{W}} \leq (v(x))^{\frac{p}{2}} - (v^{\frac{p}{2}})_{A_{j_{0}}}$, for all $x \in A_{W}$, we obtain
\begin{align}\label{ineq999}
    \int_{ A_{W}} |(v(x))^{\frac{p}{2}} - \varepsilon_{W}(u)^{\gamma_{W}}|^{q} \, \frac{dx}{|x|^{N}} & \leq \int_{ A_{W}} |(v(x))^{\frac{p}{2}} - (v^{\frac{p}{2}})_{A_{j_{0}}}|^{q} \, \frac{dx}{|x|^{N}} \nonumber \\ 
    & = \sum_{\substack{k \in \mathbb{Z} \\ A_{W} \cap A_{k} \neq \phi}} \int_{A_{W} \cap A_{k}} |(v(x))^{\frac{p}{2}} - (v^{\frac{p}{2}})_{A_{j_{0}}}|^{q} \, \frac{dx}{|x|^{N}} ,
\end{align}
where $q= \frac{2N}{N-sp}$. Define $v_{2}(x) = u(x)|x|^{\frac{N-2s}{p}}$ for all $x \in \mathbb{R}^{N}$. Note that $v_{p}=v$. For each $k \in \mathbb{Z}$ such that $A_{W} \cap A_{k} \neq \phi$, we have
\begin{align}\label{ineq09}
    \int_{A_{W} \cap A_{k}} |(v_{2}(x))^{\frac{p}{2}} - (v^{\frac{p}{2}}_{2})_{A_{j_{0}}}|^{2^{*}_{s}} \, \frac{dx}{|x|^{N}} & \leq 2^{2^{*}_{s}-1} \int_{A_{k}} |(v_{2}(x))^{\frac{p}{2}} - (v^{\frac{p}{2}}_{2})_{A_{k}}|^{2^{*}_{s}} \, \frac{dx}{|x|^{N}} \nonumber \\ & \quad + 2^{2^{*}_{s}-1} |(v^{\frac{p}{2}}_{2})_{A_{k}} - (v^{\frac{p}{2}}_{2})_{A_{j_{0}}}|^{2^{*}_{s}} \int_{A_{W} \cap A_{k}}  \, \frac{dx}{|x|^{N}}  .
\end{align}
Now, applying Lemma \ref{Lemma: fractional Sobolev inequality} with $\Omega = \{ x \in \mathbb{R}^{N} : 1<|x|<2 \}$ and $\lambda = 2^{k}$, we obtain
\begin{align*}
    \int_{A_{k}} |(v_{2}(x))^{\frac{p}{2}} - (v^{\frac{p}{2}}_{2})_{A_{k}}|^{2^{*}_{s}} \, \frac{dx}{|x|^{N}} & \leq C \fint_{A_{k}} |(v_{2}(x))^{\frac{p}{2}} - (v^{\frac{p}{2}}_{2})_{A_{k}}|^{2^{*}_{s}} \, dx \\ & \leq C \left( 2^{k(2s-N)} \int_{A_{k}} \int_{A_{k}} \frac{|(v_{2}(x))^{\frac{p}{2}}-(v_{2}(y))^{\frac{p}{2}}|^{2}}{|x-y|^{N+2s}} \, dx \, dy  \right)^{\frac{2^{*}_{s}}{2}} .
\end{align*}
For all $x,y \in A_{k}$, we have $|x| \leq 2^{k+1}$, $|y| \leq 2^{k+1}$, and hence
\begin{equation*}
    2^{(k+1)(2s-N)} \leq   \min \left\{ |x|^{-\frac{N-2s}{r}},  |y|^{-\frac{N-2s}{r}} \right\} \max \left\{ |x|^{-\frac{N-2s}{r}},  |y|^{-\frac{N-2s}{r}} \right\}^{r-1} = W^{2}_{r}(x,y).
\end{equation*}
Therefore, we obtain
\begin{align*}
     \int_{A_{k}} |(v_{2}(x))^{\frac{p}{2}} - (v^{\frac{p}{2}}_{2})_{A_{k}}|^{2^{*}_{s}} \, \frac{dx}{|x|^{N}} \leq C \left( \int_{A_{k}} \int_{A_{k}} \frac{|(v_{2}(x))^{\frac{p}{2}}-(v_{2}(y))^{\frac{p}{2}}|^{2}}{|x-y|^{N+2s}} \, W^{2}_{r} (x,y) \, dx \, dy \right)^{\frac{2^{*}_{s}}{2}} ,
\end{align*}
where $C= C(N,s)>0$. Substituting the above estimate into \eqref{ineq09}, we obtain
\begin{align*}
    \int_{A_{W} \cap A_{k}} |(v_{2}(x))^{\frac{p}{2}} & - (v^{\frac{p}{2}}_{2})_{A_{j_{0}}}|^{2^{*}_{s}} \, \frac{dx}{|x|^{N}} \\ & \leq C \left( \int_{A_{k}} \int_{A_{k}} \frac{|(v_{2}(x))^{\frac{p}{2}}-(v_{2}(y))^{\frac{p}{2}}|^{2}}{|x-y|^{N+2s}} \, W^{2}_{r} (x,y) \, dx \, dy \right)^{\frac{2^{*}_{s}}{2}} \\ & \quad +  2^{2^{*}_{s}-1} |(v^{\frac{p}{2}}_{2})_{A_{k}} - (v^{\frac{p}{2}}_{2})_{A_{j_{0}}}|^{2^{*}_{s}} \int_{A_{W} \cap A_{k}}  \, \frac{dx}{|x|^{N}} . 
\end{align*}
In the above inequality, by replacing $s$ with $\frac{sp}{2}$ (note that for $1<p<2$, we have $\frac{sp}{2} \in (0,1)$) and taking $r=p$, we obtain
\begin{align}\label{ineq123}
     \int_{A_{W} \cap A_{k}} |(v(x))^{\frac{p}{2}} & - (v^{\frac{p}{2}})_{A_{j_{0}}}|^{q} \, \frac{dx}{|x|^{N}} \nonumber \\ & \leq C \left( \int_{A_{k}} \int_{A_{k}} \frac{|(v(x))^{\frac{p}{2}}-(v(y))^{\frac{p}{2}}|^{2}}{|x-y|^{N+sp}} \, W (x,y) \, dx \, dy \right)^{\frac{q}{2}} \nonumber \\ & \quad +  C |(v^{\frac{p}{2}})_{A_{k}} - (v^{\frac{p}{2}})_{A_{j_{0}}}|^{q} \int_{A_{W} \cap A_{k}}  \, \frac{dx}{|x|^{N}} ,
\end{align}
where $q = \frac{2N}{N-sp}$. In this inequality, we have used the fact that $v_{p} = v$ and $W^{p}_{p} = W$.

\smallskip

Similarly, by following the same steps as in the case $p \geq 2$, we obtain for the case $1<p<2$,
\begin{align*}
  |(v^{\frac{p}{2}}_{2})_{A_{k}} & - (v^{\frac{p}{2}}_{2})_{A_{j_{0}}}|^{2^{*}_{s}}  \leq C \sum_{\ell=j_{0}}^{k-1} |(v^{\frac{p}{2}}_{2})_{A_{\ell}} - (v^{\frac{p}{2}}_{2})_{A_{\ell+1}}|^{2^{*}_{s}}   \\ & \leq C \sum_{\ell=j_{0}}^{k-1} \left( \int_{A_{\ell} \cup A_{\ell+1}} \int_{A_{\ell} \cup A_{\ell+1}} \frac{|(v_{2}(x))^{\frac{p}{2}}-(v_{2}(y))^{\frac{p}{2}}|^{2}}{|x-y|^{N+2s}} \, W^{2}_{r} (x,y) \, dx \, dy \right)^{\frac{2^{*}_{s}}{2}} .
\end{align*}
Again, in the above inequality, by replacing $s$ with $\frac{sp}{2}$ (note that for $1<p<2$, we have $\frac{sp}{2} \in (0,1)$), taking $r=p$ and using the fact $v_{p}=v$ and $W^{p}_{p}=W$,  we get
\begin{align*}
 |(v^{\frac{p}{2}})_{A_{k}} - (v^{\frac{p}{2}})_{A_{j_{0}}}|^{q} & \leq C   \sum_{\ell=j_{0}}^{k-1} \left( \int_{A_{\ell} \cup A_{\ell+1}} \int_{A_{\ell} \cup A_{\ell+1}} \frac{|(v(x))^{\frac{p}{2}}-(v(y))^{\frac{p}{2}}|^{2}}{|x-y|^{N+sp}} \, W (x,y) \, dx \, dy \right)^{\frac{q}{2}}  \\ & \leq C \left( \int_{\mathbb{R}^{N}} \int_{\mathbb{R}^{N}} \frac{|(v(x))^{\frac{p}{2}}-(v(y))^{\frac{p}{2}}|^{2}}{|x-y|^{N+sp}} \, W (x,y) \, dx \, dy \right)^{\frac{q}{2}} \\ & = C \varepsilon_{W}(u)^{\frac{q}{2}}.
\end{align*}
Also, using the fact that for any $x \in A_{W} \cap A_{k}$, we have $v(x) > \varepsilon_{W}(u)^{\gamma_{W}}$, it follows that
\begin{equation*}
    \int_{A_{W} \cap A_{k}}  \, \frac{dx}{|x|^{N}} \leq \frac{1}{\varepsilon_{W}(u)^{\gamma_{W} p}} \int_{A_{W} \cap A_{k}} \frac{|u(x)|^{p}}{|x|^{sp}} \, dx.
\end{equation*}
Therefore, combining the above two estimates, we obtain
\begin{align*}
  |(v^{\frac{p}{2}})_{A_{k}} - (v^{\frac{p}{2}})_{A_{j_{0}}}|^{q}   \int_{A_{W} \cap A_{k}}  \, \frac{dx}{|x|^{N}} \leq C  \varepsilon_{W}(u)^{\frac{q}{2} - \gamma_{W} p} \int_{A_{W} \cap A_{k}} \frac{|u(x)|^{p}}{|x|^{sp}} \, dx. 
\end{align*}
Therefore, combining the above inequality  with the inequality \eqref{ineq123}, the inequality \eqref{ineq999} reduces to 
\begin{align*}
    \int_{ A_{W}} |(v(x))^{\frac{p}{2}} & - \varepsilon_{W}(u)^{\gamma_{W}}|^{q} \, \frac{dx}{|x|^{N}} \\ & \leq \sum_{\substack{k \in \mathbb{Z} \\ A_{W} \cap A_{k} \neq \phi}}  \left( \int_{A_{k}} \int_{A_{k}} \frac{|(v(x))^{\frac{p}{2}}-(v(y))^{\frac{p}{2}}|^{2}}{|x-y|^{N+sp}} \, W (x,y) \, dx \, dy \right)^{\frac{q}{2}}  \\ & \quad + C \varepsilon_{W}(u)^{\frac{q}{2} - \gamma_{W} p}  \sum_{\substack{k \in \mathbb{Z} \\ A_{W} \cap A_{k} \neq \phi}} \int_{A_{W} \cap A_{k}} \frac{|u(x)|^{p}}{|x|^{sp}} \, dx \\ & \leq \left( \int_{\mathbb{R}^{N}} \int_{\mathbb{R}^{N}} \frac{|(v(x))^{\frac{p}{2}}-(v(y))^{\frac{p}{2}}|^{2}}{|x-y|^{N+sp}} \, W (x,y) \, dx \, dy \right)^{\frac{q}{2}} \\ & \quad + C \varepsilon_{W}(u)^{\frac{q}{2} - \gamma_{W} p} \int_{\mathbb{R}^{N}} \frac{|u(x)|^{p}}{|x|^{sp}} \, dx . 
\end{align*}
Using Lemma \ref{Lemma: Hardy potential and Lorentz}, the normalization $\| u \|_{L^{p^{*}_{s},p}(\mathbb{R}^{N})} = 1$, and the definition of $\varepsilon_{W}$, we obtain
 \begin{equation*}
     \int_{ A_{W}} |(v(x))^{\frac{p}{2}}  - \varepsilon_{W}(u)^{\gamma_{W}}|^{q} \, \frac{dx}{|x|^{N}} \leq C\left(  \varepsilon_{W}(u)^{\frac{q}{2}} + \varepsilon_{W}(u)^{\frac{q}{2} - \gamma_{W} p} \right).
 \end{equation*}
 Choose $\gamma_{W} =  \frac{q}{4p}$ and using the fact $\varepsilon_{W}(u) \leq 1$, we obtain
 \begin{equation*}
     \int_{ A_{W}} |(v(x))^{\frac{p}{2}}  - \varepsilon_{W}(u)^{\gamma_{W}}|^{q} \, \frac{dx}{|x|^{N}} \leq C \left(  \varepsilon_{W}(u)^{\frac{q}{2}} + \varepsilon_{W}(u)^{\frac{q}{4}} \right) \leq C \varepsilon_{W}(u)^{\frac{q}{4}} .
 \end{equation*}
Therefore, using $\omega(x)= |x|^{-\frac{N-sp}{p}}$, we obtain
\begin{align*}
    \| u^{\frac{p}{2}} - \varepsilon_{W}(u)^{\gamma_{W}} \omega^{\frac{p}{2}} \|_{L^{q, \infty}(A_{W})} & \leq  \left( \int_{A_{W}} | (u(x))^{\frac{p}{2}} - \varepsilon_{W}(u)^{\gamma} (\omega(x))^{\frac{p}{2}}|^{q} \, dx \right)^{\frac{1}{q}} \\ &  =   \left( \int_{A_{W}} |(v(x))^{\frac{p}{2}} - \varepsilon_{W}(u)^{\gamma_{W}}|^{q} \, \frac{dx}{|x|^{N}} \right)^{\frac{1}{q}}  \leq C \varepsilon_{W}(u)^{\frac{1}{4}}.
\end{align*}
Consequently,
\begin{equation*}
  \inf_{a \geq 0}  \| u^{\frac{p}{2}} -  \omega^{\frac{p}{2}}_{a} \|^{4}_{L^{q, \infty}(A_{W})} \leq  \| u^{\frac{p}{2}} - \varepsilon_{W}(u)^{\gamma_{W}} \omega^{\frac{p}{2}} \|^{4}_{L^{q, \infty}(A_{W})} \leq C \varepsilon_{W}(u) .
\end{equation*}
For any $x \in A^{c}_{W}$, we have $(v(x))^{\frac{p}{2}} \leq \varepsilon_{W}(u)^{\gamma_{W}}$. In particular,
\begin{equation*}
 \omega(x)^{-\frac{p}{2}} |(u(x))^{\frac{p}{2}}- \varepsilon_{W}(u)^{\gamma} (\omega(x))^{\frac{p}{2}}| =  |(v(x))^{\frac{p}{2}} - \varepsilon_{W}(u)^{\gamma_{W}}| \leq 2 \varepsilon_{W}(u)^{\gamma_{W}}, \quad \forall \ x \in A^{c}_{W},
\end{equation*}
and therefore
\begin{equation*}
     |(u(x))^{\frac{p}{2}}- \varepsilon_{W}(u)^{\gamma_{W}} (\omega(x))^{\frac{p}{2}}| \leq 2 \varepsilon_{W}(u)^{\gamma_{W}} (\omega(x))^{\frac{p}{2}}, \quad \forall \ x \in A^{c}_{W} .
\end{equation*}
Using that $\omega^{\frac{p}{2}} \in L^{q, \infty}(\mathbb{R}^{N})$, where $q= \frac{2N}{N-sp}$, the definition of $\gamma_{W}$, and the assumption $\varepsilon_{W}(u) \leq 1$, we obtain
\begin{equation*}
    \|u^{\frac{p}{2}}- \varepsilon_{W}(u)^{\gamma_{W}} \omega^{\frac{p}{2}}\|_{L^{q, \infty}(A^{c}_{W})} \leq C \varepsilon_{W}(u)^{\gamma_{W}} \| \omega^{\frac{p}{2}} \|_{L^{q, \infty}(A^{c}_{W})} \leq C \varepsilon(u)^{\frac{q}{4p}} \leq C \varepsilon(u)^{\frac{1}{4}}. 
\end{equation*}
Consequently,
\begin{equation*}
     \|u^{\frac{p}{2}}- \varepsilon_{W}(u)^{\gamma_{W}} \omega^{\frac{p}{2}}\|^{4}_{L^{q, \infty}(A^{c}_{W})} \leq C \varepsilon_{W}(u).
\end{equation*}
Combining this with the estimate on $A_{W}$, we conclude that
\begin{align*}
    \|u^{\frac{p}{2}}- \varepsilon_{W}(u)^{\gamma_{W}} \omega^{\frac{p}{2}}\|_{L^{q, \infty}(\mathbb{R}^{N})} & \leq C \Big( \|u^{\frac{p}{2}}- \varepsilon_{W}(u)^{\gamma_{W}} \omega^{\frac{p}{2}}\|_{L^{q, \infty}(A_{W})} \\ & \quad \quad  + \|u^{\frac{p}{2}}- \varepsilon_{W}(u)^{\gamma_{W}} \omega^{\frac{p}{2}}\|_{L^{q, \infty}(A^{c}_{W})} \Big) \\ & \leq C \varepsilon_{W}(u)^{\frac{1}{4}}. 
\end{align*}
Therefore, combining the cases $\varepsilon_{W}(u) > 1$ and $\varepsilon_{W}(u) \leq 1$, and under the normalization $\| u \|_{L^{p^{*}_{s},p}(\mathbb{R}^{N})} =1$, we obtain
\begin{equation*}
   \inf_{a \geq 0} \|u^{\frac{p}{2}}-  \omega^{\frac{p}{2}}_{a}\|^{4}_{L^{q, \infty}(\mathbb{R}^{N})} \leq \|u^{\frac{p}{2}}- \varepsilon_{W}(u)^{\gamma_{W}} \omega^{\frac{p}{2}}\|^{4}_{L^{q, \infty}(\mathbb{R}^{N})} \leq C \varepsilon_{W}(u). 
\end{equation*}
To remove the normalization assumption $\| u \|_{L^{p^{*}_{s},p}(\mathbb{R}^{N})} = 1$, we apply the previous estimate to the function $\widetilde{u}$, defined as in the previous subsection. Thus, for $u \geq 0$, we obtain
\begin{align*}
\inf_{a \geq 0} \|u^{\frac{p}{2}}-  \omega^{\frac{p}{2}}_{a}\|^{4}_{L^{q, \infty}(\mathbb{R}^{N})}  \leq C  \varepsilon_{W}(u) \| u \|^{p}_{L^{p^{*}_{s},p}(\mathbb{R}^{N})}.  
\end{align*}
Now, given any $u \in C^{1}_{c}(\mathbb{R}^{N})$, define $u^{\frac{p}{2}}_{+} = \max \{ (u(x))^{\langle \frac{p}{2}\rangle }, 0 \}$ and $u^{\frac{p}{2}}_{-} = \max \{ - (u(x))^{\langle \frac{p}{2}\rangle }, 0 \}$, the positive and negative parts of $u^{\langle \frac{p}{2}\rangle} $ respectively, so that $u^{\langle \frac{p}{2} \rangle} = u^{\frac{p}{2}}_{+} - u^{\frac{p}{2}}_{-}$. Then, following a similar approach to that used in the case $p \geq 2$ and using the definition of $\mathcal{D}_{s,p}$, we obtain
\begin{align*}
     \int_{\mathbb{R}^{N}} \int_{\mathbb{R}^{N}} \frac{|u(x)-u(y)|^{p}}{|x-y|^{N+sp}} \, dx \, dy - \mathcal{C}_{N,s,p} \int_{\mathbb{R}^{N}} \frac{|u(x)|^{p}}{|x|^{sp}} \, dx \geq C(\mathcal{D}_{s,p}(u))^{4} \int_{\mathbb{R}^{N}} & \frac{|u(x)|^{p}}{|x|^{sp}} \, dx,
\end{align*}
for all $u \in C^{1}_{c}(\mathbb{R}^{N})$. Moreover, observing that $\| (u^{\frac{p}{2}})^{\frac{2}{p}} \|_{L^{p^{*}_{s},\infty} (\mathbb{R}^{N})}  = \| u^{\frac{p}{2}} \|^{\frac{2}{p}}_{L^{q,\infty}(\mathbb{R}^{N})}$ for $u \geq 0$, and arguing as in the case $p \geq 2$, we extend the above inequality for all 
$u \in \dot{W}^{s,p}(\mathbb{R}^{N})$ by density. This completes the proof of Theorem \ref{Theorem 2}.

\section{Quantitative form of the classical (local) Hardy Inequality : Proof of Theorem~\ref{Theorem: Improvement of stability exponent C and F result}}\label{Section : 5}
In this section, we establish the quantitative stability for the classical (local) Hardy inequality. Although this result was already obtained in \cite{Cianchi2008}, we provide a proof based entirely on rearrangement-free techniques. Moreover, we improve substantially the stability exponent of the distance function, obtaining a quantitative stability result that is sharper (decay estimate for the distance function) than that of \cite{Cianchi2008} for the local Hardy inequality.

\smallskip

The next lemma establishes a Poincar\'e inequality with a scaling parameter $\lambda > 0$. This parameter plays a crucial role in deriving a scale-invariant Poincar\'e inequality with an appropriate weighted seminorm, which appears as a remainder term in the classical Hardy inequality.

\begin{lemma}\label{Lemma: Local Sobolev inequality}
  Let $\Omega$ be a bounded, connected, open set in $\mathbb{R}^{N}$ with $C^{1}$ boundary, and let $1 \leq p < N$. Define $\Omega_{\lambda} = \{ \lambda x: x \in \Omega \}$. Then there exists a constant $C=C(N,p,\Omega)>0$ such that 
    \begin{equation}
      \left(  \fint_{\Omega_{\lambda}} |u(x)- (u)_{\Omega_{\lambda}}|^{p^{*}} \, dx \right)^{\frac{1}{p^{*}}} \leq C \left( \lambda^{p-N} \int_{\Omega_{\lambda}} |\nabla u(x)|^{p} \, dx   \right)^{\frac{1}{p}},  \quad \forall \, u \in W^{1,p}(\Omega_{\lambda}), 
    \end{equation}
    where $p^{*} =  \frac{Np}{N-p}$.
\end{lemma}
\begin{proof}
   From \cite[Theorem $2$, Section $5.6$]{Evans2010}, for any $u \in W^{1,p}(\Omega)$,
    \begin{equation*}
        \| u \|_{L^{p^{*}}(\Omega)} \leq C \left( \| \nabla u \|^{p}_{L^{p}(\Omega)} +\| u \|^{p}_{L^{p}(\Omega)} \right)^{\frac{1}{p}}.
    \end{equation*}
    Applying this inequality to $u-(u)_{\Omega}$ and using the Poincaré inequality (since $\Omega$ is connected), we obtain
    \begin{equation*}
        \left(  \fint_{\Omega} |u(x)- (u)_{\Omega}|^{p^{*}} \, dx \right)^{\frac{1}{p^{*}}} \leq C \left( \int_{\Omega} |\nabla u(x)|^{p} \, dx   \right)^{\frac{1}{p}},
    \end{equation*}
   where $C=C(N,p,\Omega)>0$.  Now apply this inequality to $u(\lambda x)$. Using the identity $\fint_{\Omega} u(\lambda x) \, dx = \fint_{\Omega_{\lambda}} u(x) \, dx $, we obtain
    \begin{equation*}
         \left(  \fint_{\Omega} |u(\lambda x)- (u)_{\Omega_{\lambda}}|^{p^{*}} \, dx \right)^{\frac{1}{p^{*}}} \leq C  \left(\lambda^{p} \int_{\Omega}  |\nabla u( \lambda x)|^{p} \, dx   \right)^{\frac{1}{p}}.
    \end{equation*}
   Finally, using the change of variables $y = \lambda x$, we get
    \begin{equation*}
 \left(  \fint_{\Omega_{\lambda}} |u(x)- (u)_{\Omega_{\lambda}}|^{p^{*}} \, dx \right)^{\frac{1}{p^{*}}} \leq C \left( \lambda^{p-N} \int_{\Omega_{\lambda}} |\nabla u(x)|^{p} \, dx   \right)^{\frac{1}{p}}.
    \end{equation*}
    This proves the lemma.
\end{proof}

The next lemma provides a quantitative convexity estimate for the mapping $X \mapsto |X|^{p}$ in the range $1 < p < 2$. This estimate plays a crucial role in deriving the remainder term in the classical Hardy inequality for $1 < p < 2$.

\begin{lemma}\label{Lemma: Estimate 1<p<2}
Let $1<p<2$, and let $X,Y \in \mathbb{R}^{N} \setminus \{ 0\}$. Then
\begin{equation}
   |X+Y|^{p} \ge |X|^{p} + p|X|^{p-2} X\cdot Y 
+ \frac{p(p-1)}{2} (|X|+|Y|)^{p-2} |Y|^{2}. 
\end{equation}
\end{lemma}
\begin{proof}
Fix $\epsilon>0$ and define the smooth function $$F_{\epsilon}(z)=(|z|^2+\epsilon^2)^{\frac{p}{2}},\,\, z \in \mathbb{R}^N.$$
For any $z,h \in \mathbb{R}^N$, $$DF_{\epsilon}(z)[h]=p(|z|^2+\epsilon^2)^{\frac{p-2}{2}}z\cdot h,$$$$D^2F_{\epsilon}(z)[h,h]=p(|z|^2+\epsilon^2)^{\frac{p-2}{2}}|h|^2+p(p-2)(|z|^2+\epsilon^2)^{\frac{p-4}{2}}(z\cdot h)^2.$$Since $$(z\cdot h)^2\leq |z|^2|h|^2\leq (|z|^2+\epsilon^2)|h|^2$$ and $p-2<0$ we get $$p(p-2)(|z|^2+\epsilon^2)^{\frac{p-4}{2}}(z\cdot h)^2\geq p(p-2)(|z|^2+\epsilon^2)^{\frac{p-2}{2}}|h|^2.$$Therefore $$D^2F_{\epsilon}(z)[h,h]\geq p(p-1)(|z|^2+\epsilon^2)^{\frac{p-2}{2}}|h|^2.$$ Define $\phi_{\epsilon}(t)=F_{\epsilon}(X+tY),\,\, t\in [0,1].$ Then $\phi_{\epsilon}\in C^2[0,1].$ The one-dimensional Taylor formula gives $$\phi_{\epsilon}(1)= \phi_{\epsilon}(0)+\phi_{\epsilon}'(0)+\int_0^1(1-t)\phi_{\epsilon}''(t)\,\mathrm{d}t.$$We note that $$\phi_{\epsilon}''(t)=D^2F_{\epsilon}(X+tY)[Y,Y]\geq p(p-1)(|X+tY|^2+\epsilon^2)^{\frac{p-2}{2}}|Y|^2.$$ As $p-2<0$ so $s \mapsto s^{\frac{p-2}{2}}$ is decreasing on $(0,\infty)$ and by the triangle inequality, noting $0\leq t \leq 1$, $$|X+tY|\leq |X|+t|Y|\leq |X|+|Y|.$$Hence $(|X+tY|^2+\epsilon^2)^{\frac{p-2}{2}}\geq ((|X|+|Y|)^2+\epsilon^2)^{\frac{p-2}{2}}$ which implies $$\phi_{\epsilon}''(t)\geq p(p-1)((|X|+|Y|)^2+\epsilon^2)^{\frac{p-2}{2}}|Y|^2.$$
It is direct that $$\phi_{\epsilon}(0)=(|X|^2+\epsilon^2)^{\frac{p}{2}},$$$$\phi_{\epsilon}'(0)=p(|X|^2+\epsilon^2)^{\frac{p-2}{2}}X\cdot Y.$$Plugging everything in the Taylor expansion we get $$(|X+Y|^2+\epsilon^2)^{\frac{p}{2}}\geq (|X|^2+\epsilon^2)^{\frac{p}{2}}+p(|X|^2+\epsilon^2)^{\frac{p-2}{2}}X\cdot Y+\frac{p(p-1)}{2}((|X|+|Y|)^2+\epsilon^2)^{\frac{p-2}{2}}|Y|^2.$$ Letting $\epsilon \to 0$ we conclude our result.
\end{proof}

\medskip

We are now in a position to establish the remainder term in the classical Hardy inequality for the case $1<p<2$. Since $\omega_{0}(x) = |x|^{-\frac{N-p}{p}}$ is a virtual extremizer of \eqref{classical-hardy}, it satisfies the associated Euler--Lagrange equation
\begin{equation}\label{local Euler--Lagrange equation}
    - \operatorname{div} \big( |\nabla \omega_{0}|^{p-2} \nabla \omega_{0} \big)
    = \left( \frac{N-p}{p} \right)^{p} \frac{\omega_{0}^{p-1}}{|x|^{p}}.
\end{equation}
Recall that the classical Hardy deficit is given by
\begin{equation*}
    \delta_{p}(u) = \int_{\mathbb{R}^{N}} |\nabla u(x)|^{p} \, dx - \left( \frac{N-p}{p} \right)^{p} \int_{\mathbb{R}^{N}} \frac{|u(x)|^{p}}{|x|^{p}} \, dx .
\end{equation*}

\begin{lemma}\label{Lemma: Local Hardy Remainder 1<p<2}
   Let $N \geq 2$ and $1 < p < 2$. Suppose that $u \in C_c^1(\mathbb{R}^N)$ satisfies $\delta_{p}(u) \leq 1$ and $\|u\|_{L^{p^*,p}(\mathbb{R}^N)} = 1$. Then there exists a constant $C = C(N,p) > 0$ such that
    \begin{equation}
          \int_{\mathbb{R}^{N}} |\nabla u(x)|^{p} \, dx - \left( \frac{N-p}{p} \right)^{p} \int_{\mathbb{R}^{N}} \frac{|u(x)|^{p}}{|x|^{p}} \, dx   \geq C  \left(\int_{\mathbb{R}^{N}} \frac{|\nabla v(x)|^{p}}{|x|^{N-p}} \, dx \right)^{\frac{2}{p}},
    \end{equation}
    where $v(x) =u(x) \omega^{-1}_{0}(x)= u(x)|x|^{\frac{N-p}{p}}$.
\end{lemma}
\begin{proof}
   Let $u \in C^{1}_{c}(\mathbb{R}^{N})$ be a nonzero function.  Given that $\delta_{p}(u) \leq 1$ and $\|u\|_{L^{p^*,p}(\mathbb{R}^N)} = 1$, applying Lemma \ref{Lemma: Hardy potential and Lorentz} with $s = 1$, we obtain
\begin{equation}\label{ineq7654}
    \int_{\mathbb{R}^{N}} |\nabla u(x)|^{p} \, dx  \leq 1 +  \left( \frac{N-p}{p} \right)^{p} \int_{\mathbb{R}^{N}} \frac{|u(x)|^{p}}{|x|^{p}} \, dx \leq 1 + \left( \frac{N-p}{p} \right)^{p} \left( \frac{|\mathbb{S}^{N-1}|}{N} \right)^{\frac{p}{N}}.
\end{equation}
Since $u = v \omega_{0}$, we have
\begin{equation*}
    \nabla u = v \nabla \omega_{0} + \omega_{0} \nabla v.
\end{equation*}
Applying Lemma \ref{Lemma: Estimate 1<p<2} with $X = v \nabla \omega_{0}$, and $Y = \omega_{0} \nabla v,$ we obtain
\begin{align*}
    \int_{\mathbb{R}^{N}} |\nabla u(x)|^{p} \, dx &  =  \int_{\mathbb{R}^{N}} |X+Y|^{p} \, dx  \\ & \geq \left( \frac{N-p}{p} \right)^{p} \int_{\mathbb{R}^{N}} \frac{|u(x)|^{p}}{|x|^{p}} \, dx + p \int_{\mathbb{R}^{N}} |\nabla \omega_{0}|^{p-2} \omega_{0} v |v|^{p-2} \nabla v \cdot \nabla \omega_{0} \, dx \\ & \quad + \frac{p(p-1)}{2} \int_{\mathbb{R}^{N}} \frac{|\nabla v(x)|^{2}}{|x|^{\left(\frac{N-p}{p}\right)2}} \left( |v(x) \nabla \omega_{0}(x)| + |\omega_{0}(x) \nabla v(x)| \right)^{p-2} \, dx .
\end{align*}
The integrand of the second term on the right-hand side can be written as $p^{-1}  \omega_{0} |\nabla \omega_{0}|^{p-2} \nabla \omega_{0} \cdot \nabla (|v|^{p})$.
An integration by parts and using \eqref{local Euler--Lagrange equation}  yields that this term vanishes. Therefore, we obtain
\begin{align*}
    \int_{\mathbb{R}^{N}} &\frac{|\nabla v(x)|^{2}}{|x|^{\left(\frac{N-p}{p} \right) 2}}  \left( |v(x) \nabla \omega_{0}(x)| + |\omega_{0}(x) \nabla v(x)| \right)^{p-2} \, dx \\ & \leq \frac{2}{p(p-1)} \left(   \int_{\mathbb{R}^{N}} |\nabla u(x)|^{p} \, dx - \left( \frac{N-p}{p} \right)^{p} \int_{\mathbb{R}^{N}} \frac{|u(x)|^{p}}{|x|^{p}} \, dx \right)  =\left( \frac{2}{p(p-1)} \right) \delta_{p}(u).
\end{align*}
Applying Hölder's inequality with exponents $\frac{2}{p}$ and $\frac{2}{2-p}$ and using the above estimate, we obtain
\begin{align*}
    \int_{\mathbb{R}^{N}} \frac{|\nabla v(x)|^{p}}{|x|^{N-p}} \, dx \leq C (\delta_{p}(u))^{\frac{p}{2}} \left( \int_{\mathbb{R}^{N}} \left(  |v \nabla \omega_{0}| + |\omega_{0} \nabla v|  \right)^{p} \, dx \right)^{\frac{2-p}{2}},
\end{align*}
where $C=C(p)>0$. Using $u= v \omega_{0}$, we have $|\omega_{0} \nabla v| \leq |\omega_{0} \nabla v + v \nabla \omega_{0}| + |v \nabla \omega_{0}| = |\nabla u|+ |v \nabla \omega_{0}|$. Therefore, the above inequality reduces to 
\begin{equation*}
    \int_{\mathbb{R}^{N}} \frac{|\nabla v(x)|^{p}}{|x|^{N-p}} \, dx \leq C (\delta_{p}(u))^{\frac{p}{2}} \left( \int_{\mathbb{R}^{N}} \left( 2 |v \nabla \omega_{0}| + |\nabla u|  \right)^{p} \, dx \right)^{\frac{2-p}{2}}. 
\end{equation*}
Also, using Lemma \ref{Lemma: Hardy potential and Lorentz} with $s=1$ and the assumption $\|u\|_{L^{p^{*},p}(\mathbb{R}^{N})} =1$, we have
\begin{equation*}
    \int_{\mathbb{R}^{N}} |v \nabla \omega_{0}|^{p} \, dx  =  \left( \frac{N-p}{p} \right)^{p} \int_{\mathbb{R}^{N}} \frac{|u(x)|^{p}}{|x|^{p}} \, dx \leq \left( \frac{N-p}{p} \right)^{p} \left( \frac{|\mathbb{S}^{N-1}|}{N} \right)^{\frac{p}{N}}.
\end{equation*}
Therefore, using the above estimate and the inequality \eqref{ineq7654}, we arrive at 
\begin{equation*}
    \int_{\mathbb{R}^{N}} \frac{|\nabla v(x)|^{p}}{|x|^{N-p}} \, dx \leq C (\delta_{p}(u))^{\frac{p}{2}}. 
\end{equation*}
This finishes the proof of Lemma.
\end{proof}

\begin{rem}\label{Remark: local Hardy remainder for p>2}
    The classical Hardy inequality with a remainder term for $p \geq 2$ and $p < N$ was established by Frank and Seiringer using the ground state representation. Their result (see \cite[Remark 2.5]{Frank2008}) states that
\begin{equation}\label{Hardy inequality with remainder: p>2}
    \int_{\mathbb{R}^{N}} |\nabla u(x)|^{p} \, dx -  \left( \frac{N-p}{p} \right)^{p} \int_{\mathbb{R}^{N}} \frac{|u(x)|^{p}}{|x|^{p}} \, dx \geq c_{p} \int_{\mathbb{R}^{N}}   \frac{|\nabla v(x)|^{p}}{|x|^{N-p}} \, dx,
\end{equation}
where $v(x) =u(x) \omega^{-1}_{0}(x)= u(x)|x|^{\frac{N-p}{p}}$ and $c_{p}$ is given in \eqref{Defn: c_p}.
\end{rem}

We now prove Theorem \ref{Theorem: Improvement of stability exponent C and F result}. The remainder term in the classical Hardy inequality, obtained in Lemma \ref{Lemma: Local Hardy Remainder 1<p<2} for $1<p<2$ and in Remark \ref{Remark: local Hardy remainder for p>2} for $p>2$, together with the scale-invariant Poincar\'e inequality, is fundamental in controlling the weighted seminorm that arises as the remainder term.

\begin{proof}[\textbf{Proof of Theorem \ref{Theorem: Improvement of stability exponent C and F result}}] It is sufficient to prove the result for $u \in C^{1}_{c}(\mathbb{R}^{N})$. Accordingly, let $u \in C^{1}_{c}(\mathbb{R}^N)$ with $u \ge 0$. Let $1<p<N$, and let $\alpha := \max \{4, 2p \} $. We also assume that 
\begin{equation*}
    \| u \|_{L^{p^{*},p}(\mathbb{R}^{N})} = 1.
\end{equation*}
    Recall that the Hardy deficit is defined by
    \begin{equation*}
     \delta_{p}(u) =   \int_{\mathbb{R}^{N}} |\nabla u(x)|^{p} \, dx - \left( \frac{N-p}{p} \right)^{p} \int_{\mathbb{R}^{N}} \frac{|u(x)|^{p}}{|x|^{p}} \, dx  .
    \end{equation*}
    First assume that $\delta_{p}(u)>1$. Using the normalization $\| u\|_{L^{p^{*},p}(\mathbb{R}^{N})} = 1$, we obtain
\begin{equation*}
   \delta_{p}(u)^{\frac{1}{\alpha}} > 1 =  \| u\|_{L^{p^{*},p}(\mathbb{R}^{N})} \geq C \| u\|_{L^{p^{*},\infty}(\mathbb{R}^{N})} \geq C \inf_{a \geq 0} \| u- a |x|^{-\frac{N-p}{p}}\|_{L^{p^{*},\infty}(\mathbb{R}^{N})}.
\end{equation*}
Therefore, it is sufficient to consider the case $\delta_{p}(u) \leq 1$. Let $v(x) = u(x) |x|^{\frac{N-p}{p}}$. For $\eta>0$, to be choosen later, and define the set 
\begin{equation*}
    \mathcal{G}= \{ x \in \mathbb{R}^{N} : v(x)> \delta_{p}(u)^{\eta} \}.
\end{equation*}
Since $u \in C^{1}_{c}(\mathbb{R}^{N})$, it follows that the set $\mathcal{G}$ is open and bounded. For each $k \in \mathbb{Z}$, consider the set $A_{k}$, defined in \eqref{Defn: A_{k}}. Since $\lim_{|x| \to 0} v(x) = \lim_{|x|\to 0} u(x)|x|^{\frac{N-p}{p}}=0$, consider the smallest $n_{0}$ such that 
\begin{equation*}
    v(x) \leq \delta_{p}(u)^{\eta}, \quad \forall \, x \in A_{n_{0}} \quad \text{and} \quad  \mathcal{G} \cap A_{n_{0}+1} \neq  \phi \quad \Rightarrow \quad (v)_{A_{n_{0}}} \leq \delta_{p}(u)^{\eta}.
\end{equation*}
Therefore, using $0< v(x) - \delta_{p}(u)^{\eta} \leq v(x) - (v)_{A_{n_{0}}}$, for all $x \in \mathcal{G}$, we obtain
\begin{align}\label{ineq45632}
    \int_{\mathcal{G}} |v(x) - \delta_{p}(u)^{\eta}|^{p^{*}} \, \frac{dx}{|x|^{N}} & \leq \int_{\mathcal{G}} |v(x) - (v)_{A_{n_{0}}}|^{p^{*}} \, \frac{dx}{|x|^{N}} \nonumber \\ 
    & = \sum_{\substack{k \in \mathbb{Z} \\ \mathcal{G} \cap A_{k} \neq \phi}} \int_{\mathcal{G} \cap A_{k}} |v(x) - (v)_{A_{n_{0}}}|^{p^{*}} \, \frac{dx}{|x|^{N}}. 
\end{align}
For each $k \in \mathbb{Z}$ such that $\mathcal{G} \cap A_{k} \neq \phi$, we estimate
\begin{align}\label{ineq9877}
  \int_{\mathcal{G} \cap A_{k}} |v(x) - (v)_{A_{n_{0}}}|^{p^{*}} \, \frac{dx}{|x|^{N}} & \leq 2^{p^{*}-1} \int_{\mathcal{G} \cap A_{k}} |v(x) - (v)_{A_{k}}|^{p^{*}} \, \frac{dx}{|x|^{N}} \nonumber \\ & \quad + 2^{p^{*}-1} |(v)_{A_{k}} - (v)_{A_{n_{0}}}|^{p^{*}} \int_{\mathcal{G} \cap A_{k}}  \, \frac{dx}{|x|^{N}} \nonumber \\ & \leq 2^{p^{*}-1} \int_{A_{k}} |v(x) - (v)_{A_{k}}|^{p^{*}} \, \frac{dx}{|x|^{N}} \nonumber \\ & \quad + 2^{p^{*}-1} |(v)_{A_{k}} - (v)_{A_{n_{0}}}|^{p^{*}} \int_{\mathcal{G} \cap A_{k}}  \, \frac{dx}{|x|^{N}}  .
\end{align}
Applying Lemma \ref{Lemma: Local Sobolev inequality} with $\Omega = \{ x \in \mathbb{R}^{N} : 1<|x|<2 \}$ and $\lambda = 2^{k}$, we obtain
\begin{align*}
    \int_{A_{k}} |v(x) - (v)_{A_{k}}|^{p^{*}} \, \frac{dx}{|x|^{N}}  \leq C \fint_{A_{k}} |v(x) - (v)_{A_{k}}|^{p^{*}} \, dx  \leq C \left( 2^{k(p-N)} \int_{A_{k}} |\nabla v|^{p} \, dx  \right)^{\frac{p^{*}}{p}} .
\end{align*}
Here the constant $C:=C(N,p,\Omega)$ does not depend on $k.$ For all $x \in A_{k}$, we have $|x| \leq 2^{k+1}$, and hence  $|x|^{(N-p)} \leq 2^{(k+1)(N-p)}$. Therefore, we obtain
\begin{equation*}
     \int_{A_{k}} |v(x) - (v)_{A_{k}}|^{p^{*}} \, \frac{dx}{|x|^{N}} \leq C \left( \int_{A_{k}}  \, \frac{|\nabla v|^{p}}{|x|^{N-p}}  \, dx \right)^{\frac{p^{*}}{p}},
\end{equation*}
where $C=C(N,p)>0$. Therefore, the estimate \eqref{ineq9877}, reduces to 
\begin{align}\label{ineq87651}
     \int_{\mathcal{G} \cap A_{k}} |v(x) - (v)_{A_{n_{0}}}|^{p^{*}} \, \frac{dx}{|x|^{N}} & \leq C \left( \int_{A_{k}}   \frac{|\nabla v|^{p}}{|x|^{N-p}}  \, dx\right)^{\frac{p^{*}}{p}} \nonumber \\ & \quad + 2^{p^{*}-1} |(v)_{A_{k}} - (v)_{A_{n_{0}}}|^{p^{*}} \int_{\mathcal{G} \cap A_{k}}  \, \frac{dx}{|x|^{N}}.
\end{align}
Now, applying Lemma \ref{Lemma : on two disjoint set} with $E = A_{\ell}$ and $F = A_{\ell+1}$, for $n_{0} \leq \ell \leq k-1$, we obtain
\begin{align*}
  |(v)_{A_{k}} - (v)_{A_{n_{0}}}|^{p^{*}} & \leq 2^{p^{*}-1} \sum_{\ell = n_{0}}^{k-1}  |(v)_{A_{\ell}} - (v)_{A_{\ell+1}}|^{p^{*}} \\ & \leq C \sum_{\ell = n_{0}}^{k-1} \fint_{A_{\ell} \cup A_{\ell+1}}  |v(x) - (v)_{A_{\ell} \cup A_{\ell+1}}|^{p^{*}} \, dx, 
\end{align*}
where $C=C(N,p)>0$. Next, we apply Lemma \ref{Lemma: Local Sobolev inequality} with $\Omega = \{x \in \mathbb{R}^{N} : 1 < |x| < 4\}$ and scaling parameter $\lambda = 2^{\ell}$ and using $|x|^{(N-p)} \leq 2^{(\ell+2)(N-p)}$ for all $x\in A_{\ell} \cup A_{\ell+1}$, we obtain
\begin{align*}
    |(v)_{A_{k}} - (v)_{A_{n_{0}}}|^{p^{*}} & \leq C \sum_{\ell = n_{0}}^{k-1} \left( 2^{\ell (p-N)}  \int_{A_{\ell \cup A_{\ell+1}}} |\nabla v(x)|^{p} \, dx \right)^{\frac{p^{*}}{p}} \\ & \leq C \sum_{\ell = n_{0}}^{k-1} \left(  \int_{A_{\ell} \cup A_{\ell+1}}   \frac{|\nabla v(x)|^{p}}{|x|^{N-p}} \, dx \right)^{\frac{p^{*}}{p}}  \leq C \left( \int_{\mathbb{R}^{N}}  \, \frac{|\nabla v(x)|^{p}}{|x|^{N-p}} \, dx \right)^{\frac{p^{*}}{p}},
\end{align*}
where $C = C(N,p)>0$.
Also, using the fact that for any $x \in \mathcal{G} \cap A_{k}$, we have $v(x) > \delta_{p}(u)^\eta$, it follows that
\begin{equation*}
    \int_{\mathcal{G} \cap A_{k}}  \, \frac{dx}{|x|^{N}} \leq \frac{1}{\delta_{p}(u)^{\eta p}} \int_{\mathcal{G} \cap A_{k}} \frac{|u(x)|^{p}}{|x|^{p}} \, dx.
\end{equation*}
Therefore, combining the above two estimates and invoking the classical Hardy inequality with remainder—namely, for $p \geq 2$ in \eqref{Hardy inequality with remainder: p>2} and for $1 < p < 2$ in Lemma \ref{Lemma: Local Hardy Remainder 1<p<2}—under the assumptions $\delta_{p}(u) \leq 1$ and $\|u\|_{L^{p^{*},p}(\mathbb{R}^{N})} = 1$, we obtain
\begin{align*}
    |(v)_{A_{k}}  -  (v)_{A_{n_{0}}}|^{p^{*}}  \int_{\mathcal{G} \cap A_{k}}  \, \frac{dx}{|x|^{N}} & \leq C \delta_{p}(u)^{\frac{p^{*}}{p} \beta}  \, \frac{1}{\delta^{\eta p}} \int_{\mathcal{G} \cap A_{k}} \frac{|u(x)|^{p}}{|x|^{p}} \, dx \\ & = C \delta_{p}(u)^{\frac{p^{*}}{p} \beta - \eta p} \int_{\mathcal{G} \cap A_{k}} \frac{|u(x)|^{p}}{|x|^{p}} \, dx,
\end{align*}
where $\beta:=\min \{ 1, \frac{p}{2} \}$. Substituting the above estimate into \eqref{ineq87651}, we obtain
\begin{equation*}
    \int_{\mathcal{G} \cap A_{k}} |v(x) - (v)_{A_{n_{0}}}|^{p^{*}} \, \frac{dx}{|x|^{N}}  \leq C \left( \int_{A_{k}}  \frac{|\nabla v|^{p}}{|x|^{N-p}}  \, dx \right)^{\frac{p^{*}}{p}}+ C \delta_{p}(u)^{\frac{p^{*}}{p} \beta - \eta p} \int_{\mathcal{G} \cap A_{k}} \frac{|u(x)|^{p}}{|x|^{p}} \, dx.
\end{equation*}
Therefore, combining \eqref{ineq45632} with the estimates above, and invoking the classical Hardy inequality with remainder—namely, for $p \geq 2$ in \eqref{Hardy inequality with remainder: p>2} and for $1 < p < 2$ in Lemma \ref{Lemma: Local Hardy Remainder 1<p<2}—under the assumptions $\delta_{p}(u) \leq 1$ and the normalization $\| u \|_{L^{p^{*},p}(\mathbb{R}^{N})} = 1$, together with Lemma \ref{Lemma: Hardy potential and Lorentz} with $s = 1$, we obtain
\begin{align*}
    \int_{\mathcal{G}} |v(x) - \delta_{p}(u)^{\eta}|^{p^{*}} \, \frac{dx}{|x|^{N}} & \leq \sum_{\substack{k \in \mathbb{Z} \\ \mathcal{G} \cap A_{k} \neq \phi}} \int_{\mathcal{G} \cap A_{k}} |v(x) - (v)_{A_{n_{0}}}|^{p^{*}} \, \frac{dx}{|x|^{N}} \\ & \leq C \delta_{p}(u)^{\frac{p^{*}}{p}\beta} + C \delta_{p}(u)^{\frac{p^{*}}{p}\beta - \eta p} \int_{\mathbb{R}^{N}} \frac{|u(x)|^{p}}{|x|^{p}} \, dx \\ & \leq C \left(\delta_{p}(u)^{\frac{p^{*}}{p}\beta} + \delta_{p}(u)^{\frac{p^{*}}{p} \beta - \eta p} \right),
\end{align*}
where $C=C(N,p)>0$. Choose $\eta = \frac{p^{*}}{2 p^{2}} \beta$.  Then, using the assumption $\delta_{p}(u) \leq 1$, we obtain
\begin{align*}
    \int_{\mathcal{G}} |v(x) - \delta_{p}(u)^{\eta}|^{p^{*}} \, \frac{dx}{|x|^{N}} & \leq C \left( \delta_{p}(u)^{\frac{p^{*}}{p} \beta} + \delta_{p}(u)^{\frac{p^{*}}{2p} \beta} \right) \\ & \leq C \left( \delta_{p}(u)^{\frac{p^{*}}{2p}\beta} + \delta_{p}(u)^{\frac{p^{*}}{2p}\beta }  \right) \leq C \delta_{p}(u)^{\frac{p^{*}}{2p}\beta} .
\end{align*}
Therefore, 
\begin{align*}
    \| u - \delta_{p}(u)^{\eta} |x|^{-\frac{N-p}{p}} \|_{L^{p^{*}, \infty}(\mathcal{G})} & \leq  \left( \int_{\mathcal{G}} | u(x) - \delta_{p}(u)^{\eta} |x|^{-\frac{N-p}{p}} |^{p^{*}} \, dx \right)^{\frac{1}{p^{*}}} \\ &  =   \left( \int_{\mathcal{G}} |v(x) - \delta_{p}(u)^{\eta}|^{p^{*}} \, \frac{dx}{|x|^{N}} \right)^{\frac{1}{p^{*}}}  \leq C \delta_{p}(u)^{\frac{\beta}{2p}}.
\end{align*}
Consequently,
\begin{equation*}
    \| u - \delta^{\eta}_{p}(u) |x|^{-\frac{N-p}{p}}\|^{\alpha}_{L^{p^{*}, \infty}(\mathcal{G})} \leq C \delta_{p}(u) ,
\end{equation*}
where $\alpha= \max \{4, 2p \}$. For any $x \in \mathcal{G}^{c}$, we have $v(x) \leq \delta_{p}(u)^{\eta}$. In particular,
\begin{equation*}
 |x|^{\frac{N-p}{p}} |u(x)- \delta_{p}(u)^{\eta} |x|^{-\frac{N-p}{p}}| =  |v(x) - \delta_{p}(u)^{\eta}| \leq 2 \delta_{p}(u)^{\eta}, \quad \forall \ x \in \mathcal{G}^{c},
\end{equation*}
and therefore
\begin{equation*}
     |u(x)- \delta_{p}(u)^{\eta} |x|^{-\frac{N-p}{p}}| \leq 2 \delta_{p}(u)^{\eta} |x|^{-\frac{N-p}{p}}, \quad \forall \ x \in \mathcal{G}^{c} .
\end{equation*}
Using that $|x|^{-\frac{N-p}{p}} \in L^{p^{*}, \infty}(\mathbb{R}^{N})$, the definition of $\eta$, and the assumption $\delta_{p}(u) \leq 1$, we obtain
\begin{equation*}
    \|u- \delta_{p}(u)^{\eta} |x|^{-\frac{N-p}{p}}\|_{L^{p^{*}, \infty}(\mathcal{G}^{c})} \leq C \delta_{p}(u)^{\eta} \| |x|^{-\frac{N-p}{p}} \|_{L^{p^{*}, \infty}(\mathcal{G}^{c})} \leq C \delta_{p}(u)^{\frac{p^{*}}{2p^{2}}\beta} \leq C \delta_{p}(u)^{\frac{\beta}{2p}}. 
\end{equation*}
Consequently,
\begin{equation*}
     \|u- \delta_{p}(u)^{\eta} |x|^{-\frac{N-p}{p}}\|^{\alpha}_{L^{p^{*}, \infty}(\mathcal{G}^{c})} \leq C \delta_{p}(u).
\end{equation*}
Combining this with the estimate on $\mathcal{G}$, we conclude that
\begin{align*}
    \|u- \delta_{p}(u)^{\eta} |x|^{-\frac{N-p}{p}}\|_{L^{p^{*}, \infty}(\mathbb{R}^{N})} & \leq C \Big( \|u- \delta_{p}(u)^{\eta} |x|^{-\frac{N-p}{p}}\|_{L^{p^{*}, \infty}(\mathcal{G})} \\ & \quad \quad  + \|u- \delta_{p}(u)^{\eta} |x|^{-\frac{N-p}{p}}\|_{L^{p^{*}, \infty}(\mathcal{G}^{c})} \Big) \\ & \leq C \delta_{p}(u)^{\frac{\beta}{2p}}. 
\end{align*}
Therefore, combining the cases $\delta_{p}(u) > 1$ and $\delta_{p}(u) \leq 1$, and under the normalization $\| u \|_{L^{p^{*},p}(\mathbb{R}^{N})} =1$, we obtain
\begin{equation*}
   \inf_{a \geq 0} \|u-  a|x|^{-\frac{N-p}{p}}\|^{\alpha}_{L^{p^{*}, \infty}(\mathbb{R}^{N})} \leq \|u- \delta_{p}(u)^{\eta} |x|^{-\frac{N-p}{p}}\|^{\alpha}_{L^{p^{*}, \infty}(\mathbb{R}^{N})} \leq C \delta_{p}(u). 
\end{equation*}
To remove the normalization $\| u \|_{L^{p^{*},p}(\mathbb{R}^{N})} =1$ and the assumption $u \geq 0$, following a similar approach as done in the proof of Theorem \ref{Theorem 1}, and the distance function $d_{p}$, we obtain
\begin{equation*}
    \int_{\mathbb{R}^{N}} |\nabla u(x)|^{p} \,dx - \left(\frac{N-p}{p}\right)^{p} \int_{\mathbb{R}^{N}} \frac{|u(x)|^p}{|x|^p} \, dx \geq C ( d_{p}(u))^{\alpha} \int_{\mathbb{R}^{N}} \frac{|u(x)|^p}{|x|^p} \, dx,
\end{equation*}
where $\alpha= \max \{4,2p \}$. This proves Theorem \ref{Theorem: Improvement of stability exponent C and F result}.
\end{proof}


\section{Equivalence between local and non-local Hardy inequality: Proof of Theorem \ref{Theorem: Equivalence bw local and nonlocal via Emden Trans} and Theorem \ref{qspf}}\label{Section : 6}

This section is devoted to the study of the equivalence between local and nonlocal Hardy inequalities via a transformation $T$, defined in \eqref{Defn: The operator T}, constructed through the Emden--Fowler transformation. We begin with some preliminaries and, more importantly, present several key properties (see Theorem \ref{properties of T}) of $T$ which are of independent interest.

\medskip 

Let $N \geq 3$ and $s \in (0,1)$. Let $\mathcal{C} = \mathbb{R} \times \mathbb{S}^{N-1}$ be the infinite cylinder with coordinates $(t,w)$, where $t \in \mathbb{R}$ and $w \in \mathbb{S}^{N-1}$. Let $\mu_{\ell} = \ell (\ell+N-2)$, where $ \ell \in \mathbb{N}\cup \{ 0\}$ be the eigenvalues of the Laplacian-Beltrami operator $- \Delta_{\mathbb{S}^{N-1}}$. 

\medskip

The operator $T$ fundamentally acts as a spectral transport map, and its properties are best understood by analyzing its core component, the pseudo-differential multiplier $M_s$, on the Emden-Fowler cylinder $\mathcal{C} = \mathbb{R} \times \mathbb{S}^{N-1}$.

We first define the weighted Euclidean Lebesgue spaces. For $\alpha \in \mathbb{R}$, let $L^2(\mathbb{R}^N; |x|^{-\alpha}\,dx)$ be the space of measurable functions equipped with the norm
$$\| u \|_{L^2(\mathbb{R}^N; |x|^{-\alpha} \, dx)}^2 := \int_{\mathbb{R}^N} |u(x)|^2 |x|^{-\alpha} \, dx.$$ Let $\Phi_s: L^2(\mathbb{R}^N; |x|^{-2s}\, dx) \to L^2(\mathcal{C})$ be the transformation defined by
$$(\Phi_s u)(t, w) := e^{\frac{N-2s}{2}t} u(e^t w),$$where $x = e^t w$, $t \in \mathbb{R}$, and $w \in \mathbb{S}^{N-1}$.

\begin{theorem}\label{properties of T}
Let $0 < s < 1.$ Then the transformation $T$, defined in \eqref{Defn: The operator T}, satisfies the following properties:
\begin{itemize}
    \item[(i)] \textbf{Unitary equivalence of weighted spaces.} 
    The map $\Phi_s$ defines an isometric isomorphism from $L^2(\mathbb{R}^N; |x|^{-2s}\,dx)$ onto $L^2(\mathcal{C})$. Moreover, the operator $T$ admits the factorization
    \[
    T = \Phi_1^{-1} \, M_s \, \Phi_s,
    \]
    where $\Phi_\beta$ denotes the conformal transform with weight parameter $\beta$.

    \item[(ii)] \textbf{Symbol behaviour and smoothing order.} 
    The multiplier
    \[
    \bar{m}(\xi,\ell) = \sqrt{\frac{P_s(\xi,\ell)}{\xi^2 + \mu_\ell}}
    \]
    satisfies $\bar{m} \in L^\infty(\mathbb{R} \times (\mathbb{N}\cup\{0\}))$ and exhibits the asymptotic behaviour
    \[
    \bar{m}(\xi,\ell) \sim (\xi^2 + \ell^2)^{\frac{s-1}{2}} 
    \quad \text{as} \quad \xi^2 + \ell^2 \to \infty.
    \]
    In particular, $M_s$ is a pseudo-differential operator of strictly negative order $s-1$.

    \item[(iii)] \textbf{Boundedness and injectivity.} 
    The operator $T$ extends uniquely to a bounded and injective linear map
    \[
    T : L^2(\mathbb{R}^N; |x|^{-2s}\,dx) \longrightarrow L^2(\mathbb{R}^N; |x|^{-2}\,dx).
    \]

    \item[(iv)] \textbf{Scale invariance.} 
    Let $D_\lambda$ denote the dilation operator $(D_\lambda u)(x) = u(\lambda x)$ for $\lambda > 0$. Then $T$ satisfies the intertwining relation
    \[
    T \circ D_\lambda = \lambda^{s-1} \, D_\lambda \circ T.
    \]

    \item[(v)] \textbf{Sobolev regularization on the cylinder.} 
    Let $H^r(\mathcal{C})$ denote the fractional Sobolev space on the cylinder. Then there exists a constant $C = C(N,s) > 0$ such that
    \[
    \|M_s v\|_{H^{1-s}(\mathcal{C})} \leq C \|v\|_{L^2(\mathcal{C})},
    \quad \text{for all } v \in L^2(\mathcal{C}).
    \]
    Equivalently,
    \[
    T\big(L^2(\mathbb{R}^N; |x|^{-2s}\,dx)\big) 
    \subset \Phi_1^{-1}\big(H^{1-s}(\mathcal{C})\big).
    \]
\end{itemize}
\end{theorem}

\medskip 
\begin{proof}
     \textit{Proof of $(i)$:} Let $u \in C_c^\infty(\mathbb{R}^N \setminus \{0\})$. We compute the $L^2(\mathcal{C})$ norm of $\Phi_s u$:$$\| \Phi_s u \|_{L^2(\mathcal{C})}^2 = \int_{\mathbb{R}} \int_{\mathbb{S}^{N-1}} |e^{\frac{N-2s}{2}t} u(e^t w)|^2 \, dw \, dt.$$Applying the inverse Emden-Fowler change of variables $r = e^t$, $dt = \frac{dr}{r}$, $x = r w$, and $dx = r^{N-1} \, dr \, dw$:$$\| \Phi_s u \|_{L^2(\mathcal{C})}^2 = \int_0^\infty \int_{\mathbb{S}^{N-1}} r^{N-2s} |u(rw)|^2  \, \frac{dr}{r} \, dw = \int_{\mathbb{R}^N} \frac{|u(x)|^2}{|x|^{2s}} \, dx = \| u \|_{L^2(\mathbb{R}^N; |x|^{-2s}\, dx)}^2.$$Since $C_c^\infty(\mathbb{R}^N \setminus \{0\})$ is dense in $L^2(\mathbb{R}^N; |x|^{-2s}dx)$, $\Phi$ extends to a global isometry.
    
    By definition, $T[u](x) = |x|^{-\frac{N-2}{2}} M_s (\Phi_s u) \left(\ln|x|, \frac{x}{|x|} \right)$. The operation $|x|^{-\frac{N-2}{2}} \psi \left(\ln|x|, \frac{x}{|x|} \right)$ is precisely $\Phi_1^{-1}[\psi]$. Thus, $T = \Phi_1^{-1} M_s \Phi_s$. 
    
\medskip 

    \textit{Proof of $(ii)$:} We analyze $P_s(\xi, \ell)$ for large arguments. By definition $$P_{s} (\xi, \ell) = 2^{2s} \left| \frac{\Gamma \left( \frac{N+2s +2\ell+ 2i \xi}{4} \right)}{\Gamma \left( \frac{N-2s+2 \ell+2i \xi}{4} \right)} \right|^{2} - C_{N,s}.$$Let $z = \frac{2\ell + 2i\xi}{4} = \frac{\ell}{2} + i\frac{\xi}{2}$. As $|z| \to \infty$, we apply the complex Stirling's approximation to the ratio of Gamma functions, which states that for bounded $a, b$:$$\left| \frac{\Gamma(z+a)}{\Gamma(z+b)} \right| \sim |z|^{\text{Re}(a-b)} \quad \text{as} \quad |z| \to \infty.$$Here, $a = \frac{N+2s}{4}$ and $b = \frac{N-2s}{4}$. We have $a - b = s$. Thus
    $$\left| \frac{\Gamma(z+a)}{\Gamma(z+b)} \right|^2 \sim |z|^{2s} = \left| \frac{\ell + i\xi}{2} \right|^{2s} = 2^{-2s} (\ell^2 + \xi^2)^s.$$Therefore, for large $\xi$ or $\ell$, the dominant term of $P_s(\xi, \ell)$ is $2^{2s} \cdot 2^{-2s} (\xi^2 + \ell^2)^s = (\xi^2 + \ell^2)^s$. The denominator of $\bar{m}(\xi, \ell)^2$ is $\xi^2 + \mu_\ell = \xi^2 + \ell^2 + \ell(N-2) \sim \xi^2 + \ell^2$.
    
    Taking the square root yields the desired asymptotic:
    $$\bar{m}(\xi, \ell) = \sqrt{\frac{P_s(\xi, \ell)}{\xi^2 + \mu_\ell}} \sim \frac{(\xi^2 + \ell^2)^{s/2}}{(\xi^2 + \ell^2)^{1/2}} = (\xi^2 + \ell^2)^{\frac{s-1}{2}}.$$Because $s \in (0,1)$, we have $s - 1 < 0$. The symbol decays at infinity and is therefore bounded from above. At the origin $\xi=0, \, \ell=0$, $P_s(0,0)=0$ and $\mu_0=0$. By Taylor expanding the Gamma functions (utilizing the Trigamma function $\psi'$, defined in \eqref{Trigamma function}), the limit $\lim_{\xi \to 0} \bar{m}(\xi, 0) = K_{N,s} > 0$ exists and is finite. Thus $\bar{m} \in L^\infty$. 

    \textit{Proof of $(iii)$:} From part $(i)$ it follows, $T = \Phi_1^{-1} M_s \Phi_s$, where $\Phi_s: L^2_s \to L^2(\mathcal{C})$ and $\Phi_1^{-1}: L^2(\mathcal{C}) \to L^2(\mathbb{R}^N, |x|^{-2} dx) :=  L^2_1 $ are isometries. Thus, the operator norm satisfies $$\| T \|_{L^2_s \to L^2_1} = \| M_s \|_{L^2(\mathcal{C}) \to L^2(\mathcal{C})}.$$Since $M_s$ acts diagonally in the Fourier-spherical domain via multiplication by $\bar{m}(\xi, \ell)$, its $L^2$ operator norm is simply the $L^\infty$ norm of its symbol
    $$\| M_s \|_{L^2 \to L^2} = \sup_{(\xi, \ell) \in \mathbb{R} \times \mathbb{N}_0} \bar{m}(\xi, \ell).$$ By part $(ii)$, $\bar{m}(\xi, \ell)$ is a bounded continuous function, implying $\|T\| < \infty$.
    
    To prove injectivity, suppose $T[u] = 0$ for some $u \in L^2_s$. Since $\Phi_1^{-1}$ is an isometry, $M_s \Phi_s u = 0$. In the Fourier domain, this means
    $$\bar{m}(\xi, \ell) (\widehat{\Phi_s u})(\xi, \ell, m) = 0, \quad \text{for almost every } (\xi, \ell, m).$$ Since the symbol $\bar{m}(\xi, \ell)$ is strictly positive for all $(\xi, \ell) \in \mathbb{R} \times \mathbb{N}_0$ (it only approaches 0 as $|\xi|, \, \ell \to \infty$), we must have $(\widehat{\Phi_s u})(\xi, \ell, m) = 0$ a.e. By the injectivity of the Fourier transform and the isometry of $\Phi_s$, this forces $u = 0$ a.e. Thus, $T$ is injective.

    \medskip 
    
    \textit{Proof of $(iv)$:} Let $u \in C_c^\infty(\mathbb{R}^N \setminus \{0\})$. We first evaluate how the cylindrical projection $\Phi_s$ interacts with the Euclidean dilation
    $$(\Phi_s D_\lambda u)(t, w) = e^{\frac{N-2s}{2}t} u(\lambda e^t w).$$Let $t' = t + \ln \lambda$. We can rewrite the projection in terms of a translation operator $\tau_h f(t) = f(t+h)$ on the cylinder
    $$(\Phi_s D_\lambda u)(t, w) = \lambda^{-\frac{N-2s}{2}} e^{\frac{N-2s}{2}(t+\ln \lambda)} u(e^{t+\ln \lambda} w) = \lambda^{\frac{2s-N}{2}} (\tau_{\ln \lambda} \Phi_s u)(t, w).$$The multiplier operator $M_s$ is defined via the Fourier transform with respect to $t$. Because Fourier multipliers commute with translations, we have $M_s \tau_{\ln \lambda} = \tau_{\ln \lambda} M_s$. Applying $M_s$:$$M_s (\Phi_s D_\lambda u) = \lambda^{\frac{2s-N}{2}} \tau_{\ln \lambda} (M_s \Phi_s u).$$Now we apply the inverse map $\Phi_1^{-1}$ to construct $T[D_\lambda u](x)$:
    \begin{align*}
       T[D_\lambda u](x) & = |x|^{-\frac{N-2}{2}} \left( M_s \Phi_s D_\lambda u \right) \left(\ln|x|, \frac{x}{|x|} \right)  \\ & = |x|^{-\frac{N-2}{2}} \lambda^{\frac{2s-N}{2}} \left( \tau_{\ln \lambda} M_s \Phi_s u \right) \left(\ln|x|, \frac{x}{|x|} \right) \\ & = |x|^{-\frac{N-2}{2}} \lambda^{\frac{2s-N}{2}} (M_s \Phi_s u)\left(\ln|x| + \ln \lambda, \frac{x}{|x|} \right).
    \end{align*}
    Notice that $|\lambda x|^{-\frac{N-2}{2}} = \lambda^{-\frac{N-2}{2}} |x|^{-\frac{N-2}{2}}$. Factoring this out, we obtain$$T[D_\lambda u](x) = \frac{\lambda^{\frac{2s-N}{2}}}{\lambda^{-\frac{N-2}{2}}} |\lambda x|^{-\frac{N-2}{2}} (M_s \Phi_s u)\left(\ln|\lambda x|, \frac{\lambda x}{|\lambda x|} \right).$$The fractional powers of $\lambda$ evaluate to $\frac{2s-N}{2} - \frac{2-N}{2} = \frac{2s-2}{2} = s-1$. Thus $$T[D_\lambda u](x) = \lambda^{s-1} T[u](\lambda x).$$

\medskip 

    \textit{Proof of $(v)$:} From part $(i)$, $\Phi_1 T[u] = M_s \Phi_s u$. We evaluate the $H^{1-s}(\mathcal{C})$ Sobolev norm of $M_s \Phi_s u$. In the spectral domain, the squared Sobolev norm is given by:
    \begin{align*}
       \| M_s \Phi_s u \|_{H^{1-s}(\mathcal{C})}^2 & = \sum_{\ell=0}^\infty \sum_{m=1}^{d_\ell} \int_{\mathbb{R}} (1 + \xi^2 + \mu_\ell)^{1-s} |(\widehat{M_s \Phi_s u})(\xi, \ell, m)|^2 d\xi \\ & = \sum_{\ell=0}^\infty \sum_{m=1}^{d_\ell} \int_{\mathbb{R}} (1 + \xi^2 + \mu_\ell)^{1-s} \bar{m}(\xi, \ell)^2 |(\widehat{\Phi_s u})(\xi, \ell, m)|^2 d\xi.
    \end{align*}
    From part $(ii)$, we established that $\bar{m}(\xi, \ell)^2 \sim (\xi^2 + \ell^2)^{s-1}$. Moreover, the eigenvalues of the sphere are $\mu_\ell = \ell(\ell+N-2) \sim \ell^2$. Thus, the aggregate weight in the integrand behaves asymptotically as $$(1 + \xi^2 + \mu_\ell)^{1-s} \bar{m}(\xi, \ell)^2 \sim (\xi^2 + \ell^2)^{1-s} (\xi^2 + \ell^2)^{s-1} = 1.$$ Because the combined symbol $(1 + \xi^2 + \mu_\ell)^{1-s} \bar{m}(\xi, \ell)^2$ is uniformly bounded over $\mathbb{R} \times \mathbb{N}_0$, there exists a constant $C > 0$ such that $$\| M_s \Phi_s u \|_{H^{1-s}(\mathcal{C})}^2 \leq C \sum_{\ell=0}^\infty \sum_{m=1}^{d_\ell} \int_{\mathbb{R}} |(\widehat{\Phi_s u})(\xi, \ell, m)|^2 d\xi = C \| \Phi_s u \|_{L^2(\mathcal{C})}^2.$$This proves that $M_s$ acts as a bounded operator from $L^2(\mathcal{C})$ to $H^{1-s}(\mathcal{C})$. Translating this to Euclidean space, $T$ strictly smooths the initial state $u \in L^2_s$, augmenting its local regularity profile by $1-s$ fractional derivatives in the logarithmic radial and angular coordinates.
\end{proof}
\medskip 

Now we are in a situation to prove Theorem~\ref{Theorem: Equivalence bw local and nonlocal via Emden Trans}.

\begin{proof}[\textbf{Proof of Theorem~\ref{Theorem: Equivalence bw local and nonlocal via Emden Trans}}]
We first begin the proof of $(i).$ Let $v=T[u]$. Applying the Emden-Fowler change of variables: $x=e^{t} w$ with $r=|x|= e^{t}$, $w \in \mathbb{S}^{N-1}$ and the measure $dx = e^{Nt} \, dt \, dw$. We define $\phi_{v}(t, w) = e^{(N-2)\frac{t}{2}} v(e^{t} w)$. We obtain 
\begin{equation*}
    \nabla v(x) = e^{-\frac{N}{2} t} \left[ \left( \partial_{t} \phi_{v} - \frac{N-2}{2} \phi_{v}  \right)  w + \nabla_{w} \phi_{v} \right].
\end{equation*}
Substituting this into the definition of $\delta_{2}(v)$, we get
\begin{equation*}
    \delta_{2}(v) = \int_{\mathbb{R}} \int_{\mathbb{S}^{N-1}} \left[ \left( \partial_{t} \phi_{v} - \frac{N-2}{2} \phi_{v}  \right)^{2} + |\nabla_{w} \phi_{v}|^{2} - \left( \frac{N-2}{2} \right)^{2} \phi^{2}_{v} \right]  \, dw \, dt.
\end{equation*}
Since the cross term $-(N-2)\phi_{v} \partial_{t} \phi_{v}$ integrates to $0$ over $\mathbb{R}$ as $\phi_{v}$ vanishes at infinity. Indeed,  $$I_{cross} = -(N-2) \int_{\mathbb{S}^{N-1}} \int_{\mathbb{R}} \phi_{v} \partial_{t} \phi_{v} \, dt \, dw = -\frac{N-2}{2} \int_{\mathbb{S}^{N-1}} \int_{\mathbb{R}} \partial_{t} (\phi_{v}^{2}) \, dt \, dw.$$We know $\phi_{v} \in H^{1}(\mathbb{R} \times \mathbb{S}^{N-1})$. By the trace theorem and Sobolev embedding on the cylinder, it follows $\phi_{v}$ is absolutely continuous with respect to $t$ and satisfies $$\lim_{t \to \pm \infty} \| \phi_{v}(t, \cdot) \|_{L^{2}(\mathbb{S}^{N-1})} = 0.$$Applying the Fundamental Theorem of Calculus along the $t$-variable, we obtain $$\int_{\mathbb{R}} \partial_{t} (\phi_{v}^{2}(t,w)) \, dt = \lim_{R \to \infty} \left( \phi_{v}^{2}(R, w) - \phi_{v}^{2}(-R, w) \right).$$ Integrating over the sphere $\mathbb{S}^{N-1}$, we arrive at$$\int_{\mathbb{S}^{N-1}} \left( \lim_{R \to \infty} \phi_{v}^{2}(R, w) - \lim_{R \to \infty} \phi_{v}^{2}(-R, w) \right) \, dw = \lim_{R \to \infty} \| \phi_{v}(R, \cdot) \|_{L^{2}}^{2} - \lim_{R \to \infty} \| \phi_{v}(-R, \cdot) \|_{L^{2}}^{2} = 0.$$
Therefore the above relation reduces to
    \begin{equation*}
    \delta_{2}(v) = \int_{\mathbb{R}} \int_{\mathbb{S}^{N-1}}  \left( |\partial_{t} \phi_{v} |^{2} + |\nabla_{w} \phi_{v}|^{2}  \right) \, dw \, dt.
\end{equation*}
Now, applying the Plancherel's theorem on $\mathbb{R}$ and the spectral theorem for $-\Delta_{\mathbb{S}^{N-1}}$, we obtain
\begin{equation}\label{eqn1}
    \delta_{2}(v) = \sum_{\ell=0}^{\infty} \sum_{m=1}^{d_{\ell}} 
\int_{\mathbb{R}} \left( \xi^{2} + \mu_{\ell} \right) |\widehat{\phi}_{v}(\xi, \ell, m)|^{2} \, d \xi.
\end{equation}

We define the function $\phi_{u}(t,w) := e^{\frac{N-2s}{2}t} u(e^{t}w)$. The celebrated result of Frank, Lieb and Seiringer \cite{FrankLieb2008} established that the fractional Hardy quadratic form can be explicitly diagonalized over the cylinder $\mathcal{C}$. The quadratic form evaluates to the non-local operator $H_{s}$ on $\mathcal{C}$ minus the ground state eigenvalues, i.e.,
\begin{equation*}
    \delta_{s,2}(u) = < \phi_{u}, H_{s} \phi_{u} >_{L^{2}(\mathcal{C})} - C_{N,s} \| \phi_{u} \|^{2}_{L^{2}(\mathcal{C})}.
\end{equation*}
By passing to the Fourier-spherical basis, the Mellin transform yields the exact multiplier
\begin{equation*}
    \delta_{s,2} (u) = \sum_{\ell = 0}^{\infty} \sum_{m=1}^{d_{\ell}} \int_{\mathbb{R}} P_{s} (\xi, \ell) |\widehat{\phi}_{u}(\xi, \ell,m)|^{2} \, d \xi,
\end{equation*}
where $P_{s}$ is defined as in \eqref{PS}. 

By the definition of the operator $T$, the cylindrical projection of $v$ is exactly $\phi_{v}= M_{s} \phi_{u}$. In the frequency space, we have
\begin{equation*}
    \widehat{\phi}_{v}(\xi, \ell, m) =  \sqrt{\frac{P_{s}(\xi, \ell)}{ \xi^{2}+ \mu_{\ell}}} \widehat{\phi}_{u}(\xi, \ell, m).
\end{equation*}
Substituting this into the equation \eqref{eqn1}, we obtain
\begin{align*}
    \delta_2(v) &= \sum_{\ell =0}^{\infty} \sum_{m=1}^{d_{\ell}} \int_{\mathbb{R}} (\xi^{2}+ \mu_{\ell}) \left| \sqrt{\frac{P_{s}(\xi, \ell)}{ \xi^{2}+ \mu_{\ell}}} \widehat{\phi}_{u}(\xi, \ell, m) \right|^{2} \, d \xi \\ & = \sum_{\ell = 0}^{\infty} \sum_{m=1}^{d_{\ell}} \int_{\mathbb{R}} P_{s} (\xi, \ell) |\widehat{\phi}_{u}(\xi, \ell,m)|^{2} \, d \xi \\ & = \delta_{s,2}(u) . 
\end{align*}
This proves part $(i)$.

\textit{Proof of $(ii)$:} Let $u(x) = \omega_{s}(x) = |x|^{-(N-2s)/2}$. The cylindrical projection of $u$ is$$\phi_{u}(t,w) = e^{\frac{N-2s}{2} t} e^{- \frac{N-2s}{2}t} = 1.$$The constant function $1$ is independent of $t$ and $w$. In the Fourier spherical domain, it is perfectly localized at the zero mode$$\widehat{\phi}_{u} (\xi, \ell,m) = C \delta(\xi) \delta_{\ell, 0},$$where $\delta(\xi)$ is the Dirac distribution.

Applying the operator $M_{s}$ to $\phi_{u} =1$, we have$$\widehat{\phi}_{v}(\xi, \ell, m) = \sqrt{ \frac{P_{s}(\xi, \ell)}{\xi^{2}+ \mu_{\ell}}} C \delta(\xi) \delta_{\ell, 0}.$$ Because of the support of the Dirac Delta, we only need to evaluate the multiplier limit as $\xi \to 0$ for $\ell=0$ (noting that $\mu_{0}=0$):$$\lim_{\xi \to 0} \bar{m}(\xi, 0) = \lim_{\xi \to 0} \sqrt{ \frac{P_{s}(\xi, 0)}{\xi^{2}}}.$$By definition of $P_{s}$, we have $P_{s}(0,0)=0$. Let$$f(z)=\left| \Gamma (a+iz) \right|^{2} = \Gamma (a+iz) \Gamma(a-iz),$$where $a \in \mathbb{R}$. Then, differentiating with respect to $z$, we obtain $$f'(z) = i \Gamma' (a+iz) \Gamma (a-iz) - i \Gamma (a+iz) \Gamma' (a-iz).$$Therefore, $f'(0)=0$. Differentiating $f'$ again and evaluating at $z =0$, we obtain$$f''(0)= -\Gamma''(a) \Gamma(a) + \Gamma'(a)^{2} + \Gamma'(a)^{2} - \Gamma(a) \Gamma''(a) = 2(\Gamma'(a))^{2} - 2\Gamma(a)\Gamma''(a).$$Recall the definition of the trigamma function
\begin{equation}\label{Trigamma function}
   \psi'(a) = \frac{d}{da} \left( \frac{\Gamma'(a)}{\Gamma(a)} \right) = \frac{\Gamma''(a)\Gamma(a) - (\Gamma'(a))^{2}}{\Gamma(a)^{2}}. 
\end{equation}
Thus, we can concisely write:$$f''(0) = -2 \Gamma(a)^{2} \psi'(a).$$Let $a_{1} = (N+2s)/4$ and $a_{2} = (N-2s)/4$. To find the Taylor series expansion of $P_{s}(\xi,0)$ around $\xi=0$, we evaluate the quotient $h(z) = \frac{f_1(z)}{f_2(z)}$ where $f_1, f_2$ correspond to parameters $a_1, a_2$. Using the chain rule for $z = \xi/2$, the second derivative evaluated at $0$ yields $$P_{s}(\xi,0) \approx 2^{2s} \frac{1}{2} \left[ \frac{f_1''(0)f_2(0) - f_1(0)f_2''(0)}{f_2(0)^2} \right] \left( \frac{\xi}{2} \right)^2.$$Substituting $f''(0) = -2\Gamma(a)^2\psi'(a)$ and simplifying with the sharp constant $C_{N,s} = 2^{2s} \frac{\Gamma(a_1)^2}{\Gamma(a_2)^2}$:$$P_{s}(\xi,0) \approx \frac{\xi^{2}}{4} C_{N,s} \left( \psi'(a_{2}) - \psi'(a_{1}) \right).$$Since $\psi'$ is strictly monotonically decreasing on $(0, \infty)$ and $a_2 < a_1$, the term $\psi'(a_2) - \psi'(a_1)$ is strictly positive.

Taking the limit, we obtain$$\lim_{\xi \to 0} \frac{P_{s}(\xi,0)}{\xi^{2}} = \frac{1}{4} C_{N,s} \left( \psi'\left(\frac{N-2s}{4} \right) - \psi'\left(\frac{N+2s}{4}\right) \right) =  K^{2}_{N,s}.$$Thus, the multiplier evaluates exactly to $K_{N,s}$ at the origin.

Since the multiplier acts on the constant strictly as scalar multiplication by $K_{N,s}$, we have$$\phi_{v}(t, w) = K_{N,s} \mathbf{1}= K_{N,s}.$$Mapping this back to Euclidean space via the definition of $T$, we obtain$$v(x) = |x|^{-(N-2)/2} \phi_{v} \left( \ln|x|, \frac{x}{|x|} \right) = K_{N,s} |x|^{-(N-2)/2}.$$Therefore, $T[\omega_{s}] = K_{N,s} \omega_{1}$.
\end{proof}


\medskip 

Now we prove Theorem \ref{qspf}.  

\begin{proof}[\textbf{Proof of Theorem \ref{qspf}}]
    Let $v = T[u]$. By the spectral equivalence over the Emden-Fowler cylinder, we have $\delta_{s,2}(u) = \delta_2(v)$, where $\delta_2(v)$ is the local Hardy deficit. Moreover,  Theorem \ref{Theorem: Improvement of stability exponent C and F result} applied to the case $p=2$, the local Hardy inequality with remainder term yields
\begin{equation*}
 \delta_{2}(v)  \geq C (d_2(v))^{4} \int_{\mathbb{R}^{N}} \frac{|v(x)|^2}{|x|^2} \, dx  .
\end{equation*}
Substituting $v = T[u]$ back into the right-hand side. By the definition of our pullback weight space, we have
$$\int_{\mathbb{R}^{N}} \frac{|T[u]|^2}{|x|^2} dx = \|u\|_{\mathcal{W}_{N,s}}^2.$$ By the linearity of $T$ and the extremizer mapping $T \left[\frac{a}{K_{N,s}} w_s \right] = v_a$, the local distance $d_2(T[u])$ is exactly equivalent to the fractional distance $\mathfrak{D}_{s,2}(u)$. Arranging the terms yields $\delta_{s,2}(u) = \delta_{2}(T[u]) \ge C \|u\|_{\mathcal{W}_{N,s}}^2 \left( \mathfrak{D}_{s,2}(u)\right)^{4}$. This proves the theorem.
\end{proof} 

    

\medskip

\section{Application: Fractional Hardy-Heisenberg Uncertainty Principle.}\label{Section : 7}

As an application of the transformation $T$, defined in \eqref{Defn: The operator T}, we derive a fractional Hardy--Heisenberg uncertainty principle. 

The Emden--Fowler change of variables \(t=\ln|x|\) converts radial scaling  in \(\mathbb{R}^N\) into translations on the cylinder \(\mathcal{C}=\mathbb{R}\times\mathbb{S}^{N-1}\). In this framework, the classical Hardy deficit naturally plays the role of a momentum variance. Since the transformation \(T\) preserves the Hardy deficit exactly, this Heisenberg-type structure can be transferred from the classical setting to the fractional one.

We now state the result. Let \(N\geq 3\) and \(s\in(0,1)\). We retain the notation for the cylinder \(\mathcal{C}\), the classical deficit \(\delta_2\), the fractional deficit \(\delta_{s,2}\), and the deficit-preserving transformation \(T\) introduced earlier. To this end, we define two new weighted nonlocal capacities for a function 
$u \in \mathcal{S}(\mathbb{R}^{N})$:

{\bf The transformed Hardy mass}: $$M_{T}(u) := \int_{\mathbb{R}^{N}} \frac{|T[u](x)|^{2}}{|x|^{2}} \, dx.$$

{\bf The transformed logarithmic variance}:$$V_{T}(u) := \int_{\mathbb{R}^{N}} (\ln|x|)^{2} \frac{|T[u](x)|^{2}}{|x|^{2}} \, dx.$$ 

\begin{rem}
The right hand side of the above two quantities  is finite. Indeed, Under the Emden-Fowler change of variable, $t=\ln|x|$ and $\displaystyle \theta=\frac{x}{|x|}$, the integral reduces to $$\displaystyle \int_{\mathbb{R}^N}\frac{|T[u](x)|^2}{|x|^2}dx=\|M_s\bar{w}\|_{L^2(\mathcal{C})}^2,$$ where $\mathcal{C}=\mathbb{R}\times\mathbb{S}^{N-1}$ and $\bar{w}(t,\theta)=e^{t(\frac{N-2s}{2})}u(e^t\theta)$.
Given $u\in\mathcal{S}(\mathbb{R}^N)$ and $N>2s$, the prefactor ensures $\bar{w}(t,\theta)$ exhibits exponential decay as $t\to\pm\infty$, guaranteeing $\bar{w}\in\mathcal{S}(\mathcal{C})$.
Stirling's asymptotic expansion dictates that the spectral multiplier admits a strict polynomial bound $\displaystyle \frac{P_s(\xi,\ell)}{\xi^2+\mu_\ell}\le C(1+\xi^2+\ell^2)^K$, for some $K>0$.
By Plancherel theorem on $\mathcal{C}$, the integral evaluates to $\displaystyle \sum_{\ell=0}^\infty\sum_{m=1}^{d_\ell}\int_{\mathbb{R}}\frac{P_s(\xi,\ell)}{\xi^2+\mu_\ell}|\widehat{w}(\xi,\ell,m)|^2d\xi<\infty$, which converges absolutely due to the rapid decay of $\widehat{w}$. A similar argument shows that the transformed logarithmic variance is finite as well.
\end{rem}

\medskip 

To build the fractional result, we first establish a sharp Heisenberg-type uncertainty principle for the classical Hardy deficit on Euclidean space.

\begin{lemma}[Classical Hardy--Heisenberg Uncertainty Principle]
\label{preliminary lemma}
Let \(N \geq 4\). Then, for every \(v \in \mathcal{S}(\mathbb{R}^N)\), the following inequality holds
\begin{equation}
   \delta_{2}(v)
\left(
\int_{\mathbb{R}^{N}}
(\ln |x|)^{2}
\frac{|v(x)|^{2}}{|x|^{2}}
\, dx
\right)
\geq
\frac{1}{4}
\left(
\int_{\mathbb{R}^{N}}
\frac{|v(x)|^{2}}{|x|^{2}}
\, dx
\right)^{2}. 
\end{equation}
\end{lemma}
\begin{proof}
First let us assume $v \in C_c^{\infty}(\mathbb{R}^N \setminus \{0\}).$
We apply the Emden-Fowler change of variables $x=e^{t} w$ with $r=|x|= e^{t}$, $w \in \mathbb{S}^{N-1}$, and it follows $dx = e^{Nt} \, dt \, dw$. As shown in our fundamental identity, defining $\phi_{v}(t, w) = e^{(N-2)\frac{t}{2}} v(e^{t} w)$ yields $$\delta_{2}(v) = \int_{\mathbb{S}^{N-1}} \int_{\mathbb{R}}  \left( |\partial_{t} \phi_{v} |^{2} + |\nabla_{w} \phi_{v}|^{2}  \right) \, dt \, dw.$$Since $|\nabla_{w} \phi_{v}|^{2} \geq 0$, we can trivially bound the deficit from below by its purely radial derivative component:$$\delta_{2}(v) \geq \int_{\mathbb{S}^{N-1}} \int_{\mathbb{R}} |\partial_{t} \phi_{v}|^{2} \, dt \, dw.$$

Since \(v \in C^{\infty}_{c}(\mathbb{R}^{N}\setminus\{0\})\), the associated cylindrical function \(\phi_v(t,w)\) is compactly supported with respect to the variable \(t\in\mathbb{R}\). For any fixed $w \in \mathbb{S}^{N-1}$, we apply integration by parts to the $L^{2}(\mathbb{R})$ norm of $\phi_{v}$: $$\int_{\mathbb{R}} |\phi_{v}|^{2} \, dt = \int_{\mathbb{R}} |\phi_{v}|^{2} \partial_{t}(t) \, dt = \left[ t \, |\phi_{v}|^{2} \right]_{-\infty}^{\infty} - 2 \int_{\mathbb{R}} t \phi_{v} \partial_{t} \phi_{v} \, dt.$$

The second term yields:
$$
\int_{\mathbb{R}} |\phi_{v}|^{2} \, dt = -2 \int_{\mathbb{R}} t \phi_{v} \partial_{t} \phi_{v} \, dt.
$$

Taking the absolute value and applying the Cauchy-Schwarz inequality over $\mathbb{R}$, we obtain $$\int_{\mathbb{R}} |\phi_{v}|^{2} \, dt \leq 2 \left( \int_{\mathbb{R}} t^{2} |\phi_{v}|^{2} \, dt \right)^{\frac{1}{2}} \left( \int_{\mathbb{R}} |\partial_{t} \phi_{v}|^{2} \, dt \right)^{\frac{1}{2}}.$$We now integrate this inequality over the angular variables $w \in \mathbb{S}^{N-1}$:$$\int_{\mathbb{S}^{N-1}} \left( \int_{\mathbb{R}} |\phi_{v}|^{2} \, dt \right) \, dw \leq 2 \int_{\mathbb{S}^{N-1}} \left( \int_{\mathbb{R}} t^{2} |\phi_{v}|^{2} \, dt \right)^{\frac{1}{2}} \left( \int_{\mathbb{R}} |\partial_{t} \phi_{v}|^{2} \, dt \right)^{\frac{1}{2}} \, dw.$$ Applying the Cauchy-Schwarz inequality a second time, now over the space $L^{2}(\mathbb{S}^{N-1})$: $$\int_{\mathcal{C}} |\phi_{v}|^{2} \, dt \, dw \leq 2 \left( \int_{\mathcal{C}} t^{2} |\phi_{v}|^{2} \, dt \, dw \right)^{\frac{1}{2}} \left( \int_{\mathcal{C}} |\partial_{t} \phi_{v}|^{2} \, dt \, dw \right)^{\frac{1}{2}}.$$Squaring both sides and substituting the lower bound for $\delta_2(v)$ yields the uncertainty principle on the cylinder:$$\frac{1}{4} \left( \int_{\mathcal{C}} |\phi_{v}|^{2} \, dt \, dw \right)^{2} \leq \left( \int_{\mathcal{C}} t^{2} |\phi_{v}|^{2} \, dt \, dw \right)\delta_2(v).$$To conclude the lemma, we translate the integral operators back to Euclidean space. Using $t = \ln|x|$ and reversing the measure $dt \, dw = |x|^{-N} \, dx$:$$\int_{\mathcal{C}} |\phi_{v}|^{2} \, dt \, dw = \int_{\mathbb{R}^{N}} |x|^{N-2} |v(x)|^{2} |x|^{-N} \, dx = \int_{\mathbb{R}^{N}} \frac{|v(x)|^{2}}{|x|^{2}} \, dx.$$ Similarly, introducing the $t^{2}$ weight:$$\int_{\mathcal{C}} t^{2} |\phi_{v}|^{2} \, dt \, dw = \int_{\mathbb{R}^{N}} (\ln|x|)^{2} |x|^{N-2} |v(x)|^{2} |x|^{-N} \, dx = \int_{\mathbb{R}^{N}} (\ln|x|)^{2} \frac{|v(x)|^{2}}{|x|^{2}} \, dx.$$Substituting  back into the squared cylinder inequality completes the proof of Lemma for smooth compactly supported function. Now We shall extend the proof for functions in 
$\mathcal{S}(\mathbb{R}^N).$ We first observe  finiteness of the weighted integrals for $v \in \mathcal{S}(\mathbb{R}^N).$

Let $v \in \mathcal{S}(\mathbb{R}^N)$. Then all the derivatives decay faster than any polynomial power at infinity, in particular $v, \nabla v \in L^2(\mathbb{R}^N).$ Near the origin $v$ is smooth, hence it is bounded on $B_1.$ Therefore,
\begin{equation}\label{hardy}
    \int_{|x|<1}\frac{|v(x)|^2}{|x|^2}\, \mathrm{d}x\leq \|v\|^2_{L^{\infty}({B_1})}\int_{|x|<1} |x|^{-2}\, \mathrm{d}x.
\end{equation}
    The RHS of \eqref{hardy} is finite if and only if $N\geq 3$.

    Similarly,
    \begin{equation}\label{log-hardy}
        \int_{|x|<1}(\ln |x|)^2\frac{|v(x)|^2}{|x|^2}\, \mathrm{d}x\leq \|v\|^2_{L^{\infty}(B_1)}|\mathbb{S}^{N-1}|\int_0^1 (\ln r)^2r^{N-3}\, \mathrm{d}r.
    \end{equation}
    The RHS of \eqref{log-hardy} is finite if and only if $N \geq 4$. At infinity we use the decay estimates for the Schwartz functions to conclude: 
    \begin{equation*}
        \int_{|x|>1}\frac{|v(x)|^2}{|x|^2}\,\mathrm{d}x<\infty,\quad \int_{|x|>1}(\ln |x|)^2\frac{|v(x)|^2}{|x|^2}\,\mathrm{d}x<\infty.
    \end{equation*}
    Hence both weighted moments are finite, and therefore $\delta_2(v)$ is well defined and finite. Rest of the arguments follows from the standard approximation arguments. This completes the proof.
\end{proof}

We now state the new fractional uncertainty principle. Because the fractional Hardy deficit strictly lacks a classical chain rule or an easy integration-by-parts formula, proving this directly in the fractional Euclidean space would be difficult, and requires our multiplier transformation $T$.

\begin{theorem}[Fractional Hardy--Heisenberg Uncertainty Principle]
\label{thm:fractional-hardy-heisenberg}
Let \(N \geq 4\), \(0<s<1\), and let 
\(u \in \mathcal{S}(\mathbb{R}^{N})\). Then
\begin{equation}\label{eq:fractional-hardy-heisenberg}
    \delta_{s,2}(u)\, V_{T}(u)
    \geq
    \frac{1}{4}\bigl(M_{T}(u)\bigr)^{2}.
\end{equation}
In particular, the fractional Hardy deficit controls the logarithmic dispersion of the nonlocal transform \(T[u]\).
\end{theorem}

\begin{proof}
Let $u \in \mathcal{S}(\mathbb{R}^{N}).$ Since $T$ is a well-defined transformation mapping $\mathcal{S}(\mathbb{R}^{N})$ to $L^{2}_{\text{loc}} (\mathbb{R}^{N} \setminus \{ 0 \})$, we define $v(x) := T[u](x)$. 

Applying Lemma~\ref{preliminary lemma} directly to the transformed function $v = T[u]$, we obtain
$$\delta_{2}(T[u]) \; \left( \int_{\mathbb{R}^{N}} (\ln|x|)^{2} \frac{|T[u](x)|^{2}}{|x|^{2}} \, dx \right) \geq \frac{1}{4} \left( \int_{\mathbb{R}^{N}} \frac{|T[u](x)|^{2}}{|x|^{2}} \, dx \right)^{2}.$$ \par By the definitions of the transformed capacities $V_{T}(u)$ and $M_{T}(u)$, this simplifies exactly to:$$\delta_{2}(T[u]) V_{T}(u) \geq \frac{1}{4} \left( M_{T}(u) \right)^{2}.$$Finally, we invoke our previously proven deficit-preserving identity, which states that $\delta_2(T[u]) = \delta_{s,2}(u)$. Substituting the fractional deficit into the left-hand side yields:$$\delta_{s,2}(u) V_{T}(u) \geq \frac{1}{4} \left( M_{T}(u) \right)^{2}.$$
This completes the proof.
\end{proof}


Now we turn to address the optimality of the constant $\frac{1}{4}$ and the existence of extremals. We state the following theorem: 

\begin{theorem}[Sharpness and Extremizers of the Fractional Hardy-Heisenberg Inequality]\label{thmsharp}

Let \(N \geq 4\), \(0<s<1\), and let 
\(u \in \mathcal{S}(\mathbb{R}^{N})\). Let the  constant in the Fractional Hardy--Heisenberg Uncertainty Principle be denoted by
\[
K := \frac{1}{4}.
\] Then
\begin{itemize}
    \item[(i)]  The constant $K = \frac{1}{4}$ is optimal, and is attained for the Gaussian profiles,

$$
\phi(t,w)\;=\; A e^{-\alpha t^2} \ \mbox{for} \ A \in \mathbb{R} \setminus \{0\} \ \mbox{and} \ \alpha > 0.
$$
    
\item[(ii)] For every $u \in C^{\infty}_{c}(\mathbb{R}^{N})$, the inequality is strict:$$\delta_{s,2}(u) \cdot V_{T}(u) > \frac{1}{4} \left( M_{T}(u) \right)^{2}.$$
    
\end{itemize}
\end{theorem}

\medskip

\begin{proof}
We proceed by analyzing the exact conditions required for equality in the two inequalities utilized in the proof of Lemma~\ref{preliminary lemma}. Let $v = T[u]$ and let its cylindrical representation be $\phi_{v}(t,w) = e^{\frac{N-2}{2}t} v(e^{t}w)$.

\medskip 

\textbf{Step 1:} (The Angular Gradient Vanishing Condition)

Exploiting Lemma~\ref{preliminary lemma},we have 

\begin{equation} \label{ineq_angular}\delta_{2}(v) = \int_{\mathbb{S}^{N-1}} \int_{\mathbb{R}} \left( |\partial_{t} \phi_{v}|^{2} + |\nabla_{w} \phi_{v}|^{2} \right)  \, dt  \, dw \geq \int_{\mathbb{S}^{N-1}} \int_{\mathbb{R}} |\partial_{t} \phi_{v}|^{2} \,  dt \,  dw.
\end{equation}

Equality in \eqref{ineq_angular} holds if and only if the angular gradient vanishes almost everywhere \begin{equation*}|
\nabla_{w} \phi_{v}(t,w)|^{2} = 0 \quad \emph{i.e.,} \quad \phi_{v}(t,w) \equiv \phi_{v}(t).
\end{equation*}
This implies that to achieve equality, the transformed function $v(x)$ must be purely radial (spherically symmetric) in $\mathbb{R}^{N}.$

\medskip 

\textbf{Step 2:} (The Cauchy-Schwarz Equality Condition)

The second inequality in the proof of Lemma \ref{preliminary lemma} relied on the Cauchy-Schwarz inequality for $L^{2}(\mathbb{R})$: 

\begin{equation} \label{ineq_CS}\left( \int_{\mathbb{R}} t \phi_{v} \partial_{t} \phi_{v} \,  dt \right)^{2} \leq \left( \int_{\mathbb{R}} t^{2} |\phi_{v}|^{2} \,  dt \right) \left( \int_{\mathbb{R}} |\partial_{t} \phi_{v}|^{2}  \, dt \right).
\end{equation}

Therefore, equality holds in the above inequality if and only if the two functions in the integrand are linearly dependent. Therefore, 

\begin{equation*}\partial_{t} \phi_{v}(t) = c  t \phi_{v}(t), \quad \mbox{for some} \ c\in \mathbb{\mathbb{R}}.
\end{equation*}

Hence, upon integrating  with respect to $t$ variable, we obtain

\begin{equation*}\frac{\partial_{t} \phi_{v}(t)}{\phi_{v}(t)} = c t \quad \implies \ln|\phi_{v}(t)| = \frac{c}{2} t^{2} + C \quad \implies \phi_{v}(t) = A e^{\frac{c}{2}t^{2}},
\end{equation*}
where $A \neq 0$ (for non-zero $\phi_v$) is an arbitrary constant.

\medskip 

\textbf{Step 3:} (The Integrability Requirement)

Moreover, for the quantities $M_{T}(u)$ and $V_{T}(u)$ to be well-defined and finite, the cylindrical function $\phi_{v}(t)$ must belongs to $L^2(\mathbb{R}),$ i.e., 

\begin{equation*}
\int_{\mathbb{R}} |\phi_{v}(t)|^{2}  dt = \int_{\mathbb{R}} A^{2} e^{c t^{2}}  dt < \infty.
\end{equation*}

This integral converges if and only if $c < 0$.  Thus, the extremizing profiles on the cylinder must take the form of Gaussian

\begin{equation} \label{Extremizer_Gaussian}
\phi_{*}(t) = A e^{-\alpha t^{2}} \quad \text{for } \alpha > 0, \quad A \neq 0.
\end{equation}

\medskip 

{\bf Proof of $(i)$:} We shall prove that the constant $\frac{1}{4}$ is sharp. By evaluating the capacities directly using the Gaussian extremizer $\phi_{*}(t) = e^{-\alpha t^{2}}$ (taking $A=1$ without loss of generality) over the measure space of the cylinder, incorporating the spherical volume $|\mathbb{S}^{N-1}|$. The integral takes the value 

\begin{equation*}
M_{T} = |\mathbb{S}^{N-1}| \int_{\mathbb{R}} e^{-2\alpha t^{2}}  dt = |\mathbb{S}^{N-1}| \sqrt{\frac{\pi}{2\alpha}}.
\end{equation*}

The logarithmic variance integral takes the value 
\begin{equation*}V_{T} = |\mathbb{S}^{N-1}| \int_{\mathbb{R}} t^{2} e^{-2\alpha t^{2}} , dt = |\mathbb{S}^{N-1}| \frac{1}{4\alpha} \sqrt{\frac{\pi}{2\alpha}}.
\end{equation*}

Moreover, translating the extremizing profile $\phi_{*}(t)$ back to the Euclidean space. Using the cylindrical mapping

\begin{equation*}v_{*}(x) = |x|^{-\frac{N-2}{2}} \phi_{*}(\ln|x|) = A |x|^{-\frac{N-2}{2}} e^{-\alpha (\ln|x|)^{2}}.
\end{equation*}
Since
$e^{-\alpha (\ln|x|)^{2}} > 0$ for all $x \in \mathbb{R}^{N} \setminus \{ 0 \}$, the function $v_{*}(x)$ has support in $\mathbb{R}^{N} \setminus \{ 0 \}.$ The integral involving derivative (which equals $\delta_2(v_{*})$ because the angular gradient is zero), we conclude
\begin{align*}\partial_{t} \phi_{*}(t) = -2\alpha t e^{-\alpha t^{2}}.
\end{align*}
Hence it implies 
\begin{align*}
 \delta_2(v_{*})& = |\mathbb{S}^{N-1}| \int_{\mathbb{R}} 4\alpha^{2} t^{2} e^{-2\alpha t^{2}}  dt \\
&= |\mathbb{S}^{N-1}| 4\alpha^{2} \left( \frac{1}{4\alpha} \sqrt{\frac{\pi}{2\alpha}} \right) = |\mathbb{S}^{N-1}| \alpha \sqrt{\frac{\pi}{2\alpha}}.
\end{align*}

Now, we evaluate the quotient  of the Uncertainty Principle exactly using the above values of the integral

\begin{equation*}\frac{\delta_2(v_{*})V_{T}}{\left( M_{T} \right)^{2}} = \frac{\left( |\mathbb{S}^{N-1}| \alpha \sqrt{\frac{\pi}{2\alpha}} \right) \left( |\mathbb{S}^{N-1}| \frac{1}{4\alpha} \sqrt{\frac{\pi}{2\alpha}} \right)}{\left( |\mathbb{S}^{N-1}| \sqrt{\frac{\pi}{2\alpha}} \right)^{2}}
\;=\; \frac{1}{4}. 
\end{equation*}
Finally exploiting the fact that $T$ is deficit-preserving, $\delta_{s,2}(u_{*}) = \delta_2(v_{*})$, proving the fractional quotient exactly $\frac{1}{4}.$ This proves part $(i).$

\medskip 

{\bf Proof of part $(ii)$:} We now prove Non-Existence of strict inequality in $C^{\infty}_{c}(\mathbb{R}^{N}).$ As before translating the extremizing profile $\phi_{*}(t)$ back to the Euclidean space we see that  $v_{*}(x)$ has support over the whole $\mathbb{R}^{N} \setminus \{ 0 \}$. It does not have compact support. Consequently, its inverse fractional transformation $u_{*} = T^{-1}[v_{*}]$ cannot have compact support either. Therefore, no function strictly within $C^{\infty}_{c}(\mathbb{R}^{N})$ can satisfy the equality, proving part $(ii)$ of the Theorem.
\end{proof}

\medskip 

{\bf Proof of Theorem~\ref{HHUP}} The proof follows by combining Theorem~\ref{thm:fractional-hardy-heisenberg} and Theorem~\ref{thmsharp}. 

\medskip 

\section*{Acknowledgments} 
A. Banerjee is supported by the Doctoral Fellowship of the Indian Statistical Institute, Delhi Centre. The research of D. Ganguly is partially supported by the ANRF MATRICS Research Grant (MTR/2023/000331) and ANRF-Advanced Research Grant (ANRF/ARG/2025/000572/MS). V. Sahu gratefully acknowledges the financial support from ANRF through the National Post Doctoral Fellowship (PDF/2025/004611).



\end{document}